\documentclass[11pt]{article}
\usepackage{epsfig,amsmath,amsthm,amssymb,graphics,amsfonts,psfrag}
\usepackage[letterpaper,margin=1in]{geometry}
\usepackage{setspace}
\theoremstyle{plain}
\newtheorem{thm}{Theorem}[section]

\newtheorem{prop}[thm]{Proposition}
\newtheorem{lem}[thm]{Lemma}

\theoremstyle{definition}
\newtheorem{defn}[thm]{Definition}

\def\R{\mathbb{R}}

\def\P{\mathbb{P}}
\def\E{\mathbb{E}}

\newcommand{\be}{\begin{equation}}
\newcommand{\ee}{\end{equation}}
\newcommand{\bea}{\begin{eqnarray}}
\newcommand{\eea}{\end{eqnarray}}
\newcommand{\beann}{\begin{eqnarray*}}
\newcommand{\eeann}{\end{eqnarray*}}
\newcommand{\benn}{\begin{equation*}}
\newcommand{\eenn}{\end{equation*}}

\def\ra{\rightarrow}
\def\I{\infty}
\def\I{\infty}

\newcommand{\cA}{{\mathcal A}}  
\newcommand{\cB}{{\mathcal B}}  
\newcommand{\cC}{{\mathcal C}}  
\newcommand{\cD}{{\mathcal D}}  
\newcommand{\cF}{{\mathcal F}}  
\newcommand{\cO}{{\mathcal O}}  

\begin{document}

\author{Christian Kuehn\footnotemark[1]~\footnotemark[2]}

\renewcommand{\thefootnote}{\fnsymbol{footnote}}
\footnotetext[1]{%
Max Planck Institute,
Physics of Complex Systems, 
Dresden, 01187, Germany.}
\footnotetext[2]{%
Institute for Analysis and Scientific Computing, 
Vienna University of Technology, 
Vienna, 1040, Austria.} 
\renewcommand{\thefootnote}{\arabic{footnote}}
 
\title{A Mathematical Framework for Critical Transitions:\\ Normal Forms, Variance and Applications}

\maketitle

\begin{abstract}
Critical transitions occur in a wide variety of applications including mathematical biology, climate change, human physiology and economics. Therefore it is highly desirable to find early-warning signs. We show that it is possible to classify critical transitions by using bifurcation theory and normal forms in the singular limit. Based on this elementary classification, we analyze stochastic fluctuations and calculate scaling laws of the variance of stochastic sample paths near critical transitions for fast subsystem bifurcations up to codimension two. The theory is applied to several models: the Stommel-Cessi box model for the thermohaline circulation from geoscience, an epidemic-spreading model on an adaptive network, an activator-inhibitor switch from systems biology, a predator-prey system from ecology and to the Euler buckling problem from classical mechanics. For the Stommel-Cessi model we compare different detrending techniques to calculate early-warning signs. In the epidemics model we show that link densities could be better variables for prediction than population densities. The activator-inhibitor switch demonstrates effects in three time-scale systems and points out that excitable cells and molecular units have information for subthreshold prediction. In the predator-prey model explosive population growth near a codimension two bifurcation is investigated and we show that early-warnings from normal forms can be misleading in this context. In the biomechanical model we demonstrate that early-warning signs for buckling depend crucially on the control strategy near the instability which illustrates the effect of multiplicative noise.    
\end{abstract}


\section{Introduction}  

A critical transition (or tipping point) is a rapid sudden change of a time-dependent system. For the introduction we shall rely on this intuitive notion; the mathematical development starts in Section \ref{sec:fast_slow}. Typical examples of critical transitions are drastic changes in the climate \cite{Lentonetal,Alleyetal}, in ecological systems \cite{Clarketal,Carpenteretal}, in medical conditions \cite{ElgerLehnertz,Venegasetal} or in economics \cite{HongStein,HuangWang}. Reviews of recent progress to develop early-warning signals for these critical transitions from an applied perspective can be found in \cite{Schefferetal,Scheffer}. The goal of a mathematical theory should be to provide qualitative and quantitative conditions to check whether a drastic change in a dynamical system can be predicted before it occurs; note that it is obvious that certain transitions are very difficult to predict, for example, due to large noise effects \cite{DitlevsenJohnsen} or non-smooth transitions \cite{HastingsWysham}. 

A basic assumption in many applications is that the underlying process is deterministic but is subject to small random fluctuations. Furthermore, one often assumes that the change occurs rapidly in comparison to the current state dynamics. Elementary remarks how the mathematical theory of stochastic fast-slow systems can be used to encapsulate these hypotheses can be found in \cite{KuehnCT1}. In particular, several one-parameter normal form models were studied and a more detailed link between rising variance \cite{CarpenterBrock}, rising autocorrelation \cite{Dakosetal}, time series analysis and dynamical models was pointed out. For additional references on critical transitions we refer the reader to \cite{Schefferetal} and \cite{KuehnCT1}.\\

We outline our results without stating detailed technical assumptions. It will be assumed that the main dynamics near a critical transition is governed by an ordinary differential equation (ODE). A classification which bifurcations are critical transitions based on a definition suggested in \cite{KuehnCT1} is explained. In a suitable singular limit this classification is a simple exercise dealing with all bifurcations up to codimension two. Some of the details for this classification are explained since one has to determine, at least once, which conditions on the fast and slow subsystems of a multi-scale system near higher-codimension bifurcations lead to trajectories that resemble critical transitions observed in applications. To model the random fluctuations stochastic differential equations (SDEs) with sufficiently small white noise are used. We calculate asymptotic formulas for all possible covariance matrices associated to sample paths approaching a critical transition. The setup for the calculations is straighforward and is based on normal form assumptions, approximation by Ornstein-Uhlenbeck processes and the solution of a few algebraic equations. An error estimate for the asymptotic expansions is proven for the fold bifurcation by analyzing stochastic difference processes and applying elementary moment estimates, thereby avoiding more advanced techniques \cite{BerglundGentz} for a certain regime. The focus on the fold bifurcation is justified as it is one of the most frequently encountered critical transitions \cite{ThompsonSieber2,GuttalJayaprakash2}. For the same reason, we also provide higher-order asymptotic expansions for the variance as doubly singular limit expansions with small noise and small time scale separation for the approach towards a fold point. 

Then we use the mathematical results in a wide variety of models. For each application the theoretical predictions are compared with numerical results. We briefly describe which important results are obtained within the examples. In a box-model of atmospheric and ocean circulation we test different approaches to estimate the variance from a given time series and suggest a new method motived by fast-slow systems. In a discrete epidemic spreading model a moment expansion is used to simplify an adaptive network and to analyze the onset of an epidemic. It is shown that predictability in adaptive networks can be improved by focusing on link dynamics instead of node dynamics. A model from systems biology is used to explain the effect of two critical transitions linked in a three-time scale systems. A predator-prey model illustrates the effect of multiple system parameters which can potentially hide early-warning signals that are usually expected to occur in ecology. The last example treats buckling of a spring in the context of a biomechanics experiment. The model for this experiment shows how parameter-dependent non-additive noise influences, and systematically changes, observed early-warning signs. The examples from epidemics, biomechanics and systems biology also seem to be among the first (or even the first) ones where early-warning signs for critical transitions have been applied in the respective fields.\\

In summary, our theoretical results combine well-known elementary mathematical tools from bifurcation theory, fast-slow systems, real analysis, stochastic differential equations, probability, asymptotic analysis, numerical continuation/integration and time series analysis to systematize some of the aspects of critical transitions. In this way, we make progress towards the major open question to develop a unified critical transitions theory \cite{Schefferetal}. Although no complicated technical steps are treated we hope that our work forms a starting point to motivate new mathematical insights into predictability for dynamical systems; see also Section \ref{sec:conclusions}. Our second contribution is to show that abstract critical transitions theory can yield very useful conclusions with immediate value for applications.\\

The paper is structured as follows. Section \ref{sec:fast_slow} describes the background from deterministic fast-slow systems. Section \ref{sec:nforms} contains the classification results. Section \ref{sec:SDE_fs} reviews theory for stochastic fast-slow systems based on which we prove the error estimate for asymptotic moment results near folds. In Section \ref{sec:variance} the leading-order asymptotic scaling laws for the covariance are obtained and in Section \ref{sec:fold_asymp} these results are refined for the fold. Section \ref{sec:applications} contains the five important examples. Section \ref{sec:conclusions} provides an outlook how the framework developed here could be extended.\\

\textit{Convention:} Whenever a citation with detailed page numbers at the beginning of a result (Theorem, Lemma, etc.) is given then the statement and proof can be found in the reference.

\section{Brief Review of Fast-Slow Systems}
\label{sec:fast_slow}

We recall the necessary definitions and results from multiple time scale dynamics \cite{Desrochesetal,Jones,MisRoz,Grasman} that are required to define critical transitions. A \texttt{fast-slow system} of (ODEs) is given by
\be
\label{eq:gen1}
\begin{array}{rcrcl}
\epsilon \frac{dx}{ds}&=&\epsilon \dot{x}&=&f(x,y,\epsilon),\\
\frac{dy}{ds}&=& \dot{y}&=& g(x,y,\epsilon),\\
\end{array}
\ee
where $0<\epsilon\ll1$, $x\in\R^m$ are \texttt{fast variables} and $y\in\R^n$ are \texttt{slow variables}. The maps $f:\R^{m+n+1}\ra \R^m$ and $g:\R^{m+n+1}\ra \R^n$ are assumed to be sufficiently smooth. If $f,g$ do not depend on $\epsilon$ we omit the $\epsilon$-argument and write {e.g.} $f(x,y)$ instead of $f(x,y,\epsilon)$. Changing in \eqref{eq:gen1} from the \texttt{slow time} $s$ to the \texttt{fast time} $t=s/\epsilon$ gives
\be
\label{eq:gen2}
\begin{array}{rcrcr}
\frac{dx}{dt}&=&x'&=&f(x,y,\epsilon),\\
\frac{dy}{dt}&=&y'&=& \epsilon g(x,y,\epsilon).\\
\end{array}
\ee
Henceforth $(\dot{~})$ will denote differentiation with respect to the slow time $s$ and prime differentiation with respect to the fast time $t$. The \texttt{singular limit} $\epsilon\ra 0$ in \eqref{eq:gen1} yields the \texttt{slow subsystem}
\be
\label{eq:sss}
\begin{array}{rcl}
0&=& f(x,y,0),\\
\dot{y}&=&g(x,y,0),\\
\end{array}
\ee
which is a \texttt{differential-algebraic equation} restricted to the \texttt{critical manifold} $C_0:=\{(x,y)\in\R^{m+n}:f(x,y,0)=0\}$. The \texttt{fast subsystem} is obtained as the singular limit of \eqref{eq:gen2}
\be
\label{eq:fss}
\begin{array}{rcl}
x'&=& f(x,y,0),\\
y'&=& 0,\\
\end{array}
\ee
where the slow variables can be viewed as parameters. The flows generated by \eqref{eq:sss} and \eqref{eq:fss} are called the \texttt{slow flow} and the \texttt{fast flow} respectively. A point $p\in C_0$ is an equilibrium point of the fast subsystem. The critical manifold is \texttt{normally hyperbolic} at $p\in C_0$ if the $m\times m$ matrix $D_xf(p)$ has no eigenvalues with zero real parts. In this case, the implicit function theorem provides a map $h_0:\R^n\ra \R^m$ that describes $C_0$, locally near $p$, as a graph $C_0=\{(x,y)\in\R^{m+n}:x=h_0(y)\}$. Then the slow subsystem \eqref{eq:sss} can be written more concisely as $\dot{y}=g(h_0(y),y)$. If all eigenvalues of $D_xf(p)$ are negative (positive) then $C_0$ is \texttt{attracting} (\texttt{repelling}); other normally hyperbolic critical manifolds are of \texttt{saddle-type}. Observe that $C_0$ is attracting at $p$ if and only if the fast subsystem has a stable hyperbolic equilibrium at $p$. Fenichel's Theorem provides a complete description of the dynamics for normally hyperbolic invariant manifolds for sufficiently smooth vector fields  $(f,g)$. To state the result, we recall that the \texttt{Hausdorff distance} between two sets $V,W\subset \R^{m+n}$ is given by
\benn
d_H(V,W)=\max\left\{\sup_{v\in V} \inf_{w\in W}\|v-w\|,\sup_{w\in w} \inf_{v\in V}\|v-w\|\right\}
\eenn 
where $\|\cdot\|$ denotes the Euclidean norm.

\begin{thm}[\texttt{Fenichel's Theorem}~\cite{Fenichel4}]
\label{thm:fenichel1}
Suppose $S = S_0$ is a compact normally hyperbolic submanifold (possibly with boundary) of the critical manifold $C_0$. Then for $\epsilon > 0$ sufficiently small there exists a locally invariant manifold $S_\epsilon$ diffeomorphic to $S_0$. $S_\epsilon$ has a Hausdorff distance of $\cO(\epsilon)$ from $S_0$ and the flow on $S_\epsilon$ converges to the slow flow as $\epsilon \to 0$. $S_\epsilon$ is normally hyperbolic and has the same stability properties with respect to the fast variables as $S_0$ (attracting, repelling or saddle type). 
\end{thm}

$S_\epsilon$ is called a \texttt{slow manifold} and is usually not unique. In regions that remain at a fixed distance from the boundary of $S_\epsilon$, all manifolds satisfying Theorem \ref{thm:fenichel1} lie at a Hausdorff distance $\cO(e^{-K/\epsilon})$ from each other for some $K > 0$ with $K = \cO(1)$. The choice of representative will be irrelevant for the asymptotic analysis we are interested in here; see also \cite{MKKR}. If the choice of subset $S_0$ is understood then we also write $C_\epsilon$ for the slow manifold associated to $C_0$ and refer to $C_\epsilon$ as ``the'' \texttt{slow manifold}.\\

\begin{figure}[htbp]
\centering
\psfrag{R1}{(R2)}
\psfrag{R2}{(R1)}
\psfrag{x}{$x$}
\psfrag{y}{$y$}
\psfrag{C0}{$C_0$}
\psfrag{g0}{$\gamma_0$}
\psfrag{fold}{$\label{eq:fold} \left\{\begin{array}{rcl}\epsilon \dot{x}&=&y-x^2\\ \dot{y}&=&-1 \\ \end{array}\right.$}
 \includegraphics[width=0.85\textwidth]{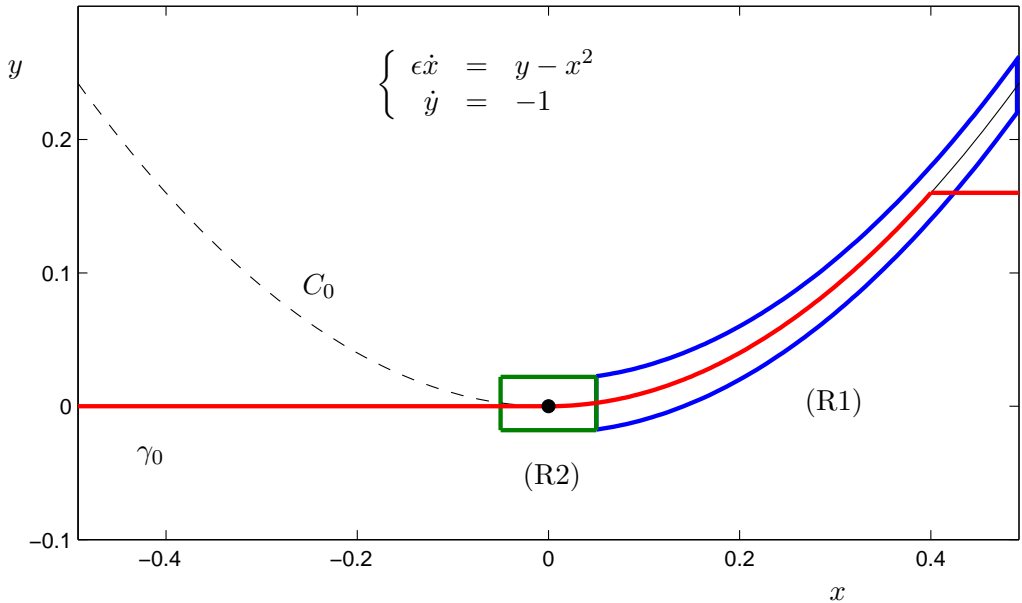}
\caption{\label{fig:fig1}Critical transition at a fold bifurcation of the fast subsystem. The critical manifold $C_0$ splits into a repelling part (dashed grey) and an attracting part (solid black). A typical candidate trajectory $\gamma_0$ (red) is shown. We have also sketched the two different regions: (R1, blue) where normal hyperbolicity of the critical manifold holds and (R2, green) near the bifurcation point.}
\end{figure}

A \texttt{candidate trajectory} $\gamma_0$ is a concatenation of slow and fast subsystem trajectories; see Figure \ref{fig:fig1}. More precisely we define a candidate as a homeomorphic image $\gamma_0(t)$ of a partitioned interval $a = t_0 < t_1 < \cdots < t_m = b$, where the image of each subinterval $\gamma_0(t_{j-1},t_j)$ is a trajectory of either the fast or the slow subsystem and the image $\gamma_0(a,b)$ has an orientation that is consistent with the orientations on each subinterval $\gamma_0(t_{j-1},t_j)$ induced by the fast and slow flows. Note that the intervals $(t_j,t_{j+1})$ do not necessarily correspond to the time parametrizations of the fast or slow subsystem; to achieve such a parametrization one has to pick a convention such as using the slow time and compactifying infinite time intervals as necessary. 

If consecutive images $\gamma_0(t_{j-1},t_j)$ and $\gamma_0(t_{j},t_{j+1})$ are trajectories for different subsystems then we say that $\gamma_0(t_j)$ is a \texttt{transition point}.

\begin{defn}
\label{defn:ct}
Let $p=(x_p,y_p)\in C_0$ be a point where the critical manifold $C_0$ is not normally hyperbolic. We say that $p$ is a \texttt{critical transition} if there is a candidate $\gamma_0$ so that 
\begin{itemize}
 \item [(C1)] $\gamma_0(t_{j-1},t_j)$ is contained in a normally hyperbolic attracting submanifold of $C_0$,
 \item [(C2)] $p=\gamma_0(t_j)$ is a transition point,
 \item [(C3)] and $\gamma_0(t_{j-1},t_j)$ is oriented from $\gamma_0(t_{j-1})$ to $\gamma_0(t_j)$. 
\end{itemize}
\end{defn}

Definition \ref{defn:ct} was suggested in \cite{KuehnCT1}. It is related to the concept of ``hard'' or ``catastrophic'' loss of stability (\cite{Kuznetsov}, p.87 or \cite{ArnoldEncy}, p.36) but does not coincide with it. Note that Definition \ref{defn:ct} is entirely based upon the singular limit $\epsilon=0$. Definition \ref{defn:ct} is simple, easy to verify for a system, concretely includes the focus on the candidate orbit occuring in an actual time series and also seems to represent all the requirements laid out in \cite{Schefferetal}. Note carefully that (C1) excludes slow canard orbit segments in repelling parts of the critical manifold but see Section \ref{sec:conclusions} for possible extensions.     

\begin{prop}
\label{prop:negative}
Suppose $p=(x_p,y_p)$ is Lyapunov stable with respect to the fast subsystem, then there is no critical transition at $p$.
\end{prop}

\begin{proof}
Suppose $\gamma_0$ is a candidate that satisfies (C1) of Definition \ref{defn:ct}. If $p= \gamma_0(t_j)$ is a transition point then the orientation condition (C3) implies that the $\gamma_0(t_{j},t_{j+1})$ is an orbit segment in the fast subsystem starting from $p$. Since $p$ is Lyapunov stable we have reached a contradiction. 
\end{proof}

In this paper, we are interested in the approach of trajectories to critical transitions as illustrated in Figure \ref{fig:fig1}. This approach towards a critical transitions can be subdivided into two main regions: (R1) Fenichel's Theorem applies near a normally hyperbolic critical manifold and (R2) Fenichel's Theorem fails near the bifurcation point. Note that Fenichel's Theorem implies that near a local bifurcation point $(x,y)=(x_p,y_p)$ of the fast subsystem the region (R2) shrinks to $(x_p,y_p)$ as $\epsilon \ra 0$. By making $\epsilon$ sufficiently small, we should start to focus on (R1). Whenever we consider decreasing $y\ra y_p$ we shall make the assumption that $\epsilon$ has been chosen small enough so that we stay inside region (R1). For example, for a \texttt{fold point} \cite{KruSzm1} it is known that the size of (R2) scales like
\benn
(x,y)\sim(\cO(\epsilon^{1/3}),\cO(\epsilon^{2/3}))\in \R^2
\eenn 
as $\epsilon \ra 0$; see also Lemma \ref{lem:asympCeps}. Therefore we would assume that $\epsilon^{1/3}\ll y$ as $y$ gets small. We formalize our assumptions by restricting the analysis to a compact domain contained in the region (R1):

\begin{itemize}
 \item[(A0)] Fast-slow systems will be considered on a compact domain for $\cD(\epsilon)=\cD=\cD_x\times \cD_y\subset \R^m\times \R^n$ that depends smoothly on $\epsilon$. $\cD(\epsilon)$ is chosen so that an attracting slow manifold $C^a_\epsilon$ is contained in $\cD(\epsilon)$, the intersection $\partial \cD(\epsilon)\cap C^a_\epsilon$ is transverse and $\cD(\epsilon)$ is contained in the basin of attraction of $C^a_\epsilon$; the slow manifold $C^a_\epsilon$ is given locally as a graph $C^a_\epsilon=\{(x,y)\in \cD(\epsilon):x=h_\epsilon(y)\}, \text{  for $h_\epsilon:\cD_y\ra \cD_x$.}$ Furthermore, fast subsystem local bifurcation points will lie on $\partial \cD(0)$ and asymptotics with respect to $y\ra y_p$ is chosen depending on $\epsilon$ so that normal hyperbolicity holds; see Figure \ref{fig:fig1}.
\end{itemize}

Note that this means that all scaling estimates we derive are restricted to a bounded domain. Within this bounded domain no other attracting critical manifold perturbs to a slow manifold.

\section{Fast Subsystem Normal Forms and Critical Transitions}
\label{sec:nforms}

We assume familiarity with the normal form approach to bifurcation theory (\cite{GH}, p.138) and apply it in the singular limit to the fast subsystem viewing $y\in\R^n$ as parameters. The number of slow variables $y\in\R^n$ is chosen as the codimension of the bifurcation. We are going to check which bifurcations are critical transitions in the sense of Definition \ref{defn:ct}. Note carefully that this classification, although complete on the singular limit level $\epsilon=0$, has interesting possible extensions which are discussed in Section \ref{sec:conclusions}. 

Assume without loss of generality that the bifurcation point is at $x=(0,\ldots,0)=:0$ and $y=(0,\ldots,0)=:0$. To reduce the analysis to normal forms we will assume that all the necessary genericity conditions (non-degeneracy and transversality) are satisfied \cite{Kuznetsov}: 

\begin{itemize}
 \item[(A1)] $f\in C^{r}(\R^{m+n+1},\R^m)$ where $r\geq 1$ is chosen according to the differentiability required by normal form theory. We also assume that $g\in C^2(\R^{m+n+1},\R^n)$.
 \item[(A2)] The genericity conditions for bifurcations hold so that normal form theory applies.
\end{itemize}

The only generic codimension one bifurcation with $x\in \R^1$ is the \texttt{fold bifurcation} with normal form (\cite{Kuznetsov}, p.84)
\be
\label{eq:nf_fold}
f(x,y)=-y-x^2.
\ee
Considering the dynamics of the slow variable as $g(x,y)$ \cite{KuehnCT1} we get the fast-slow system
\be
\label{eq:nf_fold1}
\begin{array}{rcl}
 x'&=&-y-x^2,\\
 y'&=&\epsilon g(x,y),\\
\end{array}
\ee
where $y\in\R^1$ since we have a codimension one bifurcation. The critical manifold for \eqref{eq:nf_fold1} is $C_0=\{(x,y)\in\R^2:x=\pm \sqrt{-y}=:h_0^\pm(y),y\leq 0\}$. $C_0\cap\{x=\sqrt{-y},y<0\}$ is attracting and $C_0\cap \{x=-\sqrt{-y},y<0\}$ is repelling. The linearization $D_xf(h_0^+(y),y)$ around the attracting branch of the critical manifold will be crucial for the calculations in Section \ref{sec:variance} as it describes the dynamics in the region (R1); therefore we will calculate/record this linearization for each critical transition. 

\begin{lem}
\label{lem:fold}
If $g(0,0)>0$ then \eqref{eq:nf_fold1} has a critical transition at $(x,y)=(0,0)$.
\end{lem}

The proof is obvious and similar results hold for the pitchfork and transcritical normal forms   
\bea
f(x,y)&=&yx+sx^3,\quad D_xf(h_0(y),y)=y,\quad \text{for $s=\pm 1$, (\cite{Kuznetsov}, p.284)}, \label{eq:nf_pitchfork}\\
f(x,y)&=&yx-x^2,\quad D_xf(h_0(y),y)=y. \quad \text{(\cite{GH}, p.149))} \label{eq:nf_transcritical}.
\eea

\begin{lem}
If $g(0,0)>0$ then \eqref{eq:nf_pitchfork} has a critical transition at $(x,y)=(0,0)$ if and only if the pitchfork bifurcations is subcritical ($s=1$). If $g(0,0)\neq0$ then \eqref{eq:nf_transcritical} has a critical transition at $(x,y)=(0,0)$.
\end{lem}

The remaining one-dimensional fast subsystem is the codimension two \texttt{cusp bifurcation}. The normal form is (\cite{Kuznetsov}, p.304-305)
\be
\label{eq:nf_cusp}
f(x,y)=y_1+y_2x+sx^3, \qquad \text{for $s=\pm 1$}  
\ee
where the fast dynamics $x'=f(x,y)$ is augmented with two-dimensional slow dynamics $y'=\epsilon g(x,y)$, $y=(y_1,y_2)\in\R^2$. The critical manifold for \eqref{eq:nf_cusp} is $C_0=\{(x,y)\in\R^3:0=y_1+y_2x+sx^3\}$. Due to the two-dimensional slow flow it is slightly less obvious to determine under which conditions the cusp bifurcation is a critical transition.

\begin{lem}
There is no critical transition for \eqref{eq:nf_cusp} at $(x,y)=(0,0)$ if $s=-1$. If $s=1$ then \eqref{eq:nf_cusp} has a critical transition at $(x,y)=(0,0)$ if and only if $g_2(0,0)>0$ and $g_1(0,0)=0$.
\end{lem}

\begin{proof}
First consider the case $s=-1$. At $y_1=0=y_2$ the fast subsystem is $x'=-x^3$. It is easy to see that $x=0$ is asymptotically stable and Proposition \ref{prop:negative} implies that there cannot be a critical transition at $(x,y)=(0,0)$. For $s=1$ the fast subsystem is $x'=x^3$ so that a candidate orbit $\gamma_0$ can have a segment $\gamma_0(t_j,t_{j+1})$ in the fast subsystem oriented from $\gamma_0(t_j)$ to $\gamma_0(t_{j+1})$. To see that there exists an attracting critical manifold connecting to $(x,y)=(0,0)$ we need the unfolding of a cusp bifurcation. The linearization of \eqref{eq:nf_cusp} is $D_xf(x,y)=y_2+3x^2$ and the stability of the slow manifold changes at fold points when $D_xf|_{C_0}=0$. Given the two equations 
\benn
\begin{array}{lcl}
0 &=& y_1+y_2x+x^3\\
0 &=& y_2+3x^2\\
\end{array}
\eenn
the variable $x$ can be eliminated which yields the classical cusp curve. After a projection into the $(y_1,y_2)$-plane it is given by $\Gamma:=\{(y_1,y_2)\in\R^2:4y_2^3+27y_1^2=0\}$. A repelling subset of the critical manifold is $C^r_0:=C_0\cap\{4y_2^3+27y_1^2>0\}$. The set $C_0\cap\{4y_2^3+27y_1^2<0\}$ splits into three branches corresponding to the three solutions of $f(x,y)=0$ where two branches $C^{r\pm}_0$ are repelling and one branch $C^a_0$ is attracting; observe that $y_2<0$ for any of the three branches. Now consider a candidate $\gamma_0$ with $\gamma_0(t_{j-1},t_j)\subset C^a_0=\{x=h_0(y)\}$. The slow flow on $C^a_0$ is given by 
\be
\label{eq:sf_cusp}
\begin{array}{lcl}
\dot{y}_1&=&g_1(h_0(y),y),\\
\dot{y}_2&=&g_2(h_0(y),y).\\
\end{array}
\ee
Since the $y_2$-component of $\gamma_0(t_{j-1})$ is negative we must have $g_2(0,0)>0$. Furthermore, we know that trajectories of \eqref{eq:sf_cusp} on $C^a$ reach $(y_1,y_2)=(0,0)$ if and only if $4y_1^3+27y_1^2<0$ holds for the $y$-components of $\gamma_0(t_{j-1},t_j)$. Considering the branches of $\Gamma$ we have
\benn
y_1=\Gamma_\pm(y_2)=\pm\sqrt{\frac{4}{27}y_2^3}\qquad\Rightarrow \quad \frac{d}{dy_2}\Gamma_\pm(y_2)=\pm\frac{y_2^2}{\sqrt{3y_2^3}}=\Gamma_\pm'(y_2).
\eenn
In particular, we find that $\lim_{y_2\ra 0}\Gamma_\pm'(y_2)=0$ which implies that $g_1(0,0)=0$ if the candidate $\gamma_0$ reaches $(x,y)=(0,0)$. Hence there exists $\gamma_0$ as required by Definition \ref{defn:ct} if and only if $g_2(0,0)>0$ and $g_1(0,0)=0$. 
\end{proof}

The attracting part of the critical manifold $C^a_0$ (as introduced in the previous proof) is $C^a_0=\{(x,y_1,y_2)\in\R^3:x=h_0(y)\}$. The linearization is
\be
\label{eq:A0_cusp}
D_xf(h_0(y),y)=y_2+3h_0(y)^2=\cO_y(y_2) \quad \text{as $y\ra 0$.}
\ee
where the notation $\cO_y(\cdot)$ indicates asymptotic scaling as $y\ra 0$ under the assumption that Fenichel Theory is still valid; see also Section \ref{sec:fast_slow} where this is referred to as region (R1). The asymptotic scaling in \eqref{eq:A0_cusp} holds since points on $C^a_0$ satisfy $y_2+3x^2<0$ because on $C^a_0$ we have 
\benn
y_2<0,\qquad x\in[-(-y_2/3)^{1/2},(-y_2/3)^{1/2}],\qquad  y_1=-xy_2-x^3
\eenn
Therefore, $x=h_0(y)$ grows at most like $\sqrt{y_2}$ as $y\ra 0$ and the scaling law in \eqref{eq:A0_cusp} follows. This concludes our discussion of one-dimensional fast subsystem bifurcations.\\ 

For two fast subsystem variables consider the codimension-one \texttt{Hopf bifurcation} normal form (\cite{Kuznetsov}, p.98)
\be
\label{eq:nf_Hopf}
\begin{array}{lcl}
f_1(x,y)&=&yx_1-x_2+l_1x_1(x_1^2+x_2^2),\\
f_2(x,y)&=&x_1+yx_2+l_1x_2(x_1^2+x_2^2),\\
\end{array}
\ee
where $l_1$ is the \texttt{first Lyapunov coefficient}. The critical manifold for \eqref{eq:nf_Hopf} is $C_0=\{(x,y)\in\R^3:x=0\}$ where $C_0\cap \{y<0\}$ is attracting and $C_0\cap \{y>0\}$ is repelling and the linearization is
\be
\label{eq:A0_Hopf}
D_xf(0,y)=\left(\begin{array}{cc}y & -1 \\ 1 & y \\ \end{array}\right).
\ee  

\begin{lem}
If $g(0,0)>0$ then \eqref{eq:nf_Hopf} has a critical transition at $(x,y)=(0,0)$ if and only if the Hopf bifurcation is subcritical ($l_1>0$).
\end{lem}

For vanishing first first Lyapunov coefficient ($l_1=0$) a codimension-two \texttt{generalized Hopf} (or \texttt{Bautin}) bifurcation occurs with normal form (\cite{Kuznetsov}, p.313)
\be
\label{eq:nf_Bautin}
\begin{array}{lcl}
f_1(x,y)&=&y_1x_1-x_2+y_2x_1(x_1^2+x_2^2) +l_2x_1(x_1^2+x_2^2)^2,\\
f_2(x,y)&=&x_1+y_1x_2+y_2x_2(x_1^2+x_2^2) +l_2x_2(x_1^2+x_2^2)^2,\\
\end{array}
\ee
where $l_2=\pm 1$ is the \texttt{second Lyapunov coefficient}. The critical manifold is $C_0=\{(x,y)\in\R^4:x=0=:h_0(y)\}$. The linearization $D_xf(h_0(y),y)$ coincides with the linearization \eqref{eq:A0_Hopf} for the Hopf bifurcation upon replacing $y$ by $y_1$.

\begin{lem}
\label{eq:Baut_lem1}
The Bautin bifurcation is not a critical transition if $l_2<0$.
\end{lem}

\begin{proof}
Without loss of generality let $l_2=-1$ then the fast subsystem at $(y_1,y_2)=(0,0)$ is  
\be
\label{eq:Bautin_neg_concl}
\begin{array}{lcl}
x_1'&=&-x_1(x_1^2+x_2^2)^2=:\tilde{f}_1(x),\\
x_2'&=&-x_2(x_1^2+x_2^2)^2=:\tilde{f}_2(x).\\
\end{array}
\ee
where $\tilde{f}=(\tilde{f_1},\tilde{f}_2)$. Define a function $V:\R^2\ra \R$ by $V(x_1,x_2):=x_1^2+x_2^2$. Observe that $V(x)>0$ for $x\neq 0$ and 
\benn
\frac{d}{dt}V(x)=D_xV(\tilde{f}(x))=-(2x_1^2+2x_2^2)(x_1^2+x_2^2)^2<0
\eenn
for $x\neq 0$. Therefore $V(x)$ is a Lyapunov function and $x=0$ is asymptotically stable as an equilibrium point of \eqref{eq:Bautin_neg_concl} and Proposition \ref{prop:negative} finishes the proof.
\end{proof}

\begin{lem}
\label{lem:Bautin}
If $l_2>0$ then the Bautin bifurcation is a critical transition if and only if $g_1(0,0)>0$ and either (a) $g_2(0,0)\neq 0$ or (b) $g_2(0,0)=0$ and $\frac{\partial g_2}{\partial y_2}(0,0)<\frac12$.
\end{lem}

\begin{proof}
The critical manifold splits into two 2-dimensional planes $C^a_0=C_0\cap \{y_1<0\}$ and $C^r_0=C_0\cap \{y_1>0\}$ where $C^a_0$ is attracting and $C^r_0$ is repelling. The condition $g_1(0,0)>0$ implies that the slow flow $\dot{y}=g(0,y)$ has trajectories that start in $C^a_0$ and reach $(x,y)=(0,0)$ in finite time. This guarantees the existence of a candidate $\gamma_0$ satisfying (C1) and (C3) of Definition \ref{defn:ct}. The proof of Lemma \ref{eq:Baut_lem1} shows, upon reversal of time in equation \eqref{eq:Bautin_neg_concl}, that $x=0$ is an asymptotically unstable equilibrium point of the fast subsystem at $y=0$ when $l_2=1$. We note from the unfolding of a Bautin bifurcation (see \cite{Kuznetsov}, p.314) that saddle-node bifurcations of limit cycles for the fast subsystem occur on the curve $\text{LPC}:=\{(y_1,y_2)\in\R^2:\frac14y_1^2=y_1,y_2<0\}$. The conditions on the slow flow guarantee that the candidate orbit enters the fast subsystem region without limit cycles and with an unstable equilibrium point; this region is given by
\benn
\left\{(y_1,y_2)\in\R^2:y_2<0, y_1>\frac14y_1^2\right\}\cup\left\{(y_1,y_2)\in\R^2:y_2>0,y_1>0\right\}. \qedhere
\eenn
\end{proof}

The last codimension two bifurcation with two fast variables is the \texttt{Bogdanov-Takens bifurcation} with normal form (\cite{GH}, p.365)
\be
\label{eq:BT_nform}
\begin{array}{lcl}
f_1(x,y)&=&x_2,\\
f_2(x,y)&=&y_1+y_2x_2+x_1^2+sx_1x_2,\\
\end{array}
\ee
where $s=\pm 1$. The critical manifold is $C_0=\left\{(x,y)\in\R^4: x_2=0, x_1=\pm \sqrt{-y_1}\right\}$ so that we always require $y_1\leq 0$. 

\begin{lem}
\label{lem:BT}
The Bogdanov-Takens bifurcation \eqref{eq:BT_nform} is a critical transition for $s=-1$ if and only if $g_2(0,0)>0$ and $g_1(0,0)=0$ and for $s=1$ if and only if (a) $g_1(0,0)>0$ or (b) $g_1(0,0)=0$, $g_2(0,0)>0$, $\frac{\partial g_2}{\partial y_2}(0,0)<-2$ .
\end{lem}

\begin{proof}
As usual we consider the fast subsystem at $y=0$
\be
\label{eq:csup_point}
\begin{array}{lcl}
x_1'&=&x_2,\\
x_2'&=&x_1^2+sx_1x_2.\\
\end{array}
\ee
The theory of non-hyperbolic equilibria in planar analytic vector fields (\cite{Perko}, p.151; see also \cite{AndronovLeontovichGordonMaier}) implies that $x=0$ is a cusp point. Hence there exists a candidate $\gamma_0$ with a fast-subsystem orbit segment $\gamma_0(t_j,t_{j+1})$ oriented from $\gamma_0(t_j)=(0,0)$ to $\gamma_0(t_{j+1})$. It remains to show when we can approach $(x,y)=(0,0)$ via the slow flow on an attracting critical manifold. The linearization around the critical manifold is
\be
\label{eq:A0_BT}
D_xf|_{C_0}=\left(\begin{array}{cc} 0 & 1 \\ \pm2\sqrt{-y_1} & y_2\pm s\sqrt{-y_1}\\\end{array}\right)
\ee
with $\text{Tr}(D_xf|_{C_0})=y_2\pm s\sqrt{-y_1}$ and $\det(D_xf|_{C_0})=-2(\pm\sqrt{-y_1})$. Hence the critical manifold is attracting if and only if $y_2\pm s\sqrt{-y_1}<0$ and $2(\pm\sqrt{-y_1})<0$. For $s=-1$ this yields the set
\benn
C^a_0=\{(x_1,x_2,y_1,y_2)\in\R^4:y_1<0,x_2=0,x_1=-\sqrt{-y_1},y_2<-\sqrt{-y_1}\}.
\eenn
A candidate orbit starting in $C^a_0$ will reach $(x,y)=(0,0)$ if and only if $g_2(0,0)>0$ and $g_1(0,0)=0$ where the second condition is required since the curve $\{y_2=-\sqrt{-y_1}\}$ approaches the origin tangentially i.e. $\frac{d}{dy_2}(-y_2^2)|_{y=0}=0$. The second case for $s=1$ is similar and follows using the symmetry $(x_1,x_2,y_1,y_2,t)\mapsto (x_1,-x_2,y_1,-y_2,-t)$.
\end{proof}

The linearization for the Bogdanov-Takens bifurcation has already been recorded in \eqref{eq:A0_BT} but notice that we must have the condition $y_2<\mp s\sqrt{-y_1}$ to be on $C^a_0$. The asymptotic result as $y\ra 0$ for the linearization is
\be
\label{eq:A0_BT1}
D_xf|_{C_0\cap C^a_0}=\left(\begin{array}{cc} 0 & 1 \\ \pm2\sqrt{-y_1} & \cO_y(\sqrt{-y_1})\\\end{array}\right).
\ee
The two remaining codimension-two bifurcations (fold-Hopf and Hopf-Hopf) require three and four fast dimensions. The \texttt{fold-Hopf} (or \texttt{Gavrilov-Guckenheimer}) bifurcation has normal form (\cite{Kuznetsov}, p.338)
\be
\label{eq:fH_nform}
\begin{array}{lcl}
f_1(x,y)&=& y_1+x_1^2+s(x_2^2+x_3^2),\\
f_2(x,y)&=& y_2x_2-\omega x_3+\theta x_1x_2-x_1x_3+x_1^2x_2,\\
f_3(x,y)&=& \omega x_2+y_2x_3+x_1x_2+\theta x_1x_3+x_1^2x_3,\\
\end{array}
\ee
where $s=\pm 1$ and $\theta=\theta(y)$ satisfies $\theta(0)\neq 0$ and $\omega\neq 0$. The critical manifold is given by 
\benn
C_0=\{(x,y)\in\R^5:x_2=0=x_3,x_1=\pm\sqrt{-y_1},y_1\leq 0\}.
\eenn

\begin{lem}
\label{lem:fH}
The fold-Hopf bifurcation is a critical transition 
\begin{itemize}
 \item for $\theta(0)>0$, $s=1$ if and only if (a) $g_1(0,0)>0$ or (b) $g_1(0,0)=0$ and $g_2(0,0)>0$,
 \item for $\theta(0)>0$, $s=-1$ if and only if $g_1(0,0)>0$ and $g_2(0,0)<J_2(0,0)$ where $J_2(0,0)$ is the $y_2$-component of the tangent vector to the ``cycle blow-up curve'' (cf. \cite{Kuznetsov}, p.343),
 \item for $\theta(0)<0$, $s=1$ if and only if $g_1(0,0)=0$ and $g_2(0,0)>0$,
 \item for $\theta(0)<0$, $s=-1$ if and only if $g_1(0,0)=0$ and $g_2(0,0)>0$.
\end{itemize}
\end{lem}

\begin{proof}
The same techniques as before will apply so we just sketch the proof. The fast subsystem at $y=0$ is
\be
\label{eq:fH_fss}
\begin{array}{lcl}
x_1'&=& x_1^2+s(x_2^2+x_3^2),\\
x_2'&=& -\omega x_3+\theta x_1x_2-x_1x_3+x_1^2x_2,\\
x_3'&=& \omega x_2+x_1x_2+\theta x_1x_3+x_1^2x_3,\\
\end{array}
\ee
Changing to cylindrical coordinates $(x_2,x_3)=(r\cos\phi,r\sin \phi)$ in \eqref{eq:fH_fss} and neglecting the angular component $\phi$, since it is always a neutral direction with respect to attraction and repulsion for the critical manifold of equilibrium points, we get a two-dimensional system
\be
\label{eq:fH_fss1}
\begin{array}{lcl}
x_1'&=& x_1^2+sr^2,\\
r'&=& r(\theta x_1+x_1^2).\\
\end{array}
\ee
It can be checked that the origin $(x_1,r)=(0,0)$ is unstable for \eqref{eq:fH_fss}. Therefore we can find a candidate that leaves the bifurcation point in a fast direction. The attracting part of the critical manifold is computed from the linearization 
\be
\label{eq:A0_fH}
D_xf|_{C_0}=\left(\begin{array}{ccc} \pm 2\sqrt{-y_1} & 0 & 0 \\ 
0 & y_2\pm\theta\sqrt{-y_1} & -\omega \mp\sqrt{-y_1} \\
0 & \omega\pm \sqrt{-y_1} & y_2\pm\theta\sqrt{-y_1}\\ \end{array}\right)
\ee
and is given by $C^a_0=C_0\cap \{x=-\sqrt{-y_1},y_2<\theta\sqrt{-y_1}\}$. The conditions on the slow flow $\dot{y}=g=(g_1,g_2)$ can be derived from the unfolding of the fold-Hopf bifurcation (see \cite{Kuznetsov}, p.339-345).
\end{proof}

The linearization is given by \eqref{eq:A0_fH}; we note that the condition of the approach via $C^a_0$ means that the leading order approximation to $D_xf|_{C_0\cap C_0^a}$ as $(y_1,y_2)\ra (0^-,0)$ is given by
\be
\label{eq:A0_fH1}
D_xf|_{C_0\cap C_0^a}=\left(\begin{array}{ccc} \pm 2\sqrt{-y_1} & 0 & 0 \\ 
0 & \cO_y(\sqrt{-y_1}) & -\omega \\
0 & \omega & \cO_y(\sqrt{-y_1})\\ \end{array}\right).
\ee
As the last case we are going to consider is the \texttt{Hopf-Hopf bifurcation}. We shall not discuss the complicated unfolding (\cite{Kuznetsov}, p.351-370; \cite{GH}, p.396-411) in detail to show when the Hopf-Hopf bifurcation is a critical transition. A normal form in polar coordinates $(r_1,r_2,\theta_1,\theta_2)=(r,\theta)$ is (\cite{Kuznetsov}, p.358)
\be
\label{eq:HH_nform}
\begin{array}{lcl}
f_1(r,\theta,y)&=& r_1(y_1+p_{11}r_1^2+p_{12}r_2^2+s_1r_2^4),\\
f_2(r,\theta,y)&=& r_2(y_2+p_{21}r_1^2+p_{22}r_2^2+s_2r_1^4),\\
f_3(r,\theta,y)&=& \omega_1,\\
f_4(r,\theta,y)&=& \omega_2,\\
\end{array}
\ee
where $\omega_{1,2}$ are the imaginary parts of the eigenvalues at the bifurcation point $y=0$; $p_{ij}$ and $s_{1,2}$ are further parameters. We note that all parameters also depend on the slow variables $y$ but are usually assumed to be non-zero at the bifurcation point. Observe from \eqref{eq:HH_nform} that the critical manifold is
\benn
C_0=\{(r,\theta,y)\in\R^4\times \R^2:r_1=0=r_2\}=\{(x,y)\in\R^4\times \R^2:x_j=0\text{ for $j=1,2,3,4$}\}.
\eenn  
To study whether candidate orbits can leave the fast subsystem for $y=0$ we would have to study the nonlinear stability of the origin depending on the parameters. It is not difficult to see that the Hopf-Hopf bifurcation is not always a critical transition depending on parameter values but there are cases when it is a critical transition. Instead of providing this detailed study (which can be inferred from the unfoldings in \cite{Kuznetsov}) we shall only state one important linearization
\be
\label{eq:A0_HH}
D_xf(x,y)|_{C_0\cap C_0^a}=
\left(\begin{array}{cccc}y_1 & -\omega_1 & 0 & 0 \\ \omega_1 & y_1 & 0 & 0 \\0 & 0 & y_2 & -\omega_2 \\ 0 & 0 & \omega_2 & y_2\\\end{array}\right).
\ee
It will always be assumed for \eqref{eq:A0_HH} that $\omega_{1,2}\neq 0$; note that we also exclude resonances $k\omega_1=l\omega_2$ for $k+l\leq3$ as non-resonance conditions are non-degeneracy conditions for the Hopf-Hopf bifurcation ({cf.} assumption (A2)). We remark that the main bifurcation phenomena of interest near a Hopf-Hopf bifurcation are global orbits (limit cycles and tori) which would not be captured by our local analysis anyway.

\begin{table}[htbp]
 \centering
\begin{tabular}{|l|l|l|c|} 
\hline
\textbf{Name} & \textbf{N.-Form} & \textbf{Critical, if...} & \textbf{$D_xf|_{C_0\cap C^a_0}=:A_0(y)$} \\
\hline
Fold & Eq. \eqref{eq:nf_fold} & $g>0$ & $ -2\sqrt{-y}$ \\
\hline
Pitchfork & Eq. \eqref{eq:nf_pitchfork} & $s=1$, $g>0$ & $y$ \\
\hline
Transcr. & Eq. \eqref{eq:nf_transcritical} & $g\neq 0$ & $y$ \\
\hline
Hopf & Eq. \eqref{eq:nf_Hopf} & $l_1>0$, $g>0$ & $\left(\begin{array}{cc}y & -1 \\ 1 & y \\ \end{array}\right)$ \\
\hline
Cusp & Eq. \eqref{eq:nf_cusp} & $s=-1$, $g_1=0$, $g_2>0$  & $\cO_y(y_2)$ \\
\hline
Bautin & Eq. \eqref{eq:nf_Bautin} & $\begin{array}{l} l_2>0,~ g_1>0 \\ \text{and (a) } g_2\neq 0 \text{ or}\\ \text{(b) }g_2=0,~\partial_{y_2}g_2<1/2\\  \end{array}$ & $\left(\begin{array}{cc}y_1 & -1 \\ 1 & y_1 \\ \end{array}\right)$ \\
\hline
Bog.-Tak. & Eq. \eqref{eq:BT_nform} & $\begin{array}{l}s=-1, ~g_1>0,~ g_2=0 \\ s=1 \text{ and (a) }g_1>0 \text{ or}\\\text{(b) }g_1=0,g_2>0,\partial_{y_2}<-2 \\ \end{array}$& $\left(\begin{array}{cc} 0 & 1 \\ \pm2\sqrt{-y_1} & \cO_y(\sqrt{-y_1})\\ \end{array}\right)$ \\
\hline
Fold-Hopf & Eq. \eqref{eq:fH_nform} & $\begin{array}{l}\theta<0, ~g_1>0,~ g_2=0 \\ \theta>0,s=1 \text{ and (a) } g_1>0, \\ \text{or (b) }g_1=0,g_2>0\\ \theta>0,s=-1,g_1>0,g_2<J_2\\  \end{array}$ & $\left(\begin{array}{ccc} \pm 2\sqrt{-y_1} & 0 & 0 \\ 
0 & \cO_y(\sqrt{-y_1}) & -\omega \\
0 & \omega & \cO_y(\sqrt{-y_1})\\ \end{array}\right)$ \\
\hline
Hopf-Hopf & Eq. \eqref{eq:HH_nform} & $\begin{array}{l}\text{special case only} \\ \end{array}$ & $\left(\begin{array}{cccc}y_1 & -\omega_1 & 0 & 0 \\ \omega_1 & y_1 & 0 & 0 \\0 & 0 & y_2 & -\omega_2 \\ 0 & 0 & \omega_2 & y_2\\\end{array}\right)$ \\
\hline
\end{tabular}
\caption{\label{tab:det_res}Results for fast subsystem bifurcations. The additional hypotheses on the slow flow $\dot{y}=g(x,y)$ at $(x,y)=(0,0)$ are abbreviated and we always understand $g_j$ as $g_j(0,0)$ for $j=1,2$ and $g$ as $g(0,0)$ in this table. The Hopf-Hopf bifurcation has not been analyzed in detail and only a particular case is stated. The last column records the linearization around the attracting branch of the critical manifold.}
\end{table}

Having finished the exercise it is now clear which \emph{local} fast subsystem bifurcation points are critical transitions under suitable slow flow conditions. We record the results developed in Lemmas \ref{lem:fold}-\ref{lem:fH} as well as the resulting linearizations $D_xf|_{C_0\cap C^a}$ in Table \ref{tab:det_res} where we introduced the shorthand notation $A_0(y):=D_xf(h_0(y),y)=D_xf|_{C_0\cap C^a_0}$. 

Let us point out again that the classification results are for the singular limit $\epsilon=0$. Detailed unfoldings for the deterministic case for $\epsilon>0$ are known for the fold, pitchfork, transcritical and Hopf bifurcations \cite{KruSzm1,KruSzm3,KruSzm4,Neishtadt1}. Partial results are available for the Bogdanov-Takens bifurcation \cite{Chiba1} and the cusp \cite{BroerKaperKrupa} is work in progress. Section \ref{sec:conclusions} provides an overview where future work is needed.  

\section{Sample Paths and Moments for Stochastic Fast-Slow Systems}
\label{sec:SDE_fs}

Let $\{W_s\}_{s\geq 0}$ be a $k$-dimensional Brownian motion on a probability space $(\Omega,\mathcal{F},\P)$. Consider the \texttt{fast-slow 
stochastic differential equation (fast-slow SDE)}
\be
\label{eq:gen_SDE}
\begin{array}{lcl}
dx_s&=&\frac1\epsilon f(x_s,y_s)ds+\frac{\sigma}{\sqrt{\epsilon}}F(x_s,y_s)dW_s,\\
dy_s&=& g(x_s,y_s)ds.\\
\end{array}
\ee
which is understood as an It\^{o}-SDE \cite{Oksendal}. Noise acting on the slow variables $y$ will not be considered explicitly but it is implicitly included in all of our results as it appears as a higher-order term (\cite{BerglundGentz}, p.145; \cite{KuehnCT1}, p.1026). In addition to the assumptions (A0)-(A1) that hold for the deterministic part of \eqref{eq:gen_SDE} we require the following hypothesis:

\begin{itemize}
 \item[(A3)] $F\in C^2(\R^{m+n},\R^{m\times k})$ and the \texttt{noise level} $\sigma=\sigma(\epsilon)$ depends continuously on $\epsilon$.
 \item[(A4)] We consider \texttt{small noise} with $\lim_{\epsilon \ra 0}\sigma(\epsilon)=0$.
\end{itemize}

To understand the effect of a deterministic smooth invertible normal form transformation (coordinate change) $u(x,y)=(X,Y)\in\R^m\times \R^n$, with $u\in C^r(\R^m\times \R^n,\R^m\times \R^n)$ we need the following result which directly follows from It\^{o}'s formula (\cite{Oksendal}, p.44).

\begin{lem}
Consider the fast variable equation for \eqref{eq:gen_SDE}
\benn
dx_s=\frac1\epsilon f(x_s,y_s)ds+\frac{\sigma}{\sqrt{\epsilon}}F(x_s,y_s)dW_s
\eenn
then, using the notations $z_s=(x_s,y_s)$ and $Z_s=(X_s,Y_s)$, we have
\bea
\label{eq:stoch_normal}
dX^{(i)}_s&=&\left[ \frac{1}{\epsilon}\sum_{j=1}^{m} \frac{\partial u^{(i)}}{\partial x_j}f^{(j)}(u^{-1}(Z_s))+\sum_{j=1}^n \frac{\partial u^{(i)}}{\partial y_j} g^{(j)}(u^{-1}(Z_s))+\cO\left(\frac{\sigma^2}{\epsilon}\right)\right]ds+\frac{\sigma}{\sqrt\epsilon}F^{(i)}(u^{-1}(Z_s))dW_s\nonumber\\
&=:& \left[\frac{1}{\epsilon}\tilde{f}^{(i)}(X_s,Y_s)+\cO(1)+\cO\left(\frac{\sigma^2}{\epsilon}\right)\right]ds+\frac{\sigma}{\sqrt\epsilon}F^{(i)}(X_s,Y_s)dW_s
\eea
where superscripts $(i), (j)$ denotes the $i$-th resp. $j$-th row/component.
\end{lem}

Since $g(x_s,y_s)=\cO(1)$ is smaller than the first term, the only term that could be of leading order and obstruct the transformation to normal form for the deterministic part of \eqref{eq:gen_SDE} is of order $\cO(\sigma^2/\epsilon)$. By (A4) we have $\sigma^2(\epsilon)/\epsilon\ll 1/\epsilon$ which implies that the third term is also of higher-order after the normal form transformation in comparison to the deterministic $\cO(1/\epsilon)$-term. We now \emph{formally truncate} \eqref{eq:stoch_normal} by discarding the two terms of order lower than $\cO(1/\epsilon)$ as well as the polynomial terms appearing in $\tilde{f}^{(i)}(X_s,Y_s)$ which are of higher-order than the leading normal form terms ({e.g.}~for the fold $X^2-Y+\cO(Y^2,XY,X^3)+\cO(\epsilon)+\cO(\sigma^2/\epsilon)\approx X^2-Y$). On the basis of this formal truncation we now work with \eqref{eq:gen_SDE} where $f$ is chosen from the set of deterministic normal forms discussed in Section \ref{sec:nforms}. There are several interesting remarks regarding the formal truncation; see also Section 8.\\ 

\textit{Remark 1:}  In the deterministic case on the fast time scale, discarding higer-order polynomial and $\cO(\epsilon)$ terms is well understood for \emph{generic} fast subsystem bifurcations as shown {e.g.}~in (\cite{SzmolyanWechselberger1}, Proposition 2.1, Section 4.1; \cite{KruSzm3}, equation (2.5), section 2.4). Intuitively this is also clear since small perturbations do not change the unfolding of 1- or 2-parameter generic bifurcations for general smooth vector fields \cite{Wiggins}.\\

\textit{Remark 2:} For the deterministic ($\sigma=0$) pitchfork and transcritical bifurcations, which are not generic for general smooth vector fields, the $\cO(1)$-term in \eqref{eq:stoch_normal} is relevant as shown {e.g.}~in (\cite{KruSzm4}, Lemma 2.1, Theorem 2.1) in a region of the type (R2) near the singularity. The stochastic early-warning signs in a normally hyperbolic attracting region, such as (R1), are not expected to depend upon these terms (see the discussion of the attracting regime in \cite{BerglundGentz6}) but we do not verify this here and work with the formal truncation.\\

\textit{Remark 3:} It might be possible to weaken the assumption (A4) and to give a rigorous proof for a suitable 'equivalence' of a given SDE and its normal form. For the deterministic case, topological equivalence is known but for the stochastic case one needs different concepts such as \texttt{random normal form transformations} as suggested by Arnold and co-workers \cite{ArnoldRDS}.\\

Once the SDE \eqref{eq:gen_SDE} has been transformed into normal form we study \texttt{sample paths} $(x_s,y_s)$ that solve \eqref{eq:gen_SDE} as suggested in \cite{BerglundGentz}. By (A0) there exists a deterministic attracting slow manifold 
\benn
C^a_\epsilon=\{(x,y)\in\R^m\times \R^n:x=h_\epsilon(y)\}
\eenn
for $h_\epsilon:\cD_y\ra \cD_x$. The deviation of sample paths from this deterministic slow manifold is $\xi_s=x_s-h_\epsilon(y_s)$ and the variational SDE for $\xi_s$ is
\bea
\label{eq:deviate_SDE}
d\xi_s&=&dx_s-D_yh_\epsilon(y_s)~dy_s \\
&=& \frac1\epsilon\left[ f(h_\epsilon(y_s)+\xi_s,y_s)-\epsilon D_yh_\epsilon(y_s)~g(h_\epsilon(y_s)+\xi_s,y_s)\right]ds+\frac{\sigma}{\sqrt\epsilon}F(h_\epsilon(y_s)+\xi_s,y_s) dW_s.  \nonumber
\eea
Let $y_s^{\det}$ denote the deterministic solution of \eqref{eq:gen_SDE} (i.e. a solution for $\sigma=0$). For $\xi_s=0$ the drift term in \eqref{eq:deviate_SDE} satisfies the \texttt{invariance equation} \cite{ZagarisKaperKaper} for a slow manifold
\be
\label{eq:inv2}
f(h_\epsilon(y^{\det}_s),y^{\det}_s)-\epsilon D_yh_\epsilon(y^{\det}_s)~g(h_\epsilon(y^{\det}_s),y^{\det}_s)=0.
\ee
Linearizing \eqref{eq:deviate_SDE} around $\xi_s=0$ and replacing $y_s$ by $y_s^{\det}$ yields a lowest-order system for the process $(\xi^l_s,y_s)$ given by
\be
\label{eq:lin_SDE}
\begin{array}{lcl}
d\xi^l_s &=& \frac1\epsilon [D_xf(h_\epsilon(y^{\det}_s),y^{\det}_s)-\epsilon D_yh_\epsilon(y^{\det}_s)~D_xg(h_\epsilon(y^{\det}_s),y^{\det}_s)]\xi^l_s ds\\
&&+\frac{\sigma}{\sqrt\epsilon} F(h_\epsilon(y^{\det}_s),y^{\det}_s)dW_s,\\
dy^{\det}_s&=&g(h_\epsilon(y^{\det}_s),y^{\det}_s)ds.
\end{array}
\ee 
For notational simplicity we let 
\bea
A_\epsilon(y^{\det}_s)&:=&D_xf(h_\epsilon(y^{\det}_s),y^{\det}_s)-\epsilon D_yh_\epsilon(y^{\det}_s)~D_xg(h_\epsilon(y^{\det}_s),y^{\det}_s)\label{eq:Aeps_def},\\
F_\epsilon(y^{\det}_s)&:=&F(h_\epsilon(y^{\det}_s),y^{\det}_s)\label{eq:Feps_def}.
\eea
Note carefully that for $\epsilon=0$ we get the matrix $A_0(y)=D_xf(h_0(y),y)$ which is precisely the linearization recorded in Table \ref{tab:det_res}. We shall always assume that initial conditions for \eqref{eq:lin_SDE} are deterministic and given by $(\xi^l_0,y_0)=(0,y_0)$ which corresponds to starting on the deterministic slow manifold. Now we can state an important result about the covariance $\text{Cov}(\xi^l_s)$ of the linearized process.

\begin{lem}[\cite{BerglundGentz}, p.146-147]
\label{lem:BG_lem}
Let $X_s:=\sigma^{-2}\text{Cov}(\xi^l_s)$ then $X_s$ satisfies a fast-slow ODE
\be
\label{eq:lin_fs_ODE}
\begin{array}{rcl}
 \epsilon \dot{X}&=&A_\epsilon(y)X+XA_\epsilon(y)^T+F_\epsilon(y)F_\epsilon(y)^T,\\
 \dot{y}&=& g(h_\epsilon(y),y).
\end{array}
\ee
Furthermore, for $0<\epsilon\ll 1$, the critical manifold for \eqref{eq:lin_fs_ODE} is attracting for $y\in \cD_y$ and given by
\benn
\cC_0=\{X\in \R^{m\times m}:A_0(y)X+XA_0(y)^T+F_0(y)F_0(y)^T=0\}.
\eenn
\end{lem}

Fenichel's Theorem provides an associated slow manifold $\cC_\epsilon=\{X=H_\epsilon(y)\}$ for $H_\epsilon:\R^n\ra \R^{m\times m}$. Assuming that the matrix $H_\epsilon(y)$ is invertible and that the operator norm $\|H_\epsilon^{-1}(y)\|$ is uniformly bounded for $y\in \cD_y$ one can define the \texttt{covariance neighborhood}
\benn
\cB(r):=\left\{(x,y)\in\cD:[x-h_\epsilon(y)]^T\cdot H_\epsilon(y)^{-1}[x-h_\epsilon(y)]<r^2\right\}.
\eenn
Define the \texttt{first-exit time} of the original process $(x_s,y_s)$, starting at $s=s_0$, from a set $\cA$ as
\benn
\tau_\cA:=\inf\{s\in[s_0,\infty):(x_s,y_s)\notin \mathcal{A},(x_0,y_0)\in \mathcal{A})\}
\eenn
where $\cA$ is chosen so that $\tau_\cA$ is a stopping time wrt the filtration generated by $\{(x_s,y_s)\}_{s\geq s_0}$.

\begin{thm}[\cite{BerglundGentz1}, p.149-150]
\label{thm:BG}
Sample paths stay inside $\cB(r)$ with high probability. More precisely, there exists $K(s,\epsilon,\sigma)$ and $\kappa>0$ such that $\P\left\{\tau_{\cB(r)}<\min(s,\tau_{\cD_y})\right\}\leq K(s,\epsilon,\sigma)e^{-\kappa r^2/(2\sigma^2)}$, where the pre-factor $K(s,\epsilon,\sigma)$ grows at most polynomially in its arguments as $(\epsilon,\sigma)\ra (0,0)$ and $s\ra \I$. 
\end{thm}

The main conclusion of Theorem \ref{thm:BG} is that sample paths near normally hyperbolic attracting slow manifolds are \texttt{metastable} i.e. they stay near the manifold for exponentially long times except when the slow dynamics moves the system near a fast subsystem bifurcation point so that the stopping time $\tau_{\cD_y}$ is reached. Theorem \ref{thm:BG} does not immediately guarantee that we can use moments from the linearized process $\xi^l_s$ to approximate the moments of the nonlinear process $\xi_s$. For a approach to this problem re-consider the general fast-slow SDE \eqref{eq:gen_SDE}. The associated slow flow ODE is $dy^0_s=g(h_0(y^0_s),y^0_s)ds$. Define $x^0_s:=h_0(y^0_s)$ and observe that the solutions $(x_s,y_s)$ of \eqref{eq:gen_SDE} depend implicitly on $\epsilon$. A complementary result to Theorem \ref{thm:BG} by Kabanov and Pergamenshchikov provides a convenient convergence in probability of the process $(x_s,y_s)$ to $(x^0_s,y^0_s)$ as $\epsilon \ra 0$.

\begin{thm}[\cite{KabanovPergamenshchikov}, p.45-46]
\label{thm:KP}
Suppose (A0)-(A2) and (A4) hold. We start at $s=s_0$ and consider a final time $S>0$ such that $(x_s,y_s)$ has not left $\cD$. Then for any $s\in[s_0,S]$ 
\benn
\sup_{0\leq s\leq S}|x_s-x^0_s|\stackrel{\P}{\ra} 0\qquad \text{and}\qquad \sup_{0\leq s\leq S}|y_s-y^0_s|\stackrel{\P}{\ra} 0
\eenn
as $\epsilon\ra 0$ where $\stackrel{\P}{\ra}$ indicates convergence in probability.
\end{thm}

As a direct corollary to this result the linearized process $\xi^l_s$ also approximates $\xi_s$ in probability as both processes tend to the same deterministic limit as $\epsilon\ra 0$.

\begin{prop}
\label{prop:conv}
Under the assumptions (A0)-(A2) and (A4) we have $\sup_{0\leq s\leq S}|\xi_s-\xi^l_s|\stackrel{\P}{\ra} 0$, as $\epsilon \ra 0$. In particular, we have convergence in distribution $\xi_s \stackrel{d}{\ra} \xi^l_s$ as $\epsilon \ra 0$.
\end{prop}

\begin{proof}
Observe that Theorem \eqref{thm:KP} can also be applied to the processes $\xi_s$ and $\xi_s^l$ instead of $x_s$ with $\tilde{h}_0(y^0_s)=\xi^0_s\equiv 0$. This yields 
\benn
\sup_{0\leq s\leq S}|\xi_s-\xi^l_s|=\sup_{0\leq s\leq S}|\xi_s-\xi^0_s+\xi^0_s-\xi^l_s|\leq \sup_{0\leq s\leq S}|\xi_s-0|+\sup_{0\leq s\leq S}|\xi^l_s-0| \stackrel{\P}{\ra} 0.\qedhere
\eenn
\end{proof}

Proposition \ref{prop:conv} only states that the two stochastic processes converge to the same deterministic process as $\epsilon\ra 0$. However, for a metastable approximation one must check how the $k$-th moment approximation depends on the \textit{time} $s$ and the time scale separation $\epsilon$. In particular, we are interested in the first and second moments and let
\benn
\delta_1(s,\epsilon):=\E[\xi_s]-\E[\xi_s^l],\qquad \delta_2(s,\epsilon):=\text{Cov}(\xi_s)-\text{Cov}(\xi_s^l).
\eenn
It is certainly possible to adapt previous results such as the work by Berglund and Gentz \cite{BerglundGentz} to achieve explicit moment bounds. However, the techniques are rather complicated and based upon martingale methods, Bernstein-type inequalities, subdivison of suitable time intervals and calculating new explicit moment bounds and aim to control probabilities \emph{path-wise}. Here we are going to develop a very short and essentially 'algorithmic' argument for moment bounds for truncated normal forms. The technique is elementary and only uses a suitable difference process, well-known even-moment bounds and the Cauchy-Schwartz inequality; this approach may even have the potential to simplify calculations for path-wise control such as (\cite{BerglundGentz8}, Section 4). 

We shall only discuss moment approximation for the fold bifurcation which provides an outline how moments can be controlled in the general case. The simplest normal form model for a fold bifurcation with additive noise is
\be
\label{eq:fold_SDE}
\begin{array}{rcl}
dx_s&=&\frac{1}{\epsilon}(-y_s-x_s^2)ds + \frac{\sigma}{\sqrt\epsilon} dW_s,\\
dy_s&=&1 ~ds,\\
\end{array}
\ee
where we can also view $y_s=(s-s_0)+y_{s_0}$ as a time variable. The attracting critical manifold is $C^a_0=\{(x,y)\in\R^2:x=\sqrt{-y}=h_0(y)\}$ with an associated slow manifold $C^a_\epsilon=\{x=h_\epsilon(y)=h_0(y)+\cO(\epsilon)\}$. Note that for \eqref{eq:fold_SDE} we have $y_s=y^{\det}_s$. Therefore we get that \eqref{eq:deviate_SDE} is given by
\be
\label{eq:fold_SDE1}
\begin{array}{rcl}
d\xi_s&=&\frac{1}{\epsilon}(-2\sqrt{-y_s}\xi_s-\xi_s^2+\cO(\epsilon))ds +  \frac{\sigma}{\sqrt\epsilon} dW_s,\\
dy_s&=&1 ~ds.\\
\end{array}
\ee
where we are going to formally drop the higher-order term $\cO(\epsilon)$-term from now on. The linearized problem \eqref{eq:lin_SDE} is 
\be
\label{eq:fold_SDE2}
\begin{array}{rcl}
d\xi^l_s&=&\frac{1}{\epsilon}(-2\sqrt{-y_s})\xi_s^l ds +  \frac{\sigma}{\sqrt\epsilon} dW_s,\\
dy_s&=&1 ~ds.\\
\end{array}
\ee
To analyze the transient behavior we consider the \texttt{difference process} $v_s:=\xi_s-\xi^l_s$. It satisfies the differential equation
\be
\label{eq:diff_process}
\begin{array}{lcl}
dv_s&=&\frac{1}{\epsilon}[-y_s-h_\epsilon(y_s)^2-2h_\epsilon(y_s)v_s -\xi_s^2-\epsilon D_yh_\epsilon(y_s)]ds\\
&=&\frac{1}{\epsilon}[-2\sqrt{-y_s}v_s -\xi_s^2+\cO(\epsilon)]ds.\\
\end{array}
\ee
where the $\cO(\epsilon)$-term will again be dropped. We always consider a initial condition $y_0$ at time $s_0=0$ such that $y_s=s+y_0$ and $y_{s}<0$ for $s\in[0,s^*]$ for some $s^*>0$ such that $y_s$ remains in the compact region $\cD$.

\begin{lem} 
The expected value of the difference process satisfies the ODE
\be
\label{eq:moment_equation1}
\frac{d}{ds}\E[v_s]=\frac{1}{\epsilon}\left(-2\sqrt{-y_s}~\E[v_s]-\E[\xi_s^2]\right).
\ee
\end{lem}

\begin{proof}
Substract \eqref{eq:fold_SDE2} from \eqref{eq:fold_SDE1} and take the expectation.
\end{proof}

By the variation of constants formula (\cite{Hale}, p.82) the solution of \eqref{eq:moment_equation1} is given by
\be
\label{eq:fold_ex_fund}
\E[v_s]=\E[v_0]~X(s,s_0)-\int_{s_0}^s \frac{\E[\xi_r^2]}{\epsilon}X(s,r) dr  
\ee
where $X(s,r)$ is the fundamental solution of $\frac{d}{ds}\E[v_s]=\frac1\epsilon (-2\sqrt{-y_s} ~\E[v_s])$. If we can show that $\E[\xi_s^2]$ is ``small'' then \eqref{eq:fold_ex_fund} provides a way to show that the mean of $v_s$ remains small as well.

\begin{lem}[\cite{KabanovPergamenshchikov}, p.20-25]
\label{lem:moment_estimates}
Suppose the stochastic differential equation $dX_s=\alpha(X_s,s)ds+\beta(s)dW_s$ with $X\in\R^m$ and $\beta(s)\in \R^{m\times k}$ satisfies for $s\in[s_1,s_2]$ the stability condition $X^T\alpha(X,s)\leq -\kappa \|X\|^2$ and has uniformly bounded noise term $\sup_{s\in[s_1,s_2]}\|\beta(s)\|\leq M$ then 
\be
\label{eq:KP_mom}
\E[X_s^{2p}]\leq p!\left(\frac{M^2}{\kappa}\right)^p.
\ee
\end{lem}

Applying Lemma \ref{lem:moment_estimates} to $\xi_s=X_s$ and equation \eqref{eq:fold_SDE1} we see that $\kappa=\cO(1/\epsilon)$ and $M=\sqrt\sigma/\epsilon$. Therefore using \eqref{eq:KP_mom} with $p=1$ yields 
\be
\label{eq:variance_xis}
\frac{\E[\xi_s^2]}{\epsilon}\leq \frac{\sigma^2}{\epsilon^2}\cO(\epsilon)= \cO\left(\frac{\sigma^2}{\epsilon}\right)=\cO\left(\frac{\sigma(\epsilon)^2}{\epsilon}\right)
\ee
where $s\in[0,s^*]$ to assure normal hyperbolicity with $\kappa=\cO(1/\epsilon)$. Using the estimate \eqref{eq:variance_xis} in \eqref{eq:fold_ex_fund} and assuming that $\E[v_0]=\E[\xi_0-\xi_0^l]=0$ we get the inequality
\be
\label{eq:fold_ex_fund1}
|\E[v_s]|\leq \int_{0}^s \left|\cO\left(\frac{\sigma(\epsilon)^2}{\epsilon}\right)X(s,r) \right|dr =\delta_1(s,\epsilon). 
\ee
In particular, we can use the linearized process to approximate the mean 
\benn
|\E[v_s]|=|\E[\xi_s]-E[\xi_s^l]|\leq \delta_1(s,\epsilon).
\eenn
Next, we define $V_s:=\text{Var}(\xi_s)-\text{Var}(\xi_s^l)$. 

\begin{lem}
\label{lem:var_est_final}
The difference process $V_s$ for the variance satisfies the ODE
\be
\label{eq:moment_equation2}
\frac{d}{ds}V_s=\frac2\epsilon \left( -2\sqrt{-y_s}~V_s   -\E[\xi_s^2 \xi_s] +\E[\xi_s^2]\E[\xi_s]\right).
\ee
\end{lem}

\begin{proof}
A direct calculation using It\^{o}'s formula (\cite{Socha}, p.87) shows that
\be
\begin{array}{lcl}
\frac{d}{ds}\text{Var}(\xi_s)&=&2\E\left[\frac1\epsilon \left(-2\sqrt{-y_s}\xi_s-\xi_s^2\right)(\xi_s-\E[\xi_s])\right]+\frac{\sigma^2}{\epsilon},\\
\frac{d}{ds}\text{Var}(\xi^l_s)&=&2\E\left[\frac1\epsilon \left(-2\sqrt{-y_s}\xi^l_s\right)(\xi^l_s-\E[\xi_s^l])\right]+\frac{\sigma^2}{\epsilon}.\\
\end{array}
\ee
Then using $\text{Var}(\xi_s)=\E[\xi_s^2]-\E[\xi_s]^2$ and $\text{Var}(\xi_s^l)=\E[(\xi^l_s)^2]-\E[\xi^l_s]^2$ gives \eqref{eq:moment_equation2}.
\end{proof}

\begin{lem} 
\label{lem:third_moment}
$|\E[\xi_s^2\xi_s^l]|\leq \cO(\sigma^3)$ and $|\E[\xi_s^2]\E[\xi_s^l]|\leq \cO(\sigma^3)$. 
\end{lem}

\begin{proof}
By a combination of Lemma \eqref{lem:moment_estimates} and the Cauchy-Schwarz inequality its follows that
\benn
\E[\xi_s^2\xi_s^l]\leq \E[\xi_s^4]^{1/2}\E[(\xi_s^l)^2]^{1/2}\leq \cO(\sigma^2)\cO(\sigma)=\cO(\sigma^3).
\eenn
The second results is proven similarly.
\end{proof}

As in the derivation of the bound \eqref{eq:fold_ex_fund1} we now use Lemma \ref{lem:var_est_final} and Lemma \ref{lem:third_moment} to conclude that
\be
\label{eq:fold_ex_fund2}
|Var(\xi_s)-Var(\xi_s^l)|=|V_s|\leq \int_{0}^s \left|\cO\left(\frac{\sigma(\epsilon)^3}{\epsilon}\right)\tilde{X}(s,r) \right|dr =\delta_2(s,\epsilon). 
\ee
where $\tilde{X}(s,r)$ is the fundamental solution of $\frac{d}{ds}\E[V_s]=\frac1\epsilon (-4\sqrt{-y_s} ~\E[V_s])$. The estimates \eqref{eq:fold_ex_fund1} and \eqref{eq:fold_ex_fund2} require the fundamental solutions of systems of the form
\be
\label{eq:fund_sol_1D}
\frac{d}{ds}w(s,r)=-\frac{\kappa}{\epsilon}\sqrt{-s-y_0}~w(s,r),\quad w(r,r)=1\quad \Rightarrow \quad w(s,r)=e^{\frac{2\kappa}{3\epsilon}\left[(-s-y_0)^{3/2}-(-r-y_0)^{3/2}\right]}.
\ee
where $\kappa>0$ is a constant; here $\kappa=2,4$ for the first and second moment estimates. We remark that in \eqref{eq:fund_sol_1D} the formal condition $(-s-y_0)^{3/2}\sim \epsilon$ with $s=0$ yields the critical scaling $y\sim \epsilon^{2/3}$ as expected from the loss of normal hyperbolicity near a fold (\cite{KruSzm3}, p.291; \cite{BerglundGentz}, p.87). To estimate $\delta_1(s,\epsilon)$ and $\delta_2(s,\epsilon)$ one must consider the integral
\be
\label{eq:Laplace_integral}
\int_0^s e^{\varphi(r)/\epsilon}dr \quad\text{ with }\varphi(r):=\frac{2\kappa}{3}\left[(-s-y_0)^{3/2}-(-r-y_0)^{3/2}\right].
\ee
Note carefully that \eqref{eq:Laplace_integral} has asymptotics that can be determined via Laplace's method (see \cite{BenderOrszag}, p.265-267). If no formal truncation, {e.g.}~in \eqref{eq:fold_SDE1}, is applied there are much more detailed results available in \cite{BerglundGentz8} using explicit calculations where Laplace-type integrals still appear \cite{BerglundGentz}. However, it seems that the ideas used here utilizing the difference process, the direct moment estimates from Lemma \ref{lem:moment_estimates} and the Cauchy-Schwarz inequality are a simple, and essentially algorithmic, shortcut to lead to a Laplace-type integral.

\begin{prop}
\label{prop:Laplace}
Suppose $(-s+y_0)=\cO(\epsilon^{2\alpha})$ with $\alpha<1/3$ then $\int_0^s e^{\varphi(r)/\epsilon}dr\sim\epsilon^{1-\alpha}$ as $\epsilon\ra 0$. 
\end{prop}

\begin{proof}
One calculates that $\varphi'(r)>0$ for $r\in[0,s]$ if $s<s^*$. Then applying the standard asymptotic Laplace approximation at the endpoint $s$ (see \cite{BenderOrszag}, p.266, (6.4.19b)) yields
\benn
\int_0^s e^{\varphi(r)/\epsilon}dr\sim \epsilon\frac{e^{\varphi(s)/\epsilon}}{\varphi'(s)}=\epsilon\frac{1}{\kappa(-s-y_0)^{1/2}}=\epsilon^{1-\alpha},\qquad \text{as $\epsilon\ra 0$.}\qedhere
\eenn 
\end{proof}

Hence we obtain from \eqref{eq:fold_ex_fund1}, \eqref{eq:fold_ex_fund2} and Proposition \ref{prop:Laplace} that in the normally hyperbolic regime $\cD$ with $y=\cO(\epsilon^{2\alpha})$ and $\alpha<1/3$ the moment estimates are 
\benn
\delta_1(s,\epsilon)=\cO(\sigma^2\epsilon^{-\alpha}) \qquad \text{and} \qquad \delta_2(s,\epsilon)=\cO(\sigma^3\epsilon^{-\alpha}).
\eenn
Lemma \ref{lem:BG_lem} gives for the fold bifurcation the desired moment approximation for the linearized process $\text{Var}(\xi^l_s)=\sigma^2 [H_\epsilon(y)]$ so that the approximation result for the variance is
\be
\label{eq:var_fold_final}
\text{Var}(\xi_s)=\sigma(\epsilon)^2 [H_\epsilon(y)]+\cO\left(\sigma(\epsilon)^3\epsilon^{-\alpha}\right)\qquad \text{as $\epsilon\ra 0$.}
\ee
For $\alpha=0$ the process is at an $\cO(1)$-distance from the critical transition point at the fold and $\text{Var}(\xi_s)=\sigma(\epsilon)^2 [H_\epsilon(y)]+\cO\left(\sigma(\epsilon)^3\right)$. As expected, the estimate of variance becomes less accurate the closer sample paths move towards $(x_p,y_p)=(0,0)$. For $\alpha>0$, the error term in formula  \eqref{eq:var_fold_final} is asymptotic if and only if $\sigma^2\gg \sigma^3\epsilon^{-\alpha}$ or $\epsilon^\alpha\gg \sigma$. 

Note that if $\epsilon^{k_0+\alpha}H_{k_0}(y)=\cO( \sigma)$ for all $k\geq k_0>0$ then
\benn
\text{Var}(\xi_s)=\sigma^2 \left[H_0(y)+\sum_{k=1}^{k_0-1}H_k(y)\epsilon^k\right]+\cO\left(\sigma^3\epsilon^{-\alpha}\right)
\eenn 
since we can absorb the correction terms for the slow manifold of the variance into $\cO(\sigma^3\epsilon^{-\alpha})$. In particular, if $k_0=1$ then it follows that 
\be
\label{eq:move_everything}
\text{Var}(\xi_s)=\sigma^2 H_0(y)+\cO\left(\sigma^3\epsilon^{-\alpha}\right).
\ee
In principle, we could also calculate higher-order corrections to the slow manifold defined by $X=H_\epsilon(Y)$; see (\cite{BerglundGentz}, p. 147) and Section \ref{sec:fold_asymp}. For simplicity, we shall only consider the lowest-order approximation for a general codimension-two fast subsystem bifurcation.\\

For another fast subsystem bifurcation, we will get another approximation of the moments as we used $\sqrt{-y}=x$ for the slow manifold in the fold scenario. However, we still expect that  
\be
\label{eq:estimate_var_general}
\text{Cov}(\xi_s)=\sigma^2 [H_\epsilon(y)]+\delta_2(\epsilon,s).
\ee 
where $H_\epsilon(y)=\sum_{k=0}^\I \epsilon^kH_k(y)$ and $\delta_2(\epsilon,s)$ denotes a small $\epsilon$-dependent error term for the second moments. In fact, the methods we use here based upon moment equations, integral estimates and direct asymptotics all generalize to higher-dimensional phase space and higher-codimension bifurcations. Hence we conjecture that \eqref{eq:estimate_var_general} is still valid for these cases. Although we do not calculate the asymptotic relations here, our approach provides a direct computational method for the relevant scalings.

It is very important to recall the result is still only local around the attracting slow manifold in a compact set $\cD=\cD(\epsilon)$. Although $(x_p,y_p)\in\partial \cD(0)$ one always has to use the moment approximations by a linearized process outside of a $(\epsilon,\sigma(\epsilon))$-dependent neighbourhood of the critical transition point $(x_p,y_p)$. Small $(\epsilon,\sigma(\epsilon))$-dependent neighbourhoods (R2) near the bifurcation point have to be considered separately \cite{BerglundGentzKuehn,Kuske,BerglundGentz6}. For early-warning signs it is very reasonable to ask for the earliest possible statistical indicators. Once a sample path reaches (R2) it is extremely close to a fast jump so a warning sign may be difficult to utilize in applications.   

\section{Covariance Scaling Laws near Critical Transitions}
\label{sec:variance}

To calculate $H_0(y)$ we have to solve the algebraic equation
\be
\label{eq:ma1}
0=A_0(y)X+XA_0(y)^T+F_0(y)F_0(y)^T.
\ee
where the matrices $A_0(y)$ are chosen according to normal form theory from Table \ref{tab:det_res} (see also \eqref{eq:Aeps_def}-\eqref{eq:Feps_def} for definitions). It will be convenient to introduce a notation for the symmetric matrix $F_0(y)F_0(y)^T$ that describes the noise term
\be
\label{eq:defineN}
(N_{ij}(y))=N(y):=F_0(y)F_0(y)^T
\ee
for $i,j\in\{1,2,\ldots,m\}$. If $N(y)$ is a constant matrix then we deal with purely additive noise while dependence on $y$ indicates multiplicative (or slowly parameter-dependent) noise. To distinguish between the small noise asymptotics 
\benn
\epsilon\ra 0\qquad  \Rightarrow \quad \sigma=\sigma(\epsilon)\ra 0
\eenn  
and the approach towards the fast subsystem bifurcation point $y$ tending to the origin we use the order notation $\cO^*_y$ for $y\ra 0$. Recall that (A0) specifies what type of double asymptotics we allow and that all results are constrained to a bounded domain {i.e.} a result $w(y)=\cO_y^*(W(y))$ is to be understood as, for a given sufficiently small $\epsilon>0$, and hence a given $\sigma(\epsilon)>0$, there exists a compact non-empty domain $\cD_y(\epsilon)\subset \R^n$ with $0\in\partial \cD(0)$ but $0\not\in\cD(\epsilon)$ ({cf.}~(R1) in Figure \ref{fig:fig1}) and constants $K_i$, $i=1,2$ such that
\benn
K_1W(y)\leq w(y)\leq K_2W(y)
\eenn
for all $y\in \cD_y(\epsilon)$. In particular, the early-warning signs we are going to derive are for fixed $(\epsilon,\sigma(\epsilon))$ sufficiently small, a fixed domain $\cD(\epsilon)=\cD_x(\epsilon)\times \cD_y(\epsilon)$ chosen around a slow manifold so that the approximation is good as $y$ tends to the transition point but will eventually break down in a small region near the critical transition that scales with $\epsilon$ and $\sigma$ and includes the critical transition point in its boundary for $\epsilon=0=\sigma$. Small $(\epsilon,\sigma)$-dependent regions containing a critical transition point require a special analysis and will not be considered; see the remarks on additional literature in Section 8. 

Furthermore, we agree to the convention that any limit as $y\ra 0$ is always understood as the natural one-sided limit if necessary e.g. $\cO_y^*(\sqrt{-y})$ means $y\ra 0^-$.

\begin{thm}
\label{thm:CK1}
Suppose $0<\epsilon\ll1$ and (A0)-(A4) hold for a fast subsystem bifurcation with one fast variable (fold, transcritical, pitchfork, cusp) and $\epsilon>0$ is sufficiently small. Then the variance of the process $\xi_s$ near an attracting slow manifold approaching the bifurcation satisfies
\benn
\text{Var}(\xi_s)=\sigma^2 [H_\epsilon(y)]+\delta_2(s,\epsilon).
\eenn 
where $H_\epsilon(y)=H_0(y)+\cO(\epsilon)$ and 
\begin{enumerate}
 \item[(V1)] (fold) $H_0(y)=\cO_y^*\left(\frac{N(y)}{\sqrt{y}}\right)$,
 \item[(V2)] (transcritical, pitchfork) $H_0(y)=\cO_y^*\left(\frac{N(y)}{y}\right)$,
 \item[(V3)] (cusp) $H_0(y)=\cO_{y_2}^*\left(\frac{N(y)}{y_2}\right)$; where the slow variable $y_2$ multiplies the linear term in the fast subsystem normal form \eqref{eq:nf_cusp}. 
\end{enumerate}
In particular, if $\delta_2(s,\epsilon)\ll \sigma^2$ and $N(y)$ is constant then the variance scales, to lowest order, as $\sigma^2/\sqrt{y}$ for the fold, as $\sigma^2/y$ for the transcritical/pitchfork and as $\sigma^2/y_2$ for the cusp transition.
\end{thm}

\begin{proof}
We can approximate the variance of the process $\xi_s$ by its linearization $\xi_s^l$ if $\epsilon$ is sufficiently small. The linearized process has variance
\be
\label{eq:var_asymp_easy1}
\text{Var}(\xi^l_s)=\sigma^2 (H_0(y)+\cO(\epsilon))+\delta_2(s,\epsilon).
\ee
where $X=H_0(y)\in \R^+$ is the solution of
\be
\label{eq:proof_1D}
0=2A_0(y)X+N(y),\qquad \Rightarrow \quad X=H_0(y)=-\frac{N(y)}{2A_0(y)}.
\ee
The non-degeneracy assumptions of the four bifurcations considered are satisfied. By using normal forms, we know from Table \ref{tab:det_res} that $A_0(y)=\cO_y^*(\sqrt{y})$ for the fold transition, $A_0(y)=\cO_y^*(y)$ for the transcritical and pitchfork transitions while $A_0(y)=\cO_{y_2}^*(y_2)$ for the cusp transition. Direct substitution of these results for $A_0(y)$ into \eqref{eq:proof_1D} gives the result.
\end{proof}

The codimension-one fold and the transcritical/pitchfork case in Theorem \ref{thm:CK1} can also be inferred from previous works see {e.g.}~\cite{BerglundGentz}. In fact, rigorous proofs without formal truncation are available. In these results the higher-order terms do not seem to influence the scaling law in region (R1); this is one of the motivations to consider a formal truncation. The stochastic cusp, and all the following codimension-two results, have not been considered previously. It should be noted that for fast subsystems with dimension greater than one the stochastic scaling effects are much more interesting as the next result shows. 

\begin{thm}
\label{thm:CK2}
Suppose $0<\epsilon\ll1$ and (A0)-(A4) hold for a fast subsystem bifurcation with two fast variables (Hopf, Bogdanov-Takens, Bautin) and $\epsilon>0$ is sufficiently small. Then the covariance matrix of the process $\xi_s$ near an attracting slow manifold approaching the bifurcation satisfies
\benn
\text{Cov}(\xi_s)=\sigma^2 [H_\epsilon(y)]+\delta_2(s,\epsilon).
\eenn 
where $H_\epsilon(y)=H_0(y)+\cO(\epsilon)$ and 
\begin{enumerate}
 \item[(V4)] (Hopf, Bautin)
 \benn
  H_0(y)=\left(\begin{array}{cc}
 -\frac{2N_{11}(y)y^2+2N_{12}(y)y+N_{11}(y)+N_{22}(y)}{4y(y^2+1)} & \frac{N_{11}(y)-N_{22}(y)-2N_{12}(y)y}{4(y^2+1)} \\
\frac{N_{11}(y)-N_{22}(y)-2N_{12}(y)y}{4(y^2+1)} & -\frac{2N_{22}(y)y^2-2N_{12}(y)y+N_{11}(y)+N_{22}(y)}{4y(y^2+1)}\\
\end{array} 
\right).
 \eenn
In particular, $N$ is a constant matrix with $N_{11}+N_{22} \neq 0$ then 
\benn
 H_0(y)=\left(\begin{array}{cc}
 \cO_y^*\left(\frac1y\right) & \frac{N_{11}-N_{22}}{4} +\cO_y^*(y)\\
\frac{N_{11}-N_{22}}{4} +\cO_y^*(y) & \cO_y^*\left(\frac1y\right) \\
\end{array} 
\right).
 \eenn
 \item[(V5)] (Bogdanov-Takens; we set $\cO_y^*(\sqrt{-y_1})=k\sqrt{-y_1}$ for $A_0(y)$) 
\benn
 H_0(y)=\left(\begin{array}{cc}
  \frac{-N_{22}(y)+2kN_{12}(y)\sqrt{-y_1}\pm 2N_{11}(y)\sqrt{-y_1}+N_{11}k^2y_1}{\pm 4ky_1} & -\frac{N_{11}(y)}{2} \\
-\frac{N_{11}(y)}{2} & \frac{\pm 2N_{11}(y)+N_{22}(y)y_1/(-y_1)^{3/2}}{2k} \\
\end{array} 
\right).
\eenn
In particular, if $N$ is a constant matrix and $N_{22}\neq 0$ then
\benn
 H_0(y)=\left(\begin{array}{cc}
 \cO_y^*\left(\frac{1}{y_1}\right) & -\frac{N_{11}}{2} \\
-\frac{N_{11}}{2} & \pm \frac{N_{11}}{k}+\cO_y^*\left(\frac{1}{\sqrt{-y_1}}\right) \\
\end{array} 
\right).
\eenn
\end{enumerate}
\end{thm}

\begin{proof}
The proof follows the same outline as the proof of Theorem \ref{thm:CK1}. Therefore we shall only detail the calculations for the proof of (V4). We can again reduce to the linearized process and apply the formula
\be
\label{eq:var_asymp_easy2}
\text{Cov}(\xi_s)=\sigma^2 (H_0(y)+\cO(\epsilon))+\delta_2(s,\epsilon).
\ee
Denote the elements of the scaled covariance matrix as follows
\benn
X=\sigma^{-2}\text{Cov}(\xi^l_s):=\left(\begin{array}{cc}v_{11} & v_{12} \\ v_{12} & v_{22} \\\end{array}\right).
\eenn
Using the normal form matrix $A_0(y)$ from Table \ref{tab:det_res} we calculate
\bea
\label{eq:Hopf_p}
0&\stackrel{!}{=}&A_0(y)X+XA_0(y)^T+N(y) \nonumber\\
&=&\left( \begin{array}{cc} -v_{12}+v_{11}y & v_{12}y-v_{22} \\ v_{11}+v_{12}y & v_{12}+v_{22}y\\\end{array} \right)+
\left( \begin{array}{cc} -v_{12}+v_{11}y & v_{11}+v_{12}y \\ v_{12}y-v_{22} & v_{12}+v_{22}y\\\end{array} \right)+
\left( \begin{array}{cc} N_{11}(y) & N_{12}(y) \\ N_{12}(y) & N_{22}(y)\\\end{array} \right) \nonumber\\
&=& \left(\begin{array}{cc} N_{11}(y)-2v_{12}+2v_{11}y & N_{12}(y)+v_{11}-v_{22}+2v_{12}y \\ N_{12}(y)+v_{11}-v_{22}+2v_{12}y & N_{22}(y)+2v_{12}+2v_{22}y\end{array}\right).
\eea
Equation \eqref{eq:Hopf_p} yields three independent conditions. Hence we get a linear system 
\benn
\left(\begin{array}{ccc}2y & 0 & -2\\ 0 & 2y & 2 \\ 1 & -1 & 2y \\\end{array}\right)
\left(\begin{array}{c}v_{11} \\ v_{22} \\ v_{12}\\\end{array}\right)
=
\left(\begin{array}{c}-N_{11}(y) \\ -N_{22}(y) \\ -N_{12}(y)\\\end{array}\right)
\eenn
that can be solved for $(v_{11},v_{22},v_{12})$. This result can be substituted as $X=H_0(y)$ in \eqref{eq:var_asymp_easy2} and this yields the first part of (V4). If $N$ is a constant matrix then direct asymptotics shows that 
\benn
v_{11}\sim-\frac{N_{11}+N_{22}}{4y}=\cO_y^*\left(\frac{1}{y}\right), \qquad v_{22}\sim-\frac{N_{11}+N_{22}}{4y}=\cO_y^*\left(\frac{1}{y}\right)
\eenn
where $y\ra 0^-$ as we approach the critical transition via the attracting slow manifold. For the covariance we get
\benn
v_{12}\sim \frac{N_{11}-N_{22}}{4}-\frac{2N_{12}}{4}y=\frac{N_{11}-N_{22}}{4}+\cO_y^*\left(y\right).
\eenn
The result for the Bogdanov-Takens transition follows by the same techniques.
\end{proof}

Before we continue to codimension two bifurcations in $\R^3$ and $\R^4$, let us interpret the results of Theorem \ref{thm:CK2} for $\delta_2(s,\epsilon)\ll \sigma^2$. For the Hopf transition with a fixed noise level $\sigma>0$ we have found that the variance of the coordinates increases as $\cO_y^*(1/y)$ as the bifurcation point is approach with $y\ra 0$. This result is expected as we already saw an increase in variance for the one-dimensional fast subsystem bifurcations. However, for the covariance the additive noise case with a constant matrix $N$ yields
\benn
\frac{N_{11}-N_{22}}{4} +\cO_y^*(y).
\eenn
This implies that the covariance tends to a constant as $y\ra 0$; even more surprisingly, for the reasonable assumption of equal individual diffusion $N_{11}=N_{22}$ we get that the covariance tends to zero as the bifurcation is approached. Hence we can already conclude that measuring covariances can also provide important information to predict critical transitions. For the Hopf transition with multiplicative noise, let us just consider the simplest case of linear multiplicative noise without correlation $N_{11}(y)=c_1y$, $N_{22}(y)=c_2y$, $N_{12}(y)=0$. Then we find
\beann
\text{Var}(\xi_{1,s})&\sim& -\frac{(c_1+c_2)y}{4y}-\frac{2c_1y^3}{y}=\cO_y^*(1)+\cO_y^*(y^2), \quad \text{if $c_1\neq- c_2$,}\\
\text{Var}(\xi_{2,s})&\sim& -\frac{(c_1+c_2)y}{4y}-\frac{2c_2y^3}{y}=\cO_y^*(1)+\cO_y^*(y^2), \quad \text{if $c_1\neq -c_2$,}\\
\text{Cov}(\xi_{1,s},\xi_{2,s})&\sim& -\frac{(c_1-c_2)y}{4}=\cO_y^*(y), \quad \text{if $c_1\neq c_2$.}
\eeann
Therefore measuring the variance alone is not expected to yield valuable information; indeed, variance tending to a constant could be interpreted as a normally hyperbolic regime without critical transitions for additive noise \cite{KuehnCT1}. Obviously one could discuss further interesting scalings depending on the matrix $N(y)$. It should be clear from the formulas (V1)-(V5) and the previous discussion how to approach these situations as long as the system is in normal form near the bifurcation point. We proceed to look at some results for the remaining codimension two bifurcations.

\begin{thm}
\label{thm:CK3}
Suppose $0<\epsilon\ll1$ and (A0)-(A4) hold for a codimension-two fast subsystem bifurcation with at least three fast variables (Gavrilov-Guckenheimer, Hopf-Hopf). Assume that $\epsilon$ is sufficiently small and that $N=N(y)$ is a constant matrix. Then the covariance matrix of the process $\xi_s$ near an attracting slow manifold approaching the bifurcation satisfies
\benn
\text{Cov}(\xi_s)=\sigma^2 [H_\epsilon(y)]+\delta_2(s,\epsilon).
\eenn 
where $H_\epsilon(y)=H_0(y)+\cO(\epsilon)$ and 
\begin{enumerate}
 \item[(V6)] (Gavrilov-Guckenheimer) if $N_{11}\neq 0$ and $N_{22}+N_{33}\neq 0$ then 
 \benn
 H_0(y)=\cO_y^*\left(\begin{array}{ccc}
 \frac{1}{\sqrt{y_1}} & \frac{N_{13}}{\omega} &\frac{N_{12}}{\omega} \\
\frac{N_{13}}{\omega} & \frac{1}{y_2} & \frac{N_{22}-N_{33}}{4\omega}\\
 \frac{N_{12}}{\omega}&\frac{N_{22}-N_{33}}{4\omega} & \frac{1}{y_2} \\
\end{array} 
\right).
 \eenn
 \item[(V7)] (Hopf-Hopf, special case: $A_0(y)$ given by \eqref{eq:A0_HH}) if $N_{11}+N_{22}\neq 0$ and $N_{33}+N_{44}\neq 0$ then
\benn
H_0(y)=\cO_y^*\left(\begin{array}{cccc} 
\frac{1}{y_1}& \frac{N_{11}-N_{22}}{4\omega_1} & \frac{N_{14}\omega_2-N_{23}\omega_1}{\omega_1^2-\omega_2^2} & -\frac{N_{24}\omega_1-N_{13}\omega_2}{\omega_1^2-\omega_2^2}\\
 \frac{N_{11}-N_{22}}{4\omega_1} & \frac{1}{y_1} & \frac{N_{13}\omega_1+N_{24}\omega_2}{\omega_1^2-\omega_2^2} & \frac{N_{14}\omega_1-N_{23}\omega_2}{\omega_1^2-\omega_2^2}\\
\frac{N_{14}\omega_2-N_{23}\omega_1}{\omega_1^2-\omega_2^2} & \frac{N_{13}\omega_1+N_{24}\omega_2}{\omega_1^2-\omega_2^2} & \frac{1}{y_2} & \frac{N_{33}-N_{44}}{4\omega_2} \\
-\frac{N_{24}\omega_1-N_{13}\omega_2}{\omega_1^2-\omega_2^2}& \frac{N_{14}\omega_1-N_{23}\omega_2}{\omega_1^2-\omega_2^2}  & \frac{N_{33}-N_{44}}{4\omega_2} & \frac{1}{y_2}\\
\end{array}\right).
\eenn
\end{enumerate}
\end{thm}

We shall omit the calculations for the proof of Theorem \ref{thm:CK3} as it follows the same steps as the proofs of Theorems \ref{thm:CK1}-\ref{thm:CK2}. We remark that the solution of the algebraic equation \eqref{eq:ma1} becomes much more cumbersome for systems in $\R^3$ and $\R^4$ and we compared our solution to the results obtained by a computer algebra system \cite{Mathematica}. Another important note on Theorem \ref{thm:CK3} is that we do not have to assume explicitly that $\omega_1\neq \omega_2$ for the Hopf-Hopf bifurcation since this is included in assumption (A2). The 1:1 resonance case at a Hopf-Hopf bifurcation $\omega_1=\omega_2$ (see \cite{vanGilsKrupaLangford,GH}) naturally appears as a special case in our analysis. In particular, when $|\omega_1^2-\omega_2^2|$ is small then the covariances of the two-by-two off-diagonal blocks in (V7) can also get large near a Hopf-Hopf critical transition.

\section{Double-Singular Variance Asymptotics for the Fold}
\label{sec:fold_asymp}

In the last section we have computed the leading-order term for the covariance near a critical transition for all fast subsystem bifurcations up to codimension two. In current applications of critical transitions one frequently encounters the fold bifurcation. From a mathematical viewpoint, the fold bifurcation is the lowest slow codimension, lowest fast dimension generic bifurcation without further assumptions. Both reasons warrant a more detailed asymptotic study to determine higher-order correction terms to the formula
\benn
\text{Var}(\xi_s)=\sigma^2 \left[\cO_y^*\left(\frac{1}{\sqrt{y}}\right)+\cO(\epsilon)\right]+\cO\left(\frac{\sigma^3}{\epsilon^{\alpha}}\right).
\eenn
from Theorem \ref{thm:CK1} for additive noise. We can always assume a preliminary normal form transformation \cite{MisRoz,Kuznetsov} and consider on the slow time scale
\be
\label{eq:fold1}
\begin{array}{rcl}
dx&=& \frac1\epsilon (y-x^2) ds+ \frac{\sigma}{\sqrt\epsilon }dW_s,\\
dy&=& -1~ds,\\
\end{array}
\ee 
where we assume additive noise to simplify the algebraic manipulations to follow. We also refer to Figure \ref{fig:fig1} and consider the attracting branch $C^a_0=\{(x,y)\in\R^2:x=\sqrt{y}=h_0(y)\}$ of the critical manifold. Fenichel's Theorem provides an attracting slow manifold $C^a_\epsilon=\{(x,y)\in\R^2:x=h_\epsilon(y)=h_0(y)+\cO(\epsilon)\}$. A converging series expansion for $C^a_0$ can easily be derived by direct regular asymptotics where convergence of the asymptotic series is guaranteed by Fenichel's Theorem. 

\begin{lem}
\label{lem:asympCeps}
The attracting slow manifold $C_\epsilon^a$ for the fold bifurcation normal form is given by
\be
\label{eq:fold_asymp_mani}
h_\epsilon(y)=\sqrt{y}-\epsilon\frac{1}{4y}-\epsilon^2 \frac{5}{32y^{5/2}}-\epsilon^3 \frac{15}{64y^{4}}-\epsilon^4 \frac{1105}{2048y^{11/2}}+\cO(\epsilon^5).
\ee
Terms of order $\cO(\epsilon^5)$ or higher are omitted but can easily be calculated from a recursive solution of algebraic equations.
\end{lem}

Setting $\xi_s=x_s-h_\epsilon(y)$ and using Lemma \ref{lem:BG_lem} for equation \eqref{eq:fold1} we find that the scaled variance $X_s=\sigma^{-2}\text{Var}(\xi_s)$ satisfies the ODE
\be
\label{eq:var_fold_asymp_full}
\begin{array}{rcl}
\epsilon \dot{X}&=&-4h_\epsilon(y)X+1,\\
\dot{y}&=&-1.\\
\end{array}
\ee
The attracting critical manifold of \eqref{eq:var_fold_asymp_full} is given by $\cC_0=\{(X,y)\in \R^2:X=H_0(y)\}$. Fenichel's Theorem yields an associated attracting slow manifold $\cC_\epsilon=\{(X,y)\in \R^2:X=H_\epsilon(y)=H_0(y)+\cO(\epsilon)\}$. We already know from the proof of Theorem \ref{thm:CK1} that $H_0(y)=1/(4\sqrt{y})$. 

\begin{prop}
\label{prop:var_fold}
The attracting slow manifold $\cC_\epsilon^a$ associated to \eqref{eq:var_fold_asymp_full} has an asymptotic expansion given by
\be
\label{eq:sol_var_fold}
H_\epsilon(y)=\frac{1}{4\sqrt{y}}+\epsilon\frac{3}{32y^2}+\epsilon^2\frac{7}{64y^{7/2}}+\epsilon^3\frac{201}{1024 y^5}+\epsilon^4\frac{3837}{8192y^{13/2}}+\cO(\epsilon^5).
\ee
Terms of order $\cO(\epsilon^5)$ or higher are omitted but can easily be calculated from a recursive solution of algebraic equations.
\end{prop}

\begin{proof}
We make the ansatz $H_\epsilon(y)=H_0(y)+\epsilon H_1(y)+\epsilon^2H_2(y)+\cdots$. Using this ansatz and the result from Lemma \ref{lem:asympCeps} in \eqref{eq:var_fold_asymp_full} we get a hierarchy of algebraic equations at different orders
\beann
0&=& 1-4h_0(y)H_0(y)\\
\frac{dH_{k-1}}{dy}&=& -4\sum_{i,j:~i+j=k}H_i(y)h_j(y)
\eeann 
where $k\in\{1,2,\ldots\}$. The result \eqref{eq:sol_var_fold} follows by direct calculation.
\end{proof}

We expect that the expansion up to fourth order of $H_\epsilon$ is sufficient for all practical purposes. Recall from the end of Section \ref{sec:SDE_fs} that the condition $\epsilon^{k_0+\alpha}H_{k_0}(y)=\cO( \sigma)$ determines whether terms of the expansion for $H_\epsilon$ can be moved to the higher-order correction $\cO(\sigma^3/\epsilon^{\alpha})$. Proposition \ref{prop:var_fold} yields the conditions 
\be
\label{eq:hide_const}
\frac{\epsilon^{k_0+\alpha}}{y^{(3k_0+1)/2}}=\cO(\sigma) \qquad \text{for $k_0\in\{1,2,\ldots\}$}.
\ee
For the fold bifurcation, we know that the critical scaling of $y$ to stay inside the normally hyperbolic regime is $y\sim \epsilon^{2/3}$; see Section \ref{sec:fast_slow} and assumption (A0) as well as Lemma \ref{lem:asympCeps}. Suppose $y\sim \epsilon^{2\alpha}$ for some $\alpha<1/3$ and use the scaling in \eqref{eq:hide_const} for $k_0=1$ then $\epsilon^{1-3\alpha}=\cO(\sigma)$ is the condition to move all slow manifold correction terms into the higher-order for the variance estimate; for $\alpha =0$ this condition obviously reduces to the known fact $\epsilon=\cO(\sigma)$ from equation \eqref{eq:move_everything}.

Using the critical scaling $2\alpha=2/3$ in \eqref{eq:hide_const} we get the condition $\epsilon^{1-1}=1= \cO(\sigma)$ which can never hold under assumption that $\sigma(\epsilon)\ra 0$ as $\epsilon \ra 0$ by assumption (A4). Therefore, the slow manifold approximation in (R1) obtained by the linearized process $\xi_s$ for the moments is not valid in (R2).

\section{Applications}
\label{sec:applications}

We are going to present five applications to illustrate the previous results. We also indicate how novel conclusions about the applications follow from the theory. 

\subsection{A Climate Box-Circulation Model}
\label{ssec:climate}
 
The Stommel model \cite{Stommel1} describes the \texttt{North Atlantic Thermohaline Circulation (THC)} by two boxes $B_1$ and $B_2$ representing low and high latitudes respectively. An atmospheric freshwater flux and differences in insolation can induce temperature and salinity differences $\Delta T = T_1-T_2$ and $\Delta S=S_1-S_2$. The resulting system has an Atlantic northward surface current and an Atlantic southward bottom current. For a version of Stommel's box model \cite{Cessi} it can be shown that $(\Delta T,\Delta S)$ obey a two-dimensional fast-slow system where the temperature difference represents the fast variable \cite{BerglundGentz5}. After reduction to an attracting slow manifold and a re-scaling of the variables the dynamics reduces to 
\be
\label{eq:SDE_climate}
\dot{Y}=\mu-Y\left(1+\eta^2(1-Y)^2\right)
\ee 
where $Y$ represents the salinity difference, we fix $\eta^2=7.5$ and $\mu$ is a parameter proportional to the atmospheric freshwater flux. Obviously the freshwater flux can also be viewed as a dynamical variable and we assume that it changes slower than $Y$. Furthermore we assume that \eqref{eq:SDE_climate} is subject to small stochastic perturbations which is reasonable if we decide not to model the system in more detail. Setting $x:=Y$ and $y:=\mu$ we get another two-dimensional fast-slow system
\be
\label{eq:main_box_model}
\begin{array}{rcl}
 dx_s &=& \frac1\epsilon\left[y_s-x_s(1+7.5(1-x_s)^2)\right]ds +\frac{\sigma}{\sqrt\epsilon} F(y_s)dW_s,\\
 dy_s &=& g(x_s,y_s) ds.\\
\end{array}
\ee 
The deterministic critical manifold is $C_0=\{(x,y)\in\R^2:y=x(1+7.5(1-x)^2)=:h_0(x)\}$, which is immediately recognized as a classical \texttt{S-shaped} (or \texttt{cubic}) fast subsystem nonlinearity. There are two fold points (fast subsystem fold bifurcations) at 
\benn
(x^-,y^-)=\left(\frac{1}{15}(10-\sqrt{15}),\frac{11}{9}+\frac{1}{\sqrt{15}}\right)\qquad \text{and} \qquad
(x^+,y^+)=\left(\frac{1}{15}(10+\sqrt{15}),\frac{11}{9}-\frac{1}{\sqrt{15}}\right).
\eenn
The critical manifold splits into three parts $C_0^{a,-}:=C_0\cap \{x<x_-\}$, $C_0^{r}:=C_0\cap \{x^-<x<x^+\}$, and $C_0^{a,+}:=C_0\cap \{x>x^+\}$ where $C_0^{a,\pm}$ are attracting and $C_0^r$ is repelling. The lower branch $C_0^{a,-}$ represents small salinity difference which corresponds to a weak THC. The upper branch $C_0^{a,+}$ corresponds to a strong THC which can be viewed as the present state of the climate. A critical transition from a strong to a weak THC would mean a significant cooling of the mild European climate. Therefore, we shall focus on the critical transition $(x^+,y^+)$ with initial conditions on $C^{a,+}_0$. The initial condition will be fixed at $(x_0,y_0)=(x_0,3/2)\in C_0^{a,+}$ which roughly corresponds to the \texttt{drop point} \cite{MKKR} on the upper attracting critical manifold after a transition at $(x^-,y^-)$.

\begin{figure}[htbp]
\centering
\psfrag{x}{$x$}
\psfrag{xd}{$xd$}
\psfrag{y}{$y$}
\psfrag{t}{$t$}
\psfrag{a}{(a)}
\psfrag{b}{(b)}
\psfrag{c}{(c)}
\psfrag{d}{(d)}
 \includegraphics[width=1\textwidth]{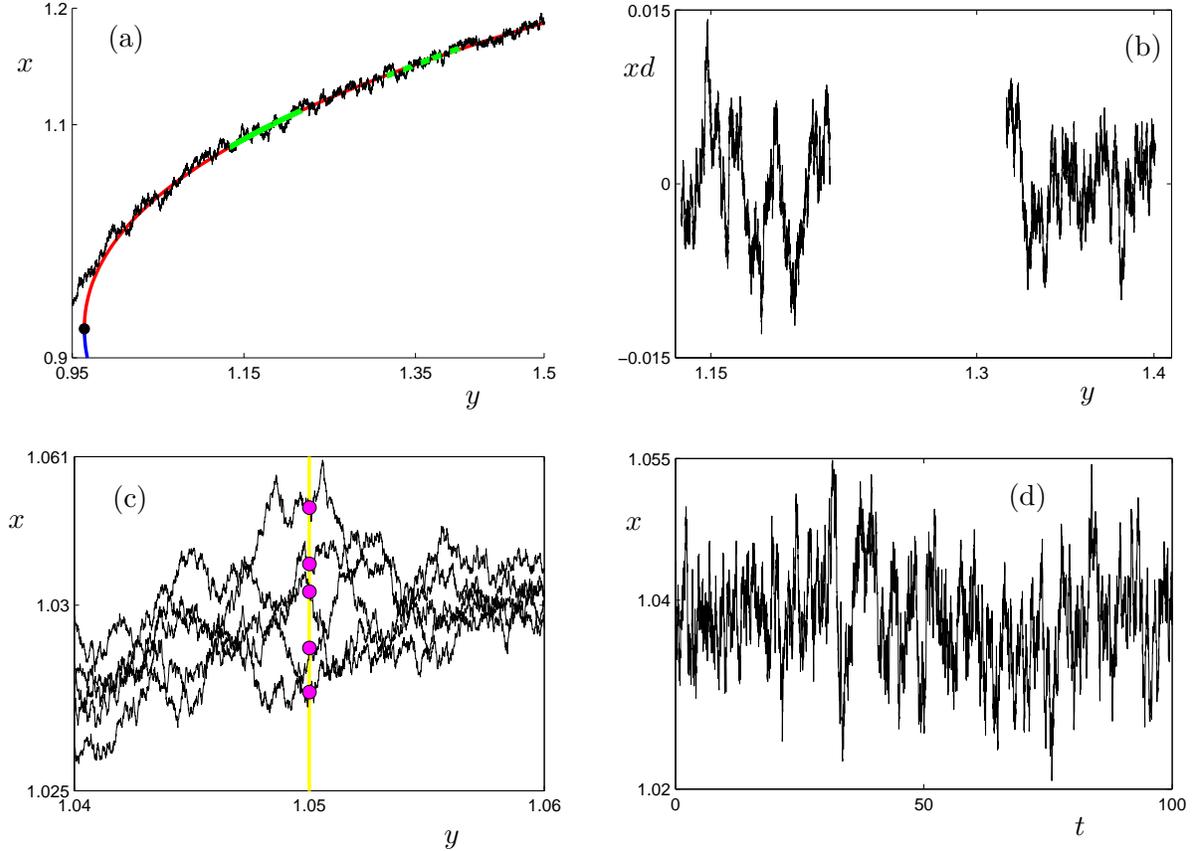}
\caption{\label{fig:fig5}Illustration of the different techniques (M1)-(M4) to approximate the variance $\text{Var}(x(y))$; we use the Stommel-Cessi model \eqref{eq:main_box_model} with parameters given in \eqref{eq:para_Cessi}. (a) Typical time series (black) near the attracting critical manifold $C^{a,+}$ (red) up to the fold point $(x^+,y^+)$ (black dot). We also show two sliding windows (green) where the dashed green line is a linear trend and the solid green line is given by $C^{a,+}_0$ i.e. the green curves are used for linear and CM detrending respectively. (b) Detrended time series $xd$ from (a) corresponding to the two (green) sliding windows. (c) Zoom near $y=1.05$, 5 sample paths are shown. The dots (magenta) mark the five points of the paths at $y=1.05$. To calculate the variance $\text{Var}(x(y=1.05))$ one simulates many paths. (d) Simulation of the fast subsystem of \eqref{eq:main_box_model} with y=1.05 for a fixed fast time $t\in[0,100]$.}
\end{figure}

We start by simulating \eqref{eq:main_box_model} using an Euler-Maruyama method \cite{Higham} using
\be
\label{eq:para_Cessi}
\epsilon=0.01,\qquad \sigma=0.01, \qquad F(y)\equiv 1, \qquad g(x,y)\equiv-1.
\ee
where the assumptions on $g$ mean that one may also interpret $y$ as a time variable. A typical sample path is shown in Figure \ref{fig:fig5}(a); the path is stopped at a final value $y=0.95$. We want to estimate the variance $\text{Var}(y_s)$ from a time series  
\be
\label{eq:tseries}
y_0=y_{s_0},y_{s_1},\ldots,y_{s_N}=0.95, \qquad x_{s_0}, x_{s_1}, \ldots,x_{s_N}.
\ee
The values $x_{s_j}=:x_j$ can be viewed as functions of $y$ since $y_s=(s-s_0)+y_0$ and we indicate this by writing $\text{Var}(x(y)):=\text{Var}(x_s)$. The goal is to estimate the variance. There are several possibilities to extract an approximation:

\begin{enumerate}
 \item[(M1)] Consider a single time series. Select a \texttt{moving window} of fixed length $M$ and compute the sample variance
for $M+1$ consecutive points $x_j,\ldots,x_{j+M}$; see Figure \ref{fig:fig5}(a)-(b). This provides an estimate for the variance $\text{Var}(x(y))$ roughly at the midpoint of the moving window $\frac1M\sum_{k=0}^M y_{j+k}$. The idea is that if the window is sufficiently small and we have sufficiently many data points inside each window then we can calculate a good approximation to $\text{Var}(x(y))$ for each $y$.
 \item[(M2)] Consider a single time series as for (M1). We can remove a given \texttt{trend} from \eqref{eq:tseries} before calculating the variance. For example, interpolating \eqref{eq:tseries} linearly and subtracting the resulting linear function from the time series yields a variance estimate with \texttt{linear detrending}. Another natural possibility is to remove the critical manifold as a trend; we call this \texttt{critical manifold (CM) detrending}. See also Figure \ref{fig:fig5}(a)-(b).
 \item[(M3)] Another possibility is to consider a large number $R$ of time series $x^{(r)}_{0}, x^{(r)}_{1}, \ldots,x^{(r)}_{N}$ for $r\in\{1,2,\ldots,R\}$ and then calculate the variance $\text{Var}(x(y_j))$ at $y_j$ as the sample variance of $\{x^{(1)}_{j}, x^{(2)}_{j}, \ldots,x^{(R)}_{j}\}$. This idea is illustrated in Figure \ref{fig:fig5}(c) and avoids the moving window technique. However, it does require multiple time series passing near the same critical point.
\item[(M4)] Instead of simulating the entire SDE \eqref{eq:main_box_model} we can also assume that $y=y_j$ is constant, simulate the fast subsystem for a sufficiently long time and then calculate $\text{Var}(x(y_j))$ from this fast subsystem time series; see Figure \ref{fig:fig5}(d).
\end{enumerate}

\begin{figure}[htbp]
\centering
\psfrag{a}{(a)}
\psfrag{b}{(b)}
\psfrag{c}{(c)}
\psfrag{d}{(d)}
\psfrag{e}{(e)}
\psfrag{f}{(f)}
\psfrag{x}{$x$}
\psfrag{y}{$y$}
\psfrag{V}{$V$}
 \includegraphics[width=1\textwidth]{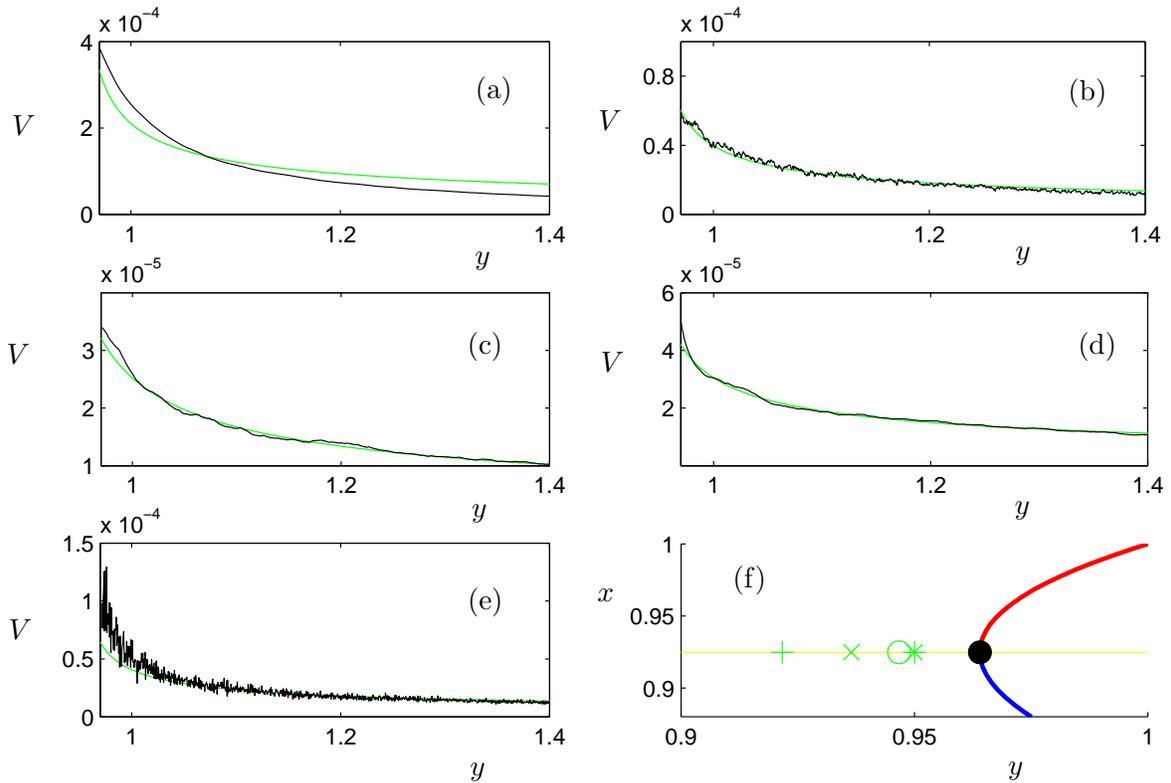}
\caption{\label{fig:fig6}Comparison of different methods to estimate the variance $V=\text{Var}(x(y))$ for the Stommel-Cessi model \eqref{eq:main_box_model} with parameters given in \eqref{eq:para_Cessi}. The black curves in (a)-(e) indicate the variance estimate and the green curves are obtained by least squares fit of \eqref{eq:ls_fit}. (a) Sliding window technique (M1) without detrending, average over 1000 sample paths. (b) Sample paths ``pointwise variance'' (M3), average over 1000 sample paths. (c) Sliding window with linear detrending (M2), average over 100 sample paths. (d) Sliding window with CM detrending, average over 100 sample paths. (e) Fast subsystem simulation (M4) for a fast time $t\in[0,100]$. (f) The critical manifold (red/blue) with fold point (black) is shown. The green markers indicate the estimators for $y_c$ from a least squares fit of \eqref{eq:ls_fit} plotted at the same x-value as the fold point; the green star ``*'' is the lower bound estimate for $y_c$ from (a) and (e), the green circle ``o'' marks $y_c$ for (b), the green plus ``+'' corresponds to (c) and the green ``x'' marks $y_c$ for (d).}
\end{figure}

Each of the methods (M1)-(M4) has different advantages and disadvantages. A direct sample variance measurement using the sliding window technique (M1) does include the curvature of the critical manifold in the estimate as demonstrated in \cite{KuehnCT1}. Linear detrending requires no a priori knowledge about the dynamics but can obviously not remove curvature near the fold point. CM detrending corresponds to the change of variable $\tilde{\xi}_s:=x_s-h_0(y_s)$ which is closest to the theoretical situation discussed in Sections \ref{sec:SDE_fs}-\ref{sec:fold_asymp}. However, this requires a priori knowledge of the critical manifold. The method (M3) requires many sample paths which is a restriction while the method (M4) requires the ability to simulate/measure the fast subsystem for a long time. In Figure \ref{fig:fig6} we compare the different methods for the Stommel-Cessi model. Figures \ref{fig:fig6}(a)-(e) provide the variance estimates together with a least squares fit of 
\be
\label{eq:ls_fit}
\text{Var}(x(y))=\frac{A}{\sqrt{y-y_c}}
\ee
with fitting parameters $A$ and $y_c$. The results in Figure \ref{fig:fig6} show that all methods can capture the variance increase as predicted by the theory. The sliding window technique seems to deviate the most from the theory compared to the other four methods but it requires the least amount of data as one basically produces a plot similar to Figure \ref{fig:fig6}(a) with just a single time series. By fitting \eqref{eq:ls_fit} we also obtain an estimate for the critical transition point $y_c$ which is slightly delayed due to positive $\epsilon$. All techniques capture this effect. We get the estimate that $y_c\in[0.92,0.95]$ which is a very good prediction compared to direct simulations. Overall one may conclude that the theoretical predictions of variance increase near a fold point apply very well in the context of the Stommel-Cessi model \eqref{eq:main_box_model} and that the different time series analysis methods all have advantages as well as disadvantages depending on the situation. Obviously we do not make any claims about the real THC with our calculations as this requires the analysis of temperature data sets.  

\subsection{Epidemics on Complex Adaptive Networks}
\label{ssec:epidemics}

Consider a network of social contacts and a disease that can be spread via these contacts as described in \cite{GrossDLimaBlasius}. Individuals of the population correspond to \texttt{nodes (or vertices)} and social contacts correspond to undirected \texttt{links (or edges)}. Denote the total number of nodes by $N$ and the number of links by $K$ and assume that $N$ and $K$ are constant; define the \texttt{mean degree} $\mu:=2K/N$. The dynamical states of the nodes are either \texttt{susceptible} or \texttt{infected} giving classical \texttt{SIS dynamics}. If a link between an infected and a susceptible node exists then the susceptible node becomes infected with probability $p$ at each time step. Infected nodes recover to susceptible status with probability $r$. Susceptibles might try to change their connection from an infected node to a susceptible one. To model this effect we assume that the network is \texttt{adaptive} so that the topology of the network influences the dynamics of the nodes and vice versa. We use the following dynamical variables to describe the SIS adaptive network
\benn
\begin{array}{rclrl}
x_1&:=&\frac{\text{\# $\{$infected$\}$} }{N}&=& \text{``density of infected individuals''},\\
x_2&:=&\frac{\text{\# $\{$links between infected and infected$\}$} }{N}&=& \text{``per capita density of II-links''},\\
x_3&:=&\frac{\text{\# $\{$links between susceptible and susceptible$\}$} }{N}&=& \text{``per capita density of SS-links''}.\\
\end{array}
\eenn
Note that the density of susceptible individuals is $(1-x_1)$ and the per capita density of SI-links is $(\mu/2-x_2-x_3)$. To capture the full adaptive network dynamics one would have to take into account also triples (triangle subgraphs) and all other higher-order \texttt{(network) moments}. We use the \texttt{moment closure pair approximation} \cite{KeelingRandMorris} to express the higher-order moments in terms of $x$ which yields \cite{GrossDLimaBlasius}
\be
\label{eq:momc_Gross}
\begin{array}{rcl}
x_1'&=&p(\frac\mu2-x_2-x_3)-rx_1,\\
x_2'&=&p(\frac\mu2-x_2-x_3)\left(\frac{\frac\mu2-x_2-x_3}{1-x_1}+1\right)-2rx_2,\\
x_3'&=&(r+w)(\frac\mu2-x_2-x_3)-\frac{2p(\frac\mu2-x_2-x_3)x_3}{1-x_1}.\\
\end{array}
\ee
For our analysis we fix the following parameters
\be
\label{eq:para_Gross}
r=0.002,\qquad w=0.4, \qquad N=10^5, \qquad K=10^6 \quad \Rightarrow \mu=20.
\ee
Assume that $p$ is a slow variable and increases over time. For example, we could think of a virus that evolves towards a more infectious variant in time. Using the standard notation for slow variables we let $y:=p$ and assume $y'=\epsilon$. It is also reasonable to consider the scenario that the density of infected nodes and the link densities can exhibit stochastic fluctuations; in particular, this might lead to a model that is more realistic than the moment closure ODEs. Combining this assumption, the slow equation and \eqref{eq:momc_Gross} we get
\be
\label{eq:momc_Gross1}
\begin{array}{rcl}
dx_1&=&\frac1\epsilon\left[y(\frac\mu2-x_2-x_3)-rx_1\right] ds + \frac{\sigma_1}{\sqrt\epsilon}dW^{(1)} ,\\
dx_2&=&\frac1\epsilon\left[y(\frac\mu2-x_2-x_3)\left(\frac{\frac\mu2-x_2-x_3}{1-x_1}+1\right)-2rx_2\right]ds +\frac{\sigma_2}{\sqrt\epsilon}dW^{(2)},\\
dx_3&=&\frac1\epsilon\left[(r+w)(\frac\mu2-x_2-x_3)-\frac{2y(\frac\mu2-x_2-x_3)x_3}{1-x_1}\right]ds+\frac{\sigma_3}{\sqrt\epsilon} dW^{(3)},\\
dy&=& 1~ds,\\
\end{array}
\ee
where we omit the subscript $s$ for $x_s$ and $W_s=(W^{(1)}_s,W^{(2)}_s,W^{(3)}_s)^T$ for notational convenience. Although the algebraic expression for the deterministic critical manifold $C_0$ of \eqref{eq:momc_Gross1} can be computed we shall only focus on the subset 
\benn
C^*_0:=\left\{(x,y)\in\left([0,1]\times [0,\mu/2]^2\right)\times [0,1]:x_1=0=x_0,x_3=\frac\mu2 \right\}\subset C_0.
\eenn 
The solution $x_1=0=x_2$ and $x_3=\mu/2$ corresponds to an equilibrium point of \eqref{eq:momc_Gross} with no infected nodes that can also be obtained by considering the initialization of the network as a random graph \cite{GrossDLimaBlasius}. The fast subsystem linearization around $C_0^*$ is given by 
\be
\label{eq:lin_Gross}
D_xf|_{C_0^*}=\left(
\begin{array}{ccc} 
-r & -y & -y \\
0 & -y-2r & -y\\
0 & y\mu-r-w & y\mu-r-w\\
\end{array}
\right).
\ee

\begin{figure}[htbp]
\centering
\psfrag{Cr}{$C_0^{*r}$}
\psfrag{Ca}{$C_0^{*a}$}
\psfrag{BP}{BP}
\psfrag{LP}{LP}
\psfrag{mH}{H}
\psfrag{x1}{$x_1$}
\psfrag{x2}{$x_2$}
\psfrag{x3}{$x_3$}
\psfrag{y}{$y$}
 \includegraphics[width=1\textwidth]{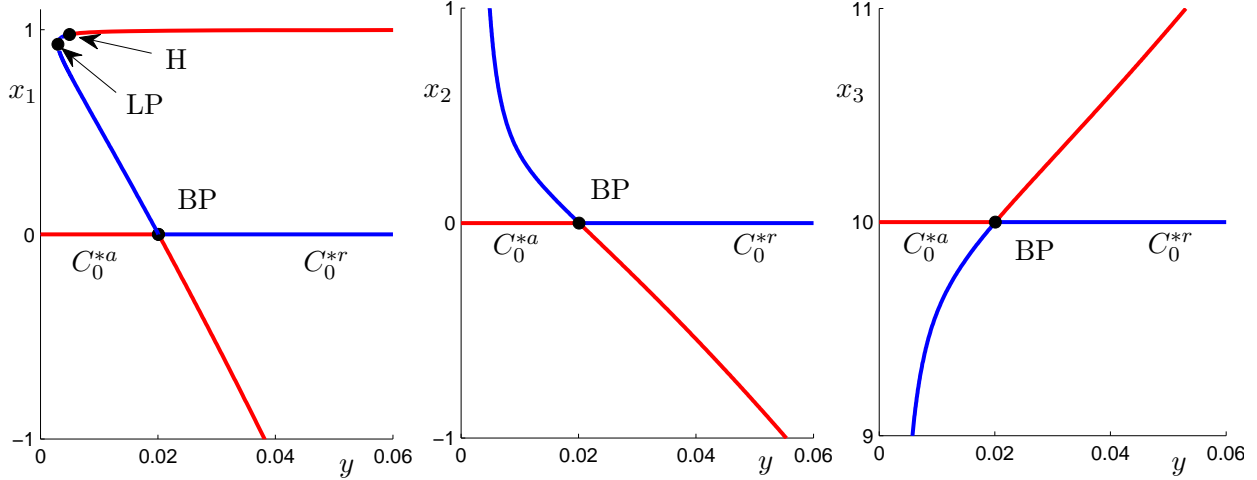}
\caption{\label{fig:fig2}Parts of the critical manifold $C_0$ for the SIS-model \eqref{eq:momc_Gross1} where attracting branches are red and repelling branches are blue; parameters are given by \eqref{eq:para_Gross}. The manifolds (fast subsystem equilibrium branches) have been computed using numerical continuation \cite{MatCont}. A transcritical bifurcation (branch point, [BP]) is detected at $y=y_c=0.0201$. For the number of infected nodes we show the continuation of $C_0$ away from the branch point; it undergoes a fold bifurcation (limit point, [LP]) and stabilizes at a supercritical Hopf bifurcation [H].}
\end{figure}

Using the parameter values \eqref{eq:para_Gross} and \eqref{eq:lin_Gross} we can easily calculate that a single eigenvalue of \eqref{eq:lin_Gross} crosses the imaginary axis at $y=y_c=0.0201$. Another direct calculation shows that $C_0^*$ splits into two subsets $C_0^{*a}=\{y<y_c\}\cap C_0^*$ and $C_0^{*r}=\{y>y_c\}\cap C_0^*$ where $C_0^{*a}$ is normally hyperbolic attracting and $C_0^{*r}$ is normally hyperbolic repelling. Note that the fast subsystem bifurcation of the trivial solution $C_0^*$ to \eqref{eq:momc_Gross} suggests a transcritical or a pitchfork bifurcation. In Figure \ref{fig:fig2} we show part of the critical manifold $C_0$ including the trivial solution $C_0^*$; the computation has been carried out using numerical continuation \cite{MatCont}. Figure \ref{fig:fig2} shows that the bifurcation is transcritical and $y=y_c$ is the infection probability threshold.

\begin{figure}[htbp]
\centering
\psfrag{a}{(a)}
\psfrag{b}{(b)}
\psfrag{c}{(c)}
\psfrag{d}{(d)}
\psfrag{e}{(e)}
\psfrag{f}{(f)}
\psfrag{x1}{$x_1$}
\psfrag{x2}{$x_2$}
\psfrag{x3}{$x_3$}
\psfrag{y}{$y$}
\psfrag{V}{$V$}
\psfrag{V-}{$\frac{1}{\bar{V}}$}
\psfrag{Vb}{$\bar{V}$}
 \includegraphics[width=1\textwidth]{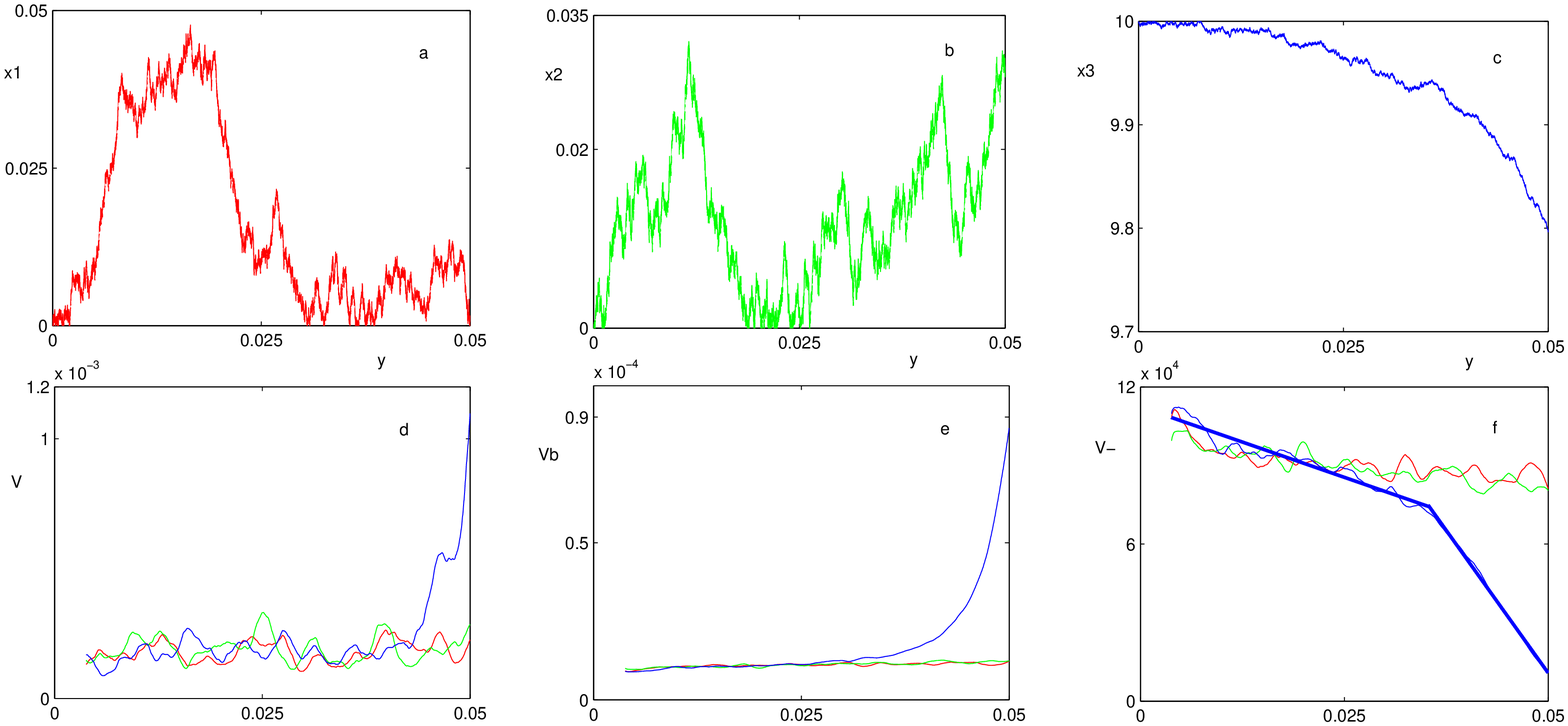}
\caption{\label{fig:fig3}Simulation results for \eqref{eq:momc_Gross1} with boundary conditions to constrain $x=(x_1,x_2,x_3)\in[0,1]\times[0,\mu/2]^2$; parameter values are given in \eqref{eq:para_Gross} and $(\sigma_1,\sigma_2,\sigma_3,\epsilon)=(0.01,0.01,0.01,0.005)$. (a)-(c) show a time series and (d) shows the associated variance of this series calculated by a sliding window technique. (e) Average $\bar{V}$ of the sliding-window variance for 1000 sample paths; we see that $\bar{V}_3$ shows an increase near the bifurcation. (f) Inverse of averaged variance $1/\bar{V}$ where we clearly see that $\bar{V}_3$ scales like $1/(y_c^\epsilon-y)$ up to a delayed epidemic threshold $y_c^\epsilon$. We also show two linear fits to $(\bar{V}_3)^{-1}$, one before the threshold (early-warning regime) and one after the threshold (start of critical transition). The actual full epidemic outbreak is not shown in the plot and occurs roughly between $y=0.05$ and $y=0.07$.}
\end{figure}

For direct simulation of \eqref{eq:para_Gross} one has to ensure that $x\in[0,1]\times[0,\mu/2]^2$ as the densities are constrained. Therefore, we set a point that lands outside of the domain at a time step to its associated boundary value, e.g. if $x_1(s_j)<0$ for some numerical time step $s_j$ then we set $x_1(s_j)=0$. This simulation is formally outside of the theory developed in Sections \ref{sec:fast_slow}-\ref{sec:fold_asymp}. Nevertheless, Figure \ref{fig:fig3} shows that the theoretical results are useful. Figure \ref{fig:fig3}(a)-(c) shows a typical sample path and we see that the $x_3$-coordinate in (c) starts to decrease beyond the singular limit critical point whereas the other two variables do not show any recognizable trend in (a)-(b). It is interesting to note that the density of infected individuals does not seem to play a role as an early-warning sign for the epidemic outbreak.

Figure \ref{fig:fig3}(d) shows the variance $V=(V_1,V_2,V_3)=(\text{Var}(x_1),\text{Var}(x_2),\text{Var}(x_3))$ associated to the sample path in (a)-(c) by using a sliding window technique; the size of the sliding window corresponds to the gap in the curves near $y=0$. Figure \ref{fig:fig3}(e) shows an average variance $\bar{V}_i$ for $i=\{1,2,3\}$ over 1000 sample paths. Observe that $x_3$ is the best predictor variable and this leads to the conjecture that the increase in variance should scale like the inverse of the distance to the critical transition; see Figure \ref{fig:fig3}(f). Note that the critical transition at $y=y_c$ for $\epsilon=0$ is delayed due to the time scale separation \cite{KuehnCT1}. We conclude from our results that it is crucial what property of a complex system we actually \emph{measure} to make predictions. Indeed, the SIS-epidemic model suggests that measuring the variance in links can be much more important than just the number of infected individuals. Furthermore, the technique we developed here can also be applied to adaptive networks in completely different contexts \cite{GrossSayama,KuehnZschalerGross}. 

\subsection{A Switch in Systems Biology}
\label{ssec:SysBio}

To understand complex molecular networks one often seeks to construct models of simpler building blocks of the network. These building blocks are composed of genes and proteins and can often act as various kinds of ``switches'' inside a more complex system. Dynamical systems methods for these systems biology questions are a highly active research area \cite{Brackleyetal}. Low-dimensional dynamical systems have been proposed to model the smallest units in a molecular network. A typical example is the \texttt{activator-inhibitor} system. Suppose the \texttt{activator} species $R$ is produced in an autocatalytic reaction but rising $R$ also promotes the production of an \texttt{inhibitor} species $X$. More concretely, one may think of both species $(R,X)$ as concentrations of proteins. Activator-inhibitor systems incorporate positive and negative feedback which can lead to oscillations. One model proposed for activator-inhibitor oscillators \cite{TysonChenNovak} is
\be
\label{eq:ActInh}
\begin{array}{lcl}
R'&=& k_0G(k_3R,k_4,J_1,J_2)+k_1S-k_2R-k_7XR\\
X'&=& k_5R-k_6X\\
\end{array}
\ee
where the \texttt{Goldbeter-Koshland function} $G$ \cite{GoldbeterKoshland,NovakPatakiCilibertoTyson} is
\benn
G(u,v,J,K)=\frac{2uK}{v-u+vJ+uK+\sqrt{(v-u+vJ+uK)^2-4(v-u)uK}}
\eenn
and $k_j$ for $j\in\{1,2,3,4,5,6,7\}$, $J_i$ for $i\in\{1,2\}$ and $S$ are parameters. The main bifurcation parameter is the \texttt{signal strength} $S$ which can be viewed as an external input to the system \eqref{eq:ActInh}. We are going to fix the other parameters following \cite{TysonChenNovak} as
\benn
k_0=4, \quad k_1=k_2=k_3=k_4=k_7=1, \quad k_5=0.1, \quad k_6=0.075, \quad J_1=J_2=0.3.
\eenn
 
\begin{figure}[htbp]
\centering
\psfrag{x1}{$x_1$}
\psfrag{x2}{$x_2$}
\psfrag{y}{$y$}
\psfrag{C0}{$C_0$}
\psfrag{LPC}{LPC}
\psfrag{H}{H}
 \includegraphics[width=1\textwidth]{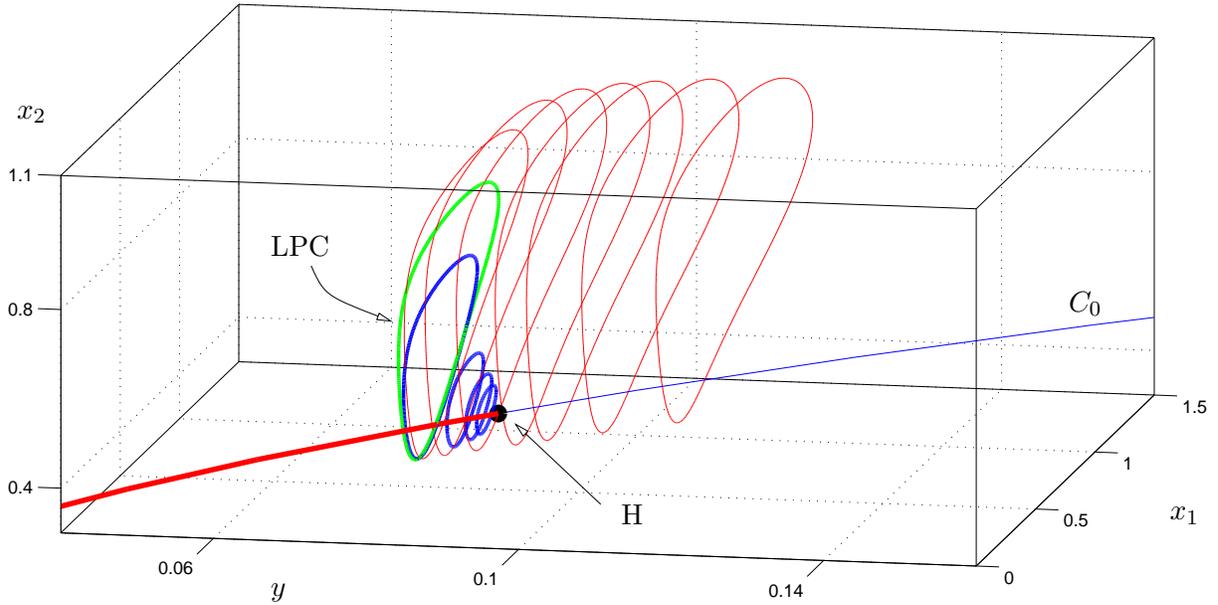}
\caption{\label{fig:fig7}Dynamics for $\epsilon=0$ for the deterministic version of the activator-inhibitor system \eqref{eq:Tyson}. The critical manifold $C_0$ is the red-blue curve which looses normal hyperbolicity at a fast subsystem subcritical Hopf bifurcation (black dot, [H]) at $y\approx 0.09146$. The generated small limit cycles (blue) are first repelling and then undergo a fold (or saddle-node, or limit point [LPC]) bifurcation; the large fast subsystem limit cycles (red) are attracting. A critical transition occurs near the Hopf bifurcation as trajectories leave the critical manifold and jump to a large limit cycle. See also Figure \ref{fig:fig8} for the fast subsystem phase portraits.}
\end{figure}

\begin{figure}[htbp]
\centering
\psfrag{x1}{$x_1$}
\psfrag{x2}{$x_2$}
\psfrag{a}{(a)}
\psfrag{b}{(b)}
\psfrag{c}{(c)}
 \includegraphics[width=1\textwidth]{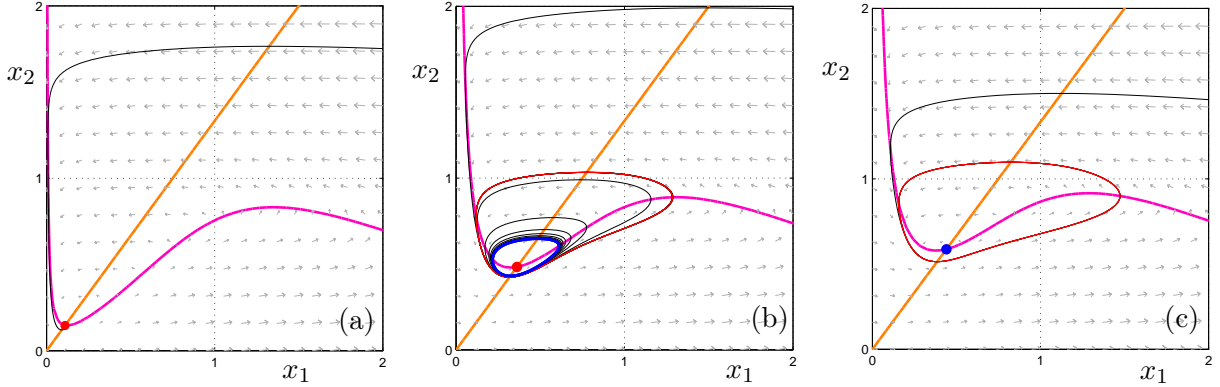}
\caption{\label{fig:fig8}Illustration of the subcritical Hopf bifurcation for the deterministic fast subsystem of \eqref{eq:Tyson}; equivalently the results apply to \eqref{eq:ActInh} with $(R,X)=(x_1,x_2)$. Nullclines are shown in magenta for $x_1$ and in orange for $x_2$. Trajectories are black and invariant sets are red (stable) and blue (unstable). (a) $y=0.01$: The system has a stable spiral sink. (b) $y=0.085$: In addition to the spiral sink there exist a small unstable and large stable limit cycle. (c) $y=0.12$: The equilibrium point is a spiral source and only the large stable limit cycle exists.}
\end{figure}

Let us consider the case when the external input $S$ is a slow signal that starts out sufficiently low so that no oscillations occur for \eqref{eq:ActInh}. Then we let $S=:y$ increase until a transition to large oscillations is observed. It is reasonable to assume that the variables $(R,X)=:(x_1,x_2)$ are stochastic with correlated noise. Under these assumptions we can write \eqref{eq:ActInh} as the SDE
\be
\label{eq:Tyson}
\begin{array}{lcl}
dx_1&=&\frac1\epsilon\left[ 4G(x_1,1,0.3,0.3)+y-x_1-x_1x_2 \right]ds +\frac{\sigma}{\sqrt\epsilon}\left(F_{11}dW^{(1)}+F_{12}dW^{(2)}\right),\\
dx_2&=& \frac1\epsilon\left[0.1x_1-0.075x_2\right]ds+\frac{\sigma}{\sqrt\epsilon}\left(F_{21}dW^{(1)}+F_{22}dW^{(2)}\right),\\
dy&=& 1 ~ds.\\
\end{array}
\ee
The critical manifold $C_0$ for the deterministic part of \eqref{eq:Tyson} is given by
\benn
C_0=\left\{(x_1,x_2,y)\in\R^{3}:x_2=\frac43 x_1,y=x_1+\frac43 x_1^2-4G(x_1,1,0.3,0.3)\right\}.
\eenn
It is easy to check that the critical manifold is attracting for $y<y_{H,1}$ and $y>y_{H,2}$ and repelling for $y_{H,1}<y<y_{H,2}$ where $y_{H,1}\approx 0.091462$ and $y_{H,2}\approx 0.440903$ are fast subsystem Hopf bifurcation points. We focus on the subcritical Hopf bifurcation at $y=y_{H,1}$. Figure \ref{fig:fig7} shows an illustration of the singular limit dynamics near this Hopf bifurcation point. Repelling fast subsystem limit cycles are generated at the Hopf bifurcation. These cycles undergo a further fold (or saddle-node, or limit point) bifurcation to attracting cycles which grow rapidly. By looking at the phase plane  of the fast subsystem in Figure \ref{fig:fig8} we observe that the $x_1$-nullcline can also be viewed as another critical manifold of the two-dimensional system $(x_1,x_2)$ where $x_2$ would be fast and $x_1$ be even faster which yields a three-scale system with canard explosion \cite{KrupaPopovicKopell}. The important outcome of this mechanism is that passing from the attracting critical manifold 
\benn
C^a_0:=\{(x_1,x_2,y)\in\R^3:y<y_{H,1}\}\cap C_0
\eenn
through the Hopf bifurcation produces a critical transition to large limit cycle oscillations. The critical transition can be viewed as an almost instantaneous switch to sustained oscillations. For the stochastic simulation we recall from the definition in equation \eqref{eq:defineN} that
\benn
N=\left(\begin{array}{cc} N_{11} & N_{12} \\ N_{12} & N_{22} \\\end{array}\right)=
\left(\begin{array}{cc} F_{11}^2+F_{12}^2 & F_{11}F_{21}F_{12}F_{22} \\ F_{11}F_{21}F_{12}F_{22} & F_{21}^2+F_{22}^2 \\\end{array}\right).
\eenn

\begin{figure}[htbp]
\centering
\psfrag{y}{$y$}
\psfrag{V1}{$V_1$}
\psfrag{V2}{$V_2$}
\psfrag{V3}{$C_{12}$}
\psfrag{AI}{Act.-Inh. \eqref{eq:Tyson}}
\psfrag{Hopf}{Hopf n.f. \eqref{eq:Hopf_ex_nf}}
\psfrag{a1}{(a1)}
\psfrag{a2}{(a2)}
\psfrag{a3}{(a3)}
\psfrag{b1}{(b1)}
\psfrag{b2}{(b2)}
\psfrag{b3}{(b3)}
 \includegraphics[width=1\textwidth]{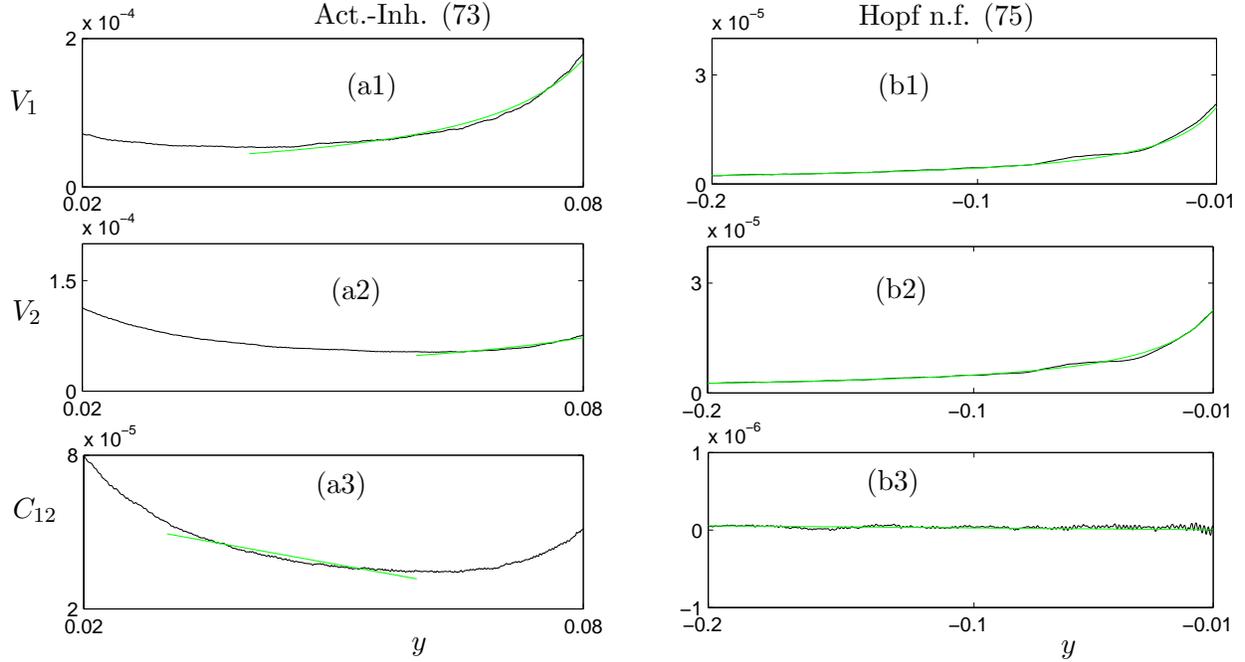}
\caption{\label{fig:fig9}The row labels denote $V_1=\text{Var}(x_1)$, $V_2=\text{Var}(x_2)$ and $C_{1,2}=\text{Cov}(x_1,x_2)$. (a1)-(a3) Parameter values are $\epsilon=10^{-5}$ and $\sigma=10^{-3}$ for \eqref{eq:Tyson}. (b1)-(b3) Parameter values are $\epsilon=5\times10^{-4}$ and $\sigma=10^{-3}$ for \eqref{eq:Hopf_ex_nf}. All figures have been computed from 100 sample paths by a sliding window technique (black curves). The variances have been fitted using \eqref{eq:fit_Hopf} and the covariance have been fitted linearly (green curves). We observe that the normal form corresponds perfectly to the theory but that the three-time scale structure of the activator-inhibitor system becomes visible in the variance and covariance measurements.}
\end{figure}

For numerical simulations fix $N_{11}=1=N_{22}$ and $N_{12}=0.2$. In Figure \ref{fig:fig9}(a1)-(a3) the variance and covariance near the subcritical Hopf bifurcation at $y=y_{H,1}$ for the activator-inhibitor system \eqref{eq:Tyson} are shown. The variances $\text{Var}(x_{1,2})$ have been fitted using
\be
\label{eq:fit_Hopf}
\text{Var}(x_{j}(y))=\frac{A}{y-y_c}, \qquad \text{for $j\in\{1,2\}$}
\ee
with fit parameters $A$ and $y_c$. The covariance has been fitted linearly. The variance of the fastest variable $\text{Var}(x_1(y))$ behaves approximately as predicted near the critical transition as $\cO(1/y)$. However, the variance $\text{Var}(x_2(y))$ of the slower variable $x_2$ does not show a clear increase and the covariance near the critical transition is not constant. This shows that the three-time scale structure requires a very careful analysis and a transformation to normal form would be needed to apply Theorem \ref{thm:CK2}. A prediction of the critical transition point can still work {e.g.}~using $\text{Var}(x_1(y))$ produces the estimate $y_c\approx 0.094$. We have also compared the activator-inhibitor results to a Hopf bifurcation normal form system
\be
\label{eq:Hopf_ex_nf}
\begin{array}{lcl}
dx_1&=&\frac1\epsilon\left[ yx_1-x_2+x_1(x_1^2+x_2^2)\right]ds +\frac{\sigma}{\sqrt\epsilon}\left(F_{11}dW^{(1)}+F_{12}dW^{(2)}\right),\\
dx_2&=& \frac1\epsilon\left[x_1+yx_2+x_2(x_1^2+x_2^2)\right]ds+\frac{\sigma}{\sqrt\epsilon}\left(F_{21}dW^{(1)}+F_{22}dW^{(2)}\right),\\
dy&=& 1 ~ds.\\
\end{array}
\ee
Figure \ref{fig:fig9}(b1)-(b3) shows the results which match Theorem \ref{thm:CK2} as expected. For the covariance there is a clear difference between \eqref{eq:Hopf_ex_nf} and the activator-inhibitor system; compare Figures \ref{fig:fig9}(a3) and \ref{fig:fig9}(b3). The increase of the covariance near the critical transition is not expected and might be related to deterministic rotation around the slow manifold $C^a_\epsilon$ i.e. the manifold is attracting but also a spiral sink of the fast subsystem; see also Section \ref{ssec:ecology} where another possible explanation is given.\\

To conclude, observe that the bifurcation structure displayed by \eqref{eq:Tyson} has a fast subsystem with an S-shaped (cubic) critical manifold which makes the results applicable also to typical neuroscience models such as bursting neurons \cite{Izhikevich,Rinzel}. Therefore, we have shown that subunits of molecular networks as well as neurons in neural networks do have information available that allows them to predict a future state without previous knowledge of the exact position of this state. Whether this predictive potential is actually used in a real molecular or neural network is far beyond the scope of this paper but certainly constitutes a fascinating question. For a recent application to excitable neuron models and epileptic seizures see \cite{MeiselKuehn}.

\subsection{A Predator-Prey Systems near Codimension Two Bifurcation}
\label{ssec:ecology}

Sudden shifts in ecosystems have been a primary motivation to develop the theory of critical transitions \cite{SchefferCarpenter}. Recently also experimental evidence has been provided \cite{DrakeGriffen}. However, many studies seem to view fold critical transitions as the only relevant transition \cite{vanNesScheffer}. This viewpoint does not seem to be appropriate as codimension two (and higher codimension) bifurcations occur very frequently in ecological models \cite{BazykinKhibnikKrauskopf}. Here we focus on the analysis of a classical \texttt{predator-prey model} \cite{BazykinKhibnikKrauskopf}
\be
\label{eq:Kuznetsov_Baz}
\begin{array}{lcl}
x_1'&=& x_1-\frac{x_1x_2}{1+\alpha x_1}-\xi x_1^2,\\
x_2'&=&-\gamma x_2+\frac{x_1x_2}{1+\alpha x_1}-\delta x_2^2,\\
\end{array}
\ee
where $x_1$ represents prey, $x_2$ represents predators and $\alpha$, $\delta$, $\xi$, $\gamma$ are positive parameters. The bifurcation analysis of \eqref{eq:Kuznetsov_Baz} in the $(\alpha,\delta)$-parameter plane has been nicely described by Kuznetsov (see \cite{Kuznetsov}, p.327-332) under the assumptions $\gamma=1$ and $0<\xi\ll1$. For numerical simulation we fix $\gamma=1$ and $\xi=0.01$.

\begin{figure}[htbp]
\centering
\psfrag{y1}{$y_1$}
\psfrag{y2}{$y_2$}
\psfrag{R1}{$Q_1$}
\psfrag{R2}{$Q_2$}
\psfrag{R3}{$Q_3$}
\psfrag{LP}{LP}
\psfrag{H}{H}
\psfrag{BT}{BT}
 \includegraphics[width=0.85\textwidth]{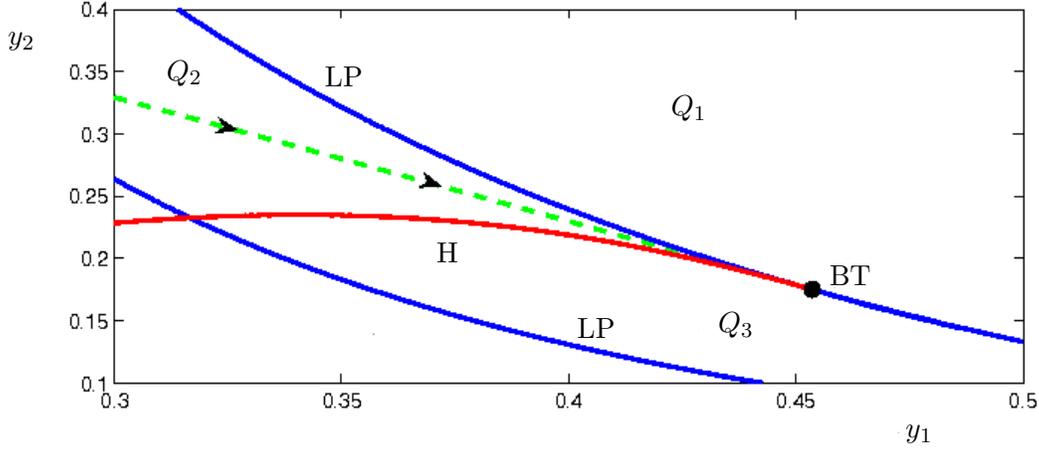}
\caption{\label{fig:fig10}Partial bifurcation diagram of the Bazykin predator-prey model \eqref{eq:Kuznetsov_Baz} with $\gamma=1$ and $\xi=0.01$. The parameters $(y_1,y_2)$ can be viewed as slow variables. The main organizing center in the diagram is the codimension-two Bogdanov-Takens (black dot, [BT]) point that occurs at a tangency of Hopf (red, [H]) and fold (blue, [LP]) bifurcation curves. Phase space diagrams for the different regions $Q_1$, $Q_2$ and $Q_3$ are shown in Figure \ref{fig:fig11}; note that $Q_3$ splits into two sub-regions by a homoclinic bifurcation curve which we do not show here. The dashed curve (green) shows a slow subsystem trajectory that approaches the BT point.}
\end{figure}

\begin{figure}[htbp]
\centering
\psfrag{y1}{$x_1$}
\psfrag{y2}{$x_2$}
\psfrag{R1}{$Q_1$}
\psfrag{R2}{$Q_2$}
\psfrag{R3}{$Q_3$}
 \includegraphics[width=0.98\textwidth]{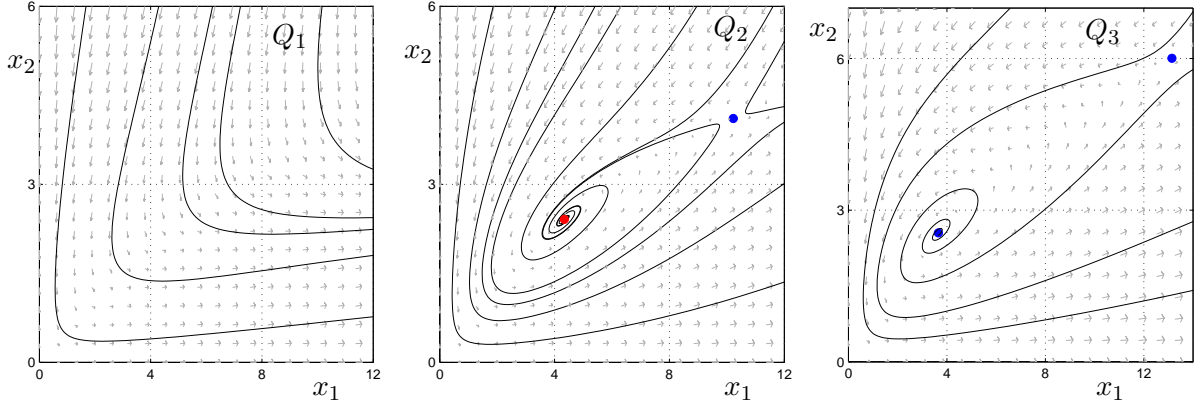}
\caption{\label{fig:fig11}Phase space diagrams for different parameter regions in Figure \ref{fig:fig10}; black curves are trajectories. $Q_1$: $(y_1,y_2)=(0.45,0.35)$; $Q_2$: $(y_1,y_2)=(0.35,0.3)$; $Q_3$: $(y_1,y_2)=(0.45,0.15)$. In $Q_1$ there is a stable spiral sink outside of the chosen range at $(x_1,x_2)\approx(92.12,3.34)$. A spiral sink equilibrium point also exists in $Q_2$ and $Q_3$ outside of the displayed ranges. In $Q_2$ we have a spiral sink (red dot) and a saddle point (blue dot) that correspond to attracting and saddle-type branches of the critical manifold. In $Q_3$ we have a spiral source and a saddle point corresponding to unstable and saddle-type critical manifolds.}
\end{figure}

We set $y_1:=\alpha$ and $y_2:=\delta$ to indicate that these parameters will be viewed as slow variables. Part of the bifurcation diagram for \eqref{eq:Kuznetsov_Baz} is shown in Figure \ref{fig:fig10}. Figure \ref{fig:fig10} shows two curves of fold bifurcations, which actually form a closed curve $c_{LP}$ (``isola'') in parameter space. This curve has a tangency with a supercritical Hopf bifurcation curve $c_H$ at a codimension-two Bogdanov-Takens (BT) point. We do not show the homoclinic bifurcation curve originating at the BT point in Figure \ref{fig:fig10}. The curves $c_{LP}$ and $c_{H}$ can be calculated explicitly \cite{Kuznetsov}. One simply uses the linearization of $D_xF(x^*)$ of \eqref{eq:Kuznetsov_Baz} at an equilibrium point $(x_1,x_2)=x^*$ and applies the conditions $\det(D_xF(x^*))=0$ and $\text{Tr}(D_xF(x^*))=0$; this gives
\benn
\begin{array}{llll}
c_{LP}&=&\{y\in\R^2:&4\xi(y_1-1)^3+((y_1^2-20y_1-8)\xi^2+2y_1\xi(y_1^2-11y_1+10)\\
&&&+y_1^2(y_1-1)^2)y_2-4(y_1+\xi)^3y_2^2=0\},\\
c_{H}&=&\{y\in\R^2:&4\xi(y_1(y_1-1)+\xi(y_1+1))+(2(\xi+1)y_1^2+(3\xi^2-2\xi-1)y_1\\
&&&+\xi(\xi^2-2\xi+5))y_2+(y_1+\xi-1)^2y_2^2\}.\\
\end{array}
\eenn
The Bogdanov-Takens point satisfies all genericity conditions required by assumption (A2) so that Lemma \ref{lem:BT} and Theorem \ref{thm:CK2} apply. The normal form coefficient is $s=-1$ in equation \eqref{eq:BT_nform}. Hence the only stable equilibrium point near the BT-point can be found between the Hopf and fold curves in region $Q_2$ in Figure \ref{fig:fig10}. Phase portraits for different regions are shown in Figure \ref{fig:fig11}. A stable limit cycle can occur between the Hopf and homoclinic bifurcation curves in region $Q_3$ but we ignore this possibility and restrict ourselves to critical transitions via fast subsystem stable equilibrium points. It is natural to assume that the populations $(x_1,x_2)$ are subject to stochastic fluctuations and to view $(y_1,y_2)$ as slow dynamic variables, changing slowly due to evolutionary or environmental effects. This converts \eqref{eq:Kuznetsov_Baz} into the SDE
\be
\label{eq:Bazykin}
\begin{array}{lcl}
dx_1&=& \frac1\epsilon\left[x_1-\frac{x_1x_2}{1+\alpha x_1}-\xi x_1^2\right] ds+\frac{\sigma_1}{\sqrt\epsilon}dW^{(1)},\\
dx_2&=&\frac1\epsilon \left[-\gamma x_2+\frac{x_1x_2}{1+\alpha x_1}-\delta x_2^2\right]ds +\frac{\sigma_2}{\sqrt\epsilon}dW^{(2)},\\
dy_1&=&g_1(x,y)ds,\\
dy_2&=&g_2(x,y)ds,\\
\end{array}
\ee
where we have assumed uncorrelated noise in the fast variables. The critical manifold $C_0$ of the deterministic version of \eqref{eq:Bazykin} has an attracting branch $C^a_0$ in the region $Q_2$ (see Figures \ref{fig:fig10} and \ref{fig:fig11}) corresponding to a spiral sink of the fast subsystem \eqref{eq:Kuznetsov_Baz}. We want to approach the Bogdanov-Takens critical transition via a slow flow inside the region $Q_2$. Figure \ref{fig:fig11} shows a dashed curve (green) which is a possible slow subsystem trajectory. It is part of a candidate $\gamma_0$ that undergoes a critical transition according to Lemma \ref{lem:BT}. In principle, we could try to embed such a candidate into an explicit slow flow $\dot{y}=g(x,y)=(g_1(x,y),g_2(x,y))^T$.

For numerical simulations of \eqref{eq:Bazykin} it will suffice to define a single trajectory $\gamma_0$ along which we approach the BT transition. We can obtain $\gamma_0$, for example, by polynomial interpolation of a suitable set points lying in $Q_2$ and the BT point. The initial condition for our numerical simulation is chosen as $(x_1,x_2,y_1,y_2)\approx(3.1544,1.8849,0.3,0.3293)$, where the y-coordinates lie on the dashed curve indicated in Figure \ref{fig:fig10} and the x-coordinates are on the attracting critical manifold $C^a_0$. Calculations have been carried out for 50 sample paths and the variance has been calculated via a moving window method for each path (see the gap in Figure \ref{fig:fig12}(a) for the window size) with linear detrending. Then the results is averaged over the 50 paths. Figure \ref{fig:fig12}(a) compares the variances $V_i=\text{Var}(x_i(y))$ for $i\in\{1,2\}$.

\begin{figure}[htbp]
\centering
\psfrag{Vi}{$V_i$}
\psfrag{y1}{$y_1$}
\psfrag{V1}{$V_1$}
\psfrag{V2}{$V_2$}
\psfrag{Vari}{$V_i=\text{Var}(x_i)$ $\downarrow$}
\psfrag{Fits}{Fits of $V_i$ $\rightarrow$}
\psfrag{a}{(a)}
\psfrag{b}{(b)}
\psfrag{c}{(c)}
 \includegraphics[width=1\textwidth]{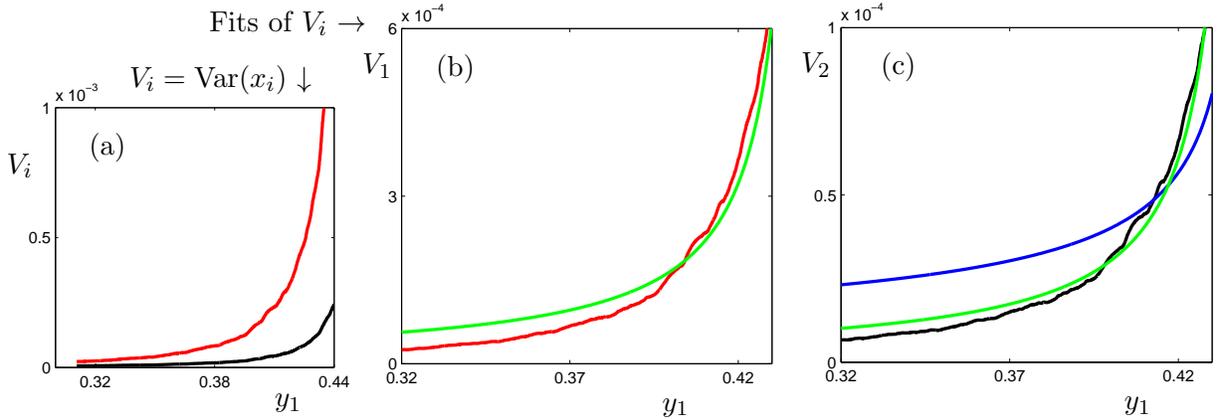}
\caption{\label{fig:fig12}Simulations averaged over 50 sample paths of the Bazykin predator-prey model \eqref{eq:Bazykin} with $\gamma=1$, $\xi=0.01$ and $(\epsilon,\sigma)=(3\times 10^{-5},1\times 10^{-3})$. (a) Variance curves $V_i=\text{Var}(x_i(y))$ for $i\in{1,2}$; the red curve corresponds to $V_1$ and the black curve to $V_2$. In (b) and (c) we repeat these curves and show different fits. The green curves correspond to \eqref{eq:fit_BT1} and the blue curve to \eqref{eq:fit_BT2}.}
\end{figure}

Figures \ref{fig:fig12}(b)-(c) show fits (green curves) of the variances
\be
\label{eq:fit_BT1}
\text{Var}(x_{i}(y))=\frac{A}{y_{1,c}-y_1}, \qquad \text{for $i\in\{1,2\}$}
\ee
and also an inverse square-root fit (blue curve) 
\be
\label{eq:fit_BT2}
\text{Var}(x_{2}(y))=\frac{A}{\sqrt{y_{1,c}-y_1}}
\ee
where $A$, $y_{1,c}$ are the fitting parameters. Note that \emph{both} variances increase like $\cO_y^*(1/(y_{1,c}-y_1))$ near the critical transition and that \eqref{eq:fit_BT1} is a good fit for $V_2$ while \eqref{eq:fit_BT2} is not. At first, this might look unexpected since the normal form analysis predicts one variance to increase like $\cO_y^*(1/\sqrt{y_{1,c}-y_1})$. However, equation \eqref{eq:Bazykin} is not in normal form. To explain the effect let us consider the Bogdanov-Takens normal form  
\be
\label{eq:BT_nform_new}
\begin{array}{lcl}
d\tilde{x}_1&=&\frac1\epsilon[\tilde{x}_2]ds+\frac{\sigma}{\sqrt\epsilon}F_1dW^{(1)},\\
d\tilde{x}_2&=&\frac1\epsilon[y_1+y_2\tilde{x}_2+\tilde{x}_1^2+s\tilde{x}_1\tilde{x}_2]ds+\frac{\sigma}{\sqrt\epsilon}F_2dW^{(2)},\\
\end{array}
\ee
with suitable slow variables $y=(y_1,y_2)$ so that we approach the critical BT-transition at $(\tilde{x},y)=(0,0)$. Consider a linear map
\benn
\left(\begin{array}{c} x_1 \\ x_2 \\\end{array}\right)= \left(\begin{array}{cc} b_{11} & b_{12} \\ b_{21} & b_{22}\\ \end{array}\right) \left(\begin{array}{c} \tilde{x}_1 \\ \tilde{x}_2 \\\end{array}\right)=B\left(\begin{array}{c} \tilde{x}_1 \\ \tilde{x}_2 \\\end{array}\right)
\eenn
where $B\in\R^{2\times 2}$ is invertible. We know that 
\benn
\text{Var}(\tilde{x}_1(y))=\cO_y^*\left( \frac{1}{y_1}\right),\qquad \text{Var}(\tilde{x}_2(y))=\cO_y^*\left( \frac{1}{\sqrt{y_1}}\right), \qquad \text{Cov}(\tilde{x}_1(y),\tilde{x}_2(y))=\cO_y^*(1)
\eenn
as $y_1\ra 0$. After applying the transformation $B$ a formal calculation yields
\beann
\text{Var}(x_1)&=&\text{Var}(b_{11}\tilde{x}_1+b_{12}\tilde{x}_2)=b_{11}^2\text{Var}(\tilde{x}_1)+b_{12}^2\text{Var}(\tilde{x}_1)+2b_{11}b_{12}\text{Cov}(\tilde{x}_1,\tilde{x}_2)\\
&=& \cO_y^*\left(\frac{1}{y_1}\right)+\cO_y^*\left(\frac{1}{\sqrt{y_1}}\right)+\cO_y^*(1)=\cO_y^*\left(\frac{1}{y_1}\right),\\
\text{Var}(x_2)&=&\text{Var}(b_{21}\tilde{x}_1+b_{22}\tilde{x}_2)=b_{21}^2\text{Var}(\tilde{x}_1)+b_{22}^2\text{Var}(\tilde{x}_1)+2b_{21}b_{22}\text{Cov}(\tilde{x}_1,\tilde{x}_2)\\
&=& \cO_y^*\left(\frac{1}{y_1}\right)+\cO_y^*\left(\frac{1}{\sqrt{y_1}}\right)+\cO_y^*(1)=\cO_y^*\left(\frac{1}{y_1}\right).
\eeann
This an explanation why both variances $V_i$ increase like $\cO_y^*(1/(y_{1,c}-y_1))$ in Figure \ref{fig:fig12}. The scaling law from the Hopf bifurcation dominates the scaling law from the saddle-node bifurcation near a codimension-two Bogdanov-Takens point when the system is not in normal form.

We conclude this section with some potential implications for ecological modeling and ecosystem management. Once we have passed the BT-point the system transitions with high probability to a far-away equilibrium (see Figure \ref{fig:fig11}). In particular, the density of the prey population increases dramatically. In this scenario it will be very difficult to reverse the system to the original state as the region $Q_2$ of slow variable/parameters is very narrow near the BT-point. The most interesting aspect of the BT-transition in the Bazykin model \eqref{eq:Bazykin} is that just measuring the variances, without a preliminary normal form transformation, can be misleading. Measurement and fitting indicate a variance increase governed by $\cO_y^*(1/y)$ which could just indicate a supercritical Hopf transition from region $Q_2$ to $Q_3$ i.e. passing the (red) Hopf curve in Figure \ref{fig:fig10}. This transition would not be critical and can easily be reversed. The slower variance increase of the critical fold transition is \emph{hidden} near the BT-point!

\subsection{Biomechanics and Control near Instability}
\label{ssec:mechanics}

\begin{figure}[htbp]
\centering
\psfrag{A}{A}
\psfrag{B}{B}
\psfrag{LP}{LP}
\psfrag{BP}{BP}
\psfrag{th}{$\theta$}
\psfrag{Fs}{$\cF_s$}
 \includegraphics[width=1\textwidth]{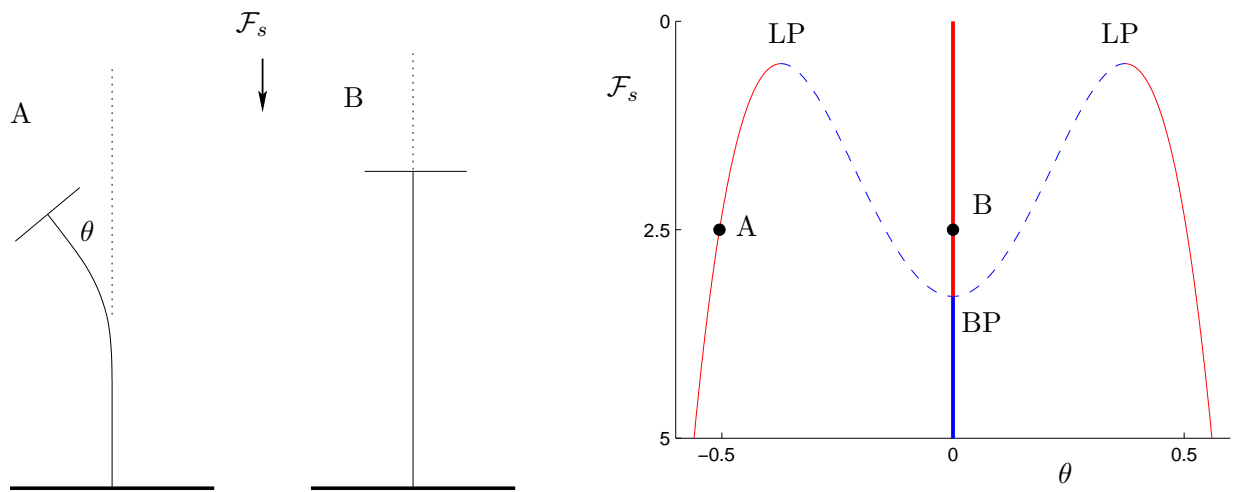}
\caption{\label{fig:fig13}Panels A and B show a sketch of the Euler buckling experiment as considered in \cite{VenkadesanGuckenheimerValero-Cuevas}. The force $\cF_s$ compresses the spring which should stay in the upright/vertical position as shown in $B$. The bifurcation diagram on the right shows the subcritical pitchfork \eqref{eq:pitch1} with parameter values \eqref{eq:pvals_pitch}. The pitchfork (branch point [BP]) from the attracting equilibrium branch (think red line) occurs at $F_s=3.3$. The unstable branches (dashed blue) undergo a further fold bifurcation (limit point [LP]). In $A$ we see what happens when the spring buckles and leaves the vertical position.}
\end{figure}

\begin{figure}[htbp]
\centering
\psfrag{a}{(a)}
\psfrag{b}{(b)}
\psfrag{c}{(c)}
\psfrag{d}{(d)}
\psfrag{x}{$x$}
\psfrag{y}{$y$}
\psfrag{V}{$V$}
 \includegraphics[width=1\textwidth]{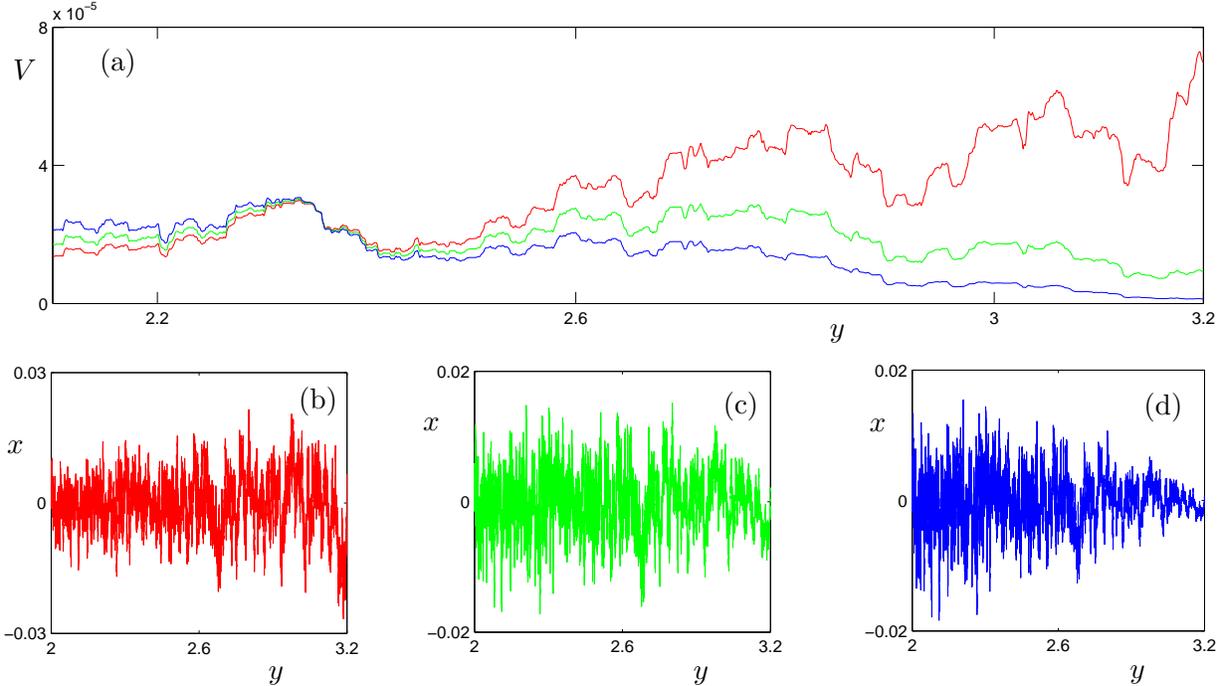}
\caption{\label{fig:fig14}(b)-(d) Sample paths for \eqref{eq:Madu} with $F(y)=1$ (red), $F(y)=\sqrt{y_c-y}$ (green), $F(y)=y_c-y$ (blue) and $g(x,y)=1$; fixed parameter values are given in \eqref{eq:pvals_pitch} and $(\epsilon,\sigma)=(0.005,0.01)$. The initial condition is $(x_0,y_0)=(0,2)$. The realization of the noise $W=W_s$ is the same for all three paths. In (a) we calculate the variance $V=\text{Var}(x(y))$ for each path using a sliding window technique. Note that we can already spot in the time series that variance increases for $F(y)=1$, stays roughly constant for $F(y)=\sqrt{y_c-y}$ and decays to zero for $F(y)=y_c-y$ as $y$ tends towards the pitchfork critical transition at $y_c=3.3$.}
\end{figure}

\begin{figure}[htbp]
\centering
\psfrag{F1}{$F=1$}
\psfrag{F2}{$F=\sqrt{y_c-y}$}
\psfrag{F3}{$F=y_c-y$}
\psfrag{y}{$y$}
\psfrag{V}{$V$}
 \includegraphics[width=1\textwidth]{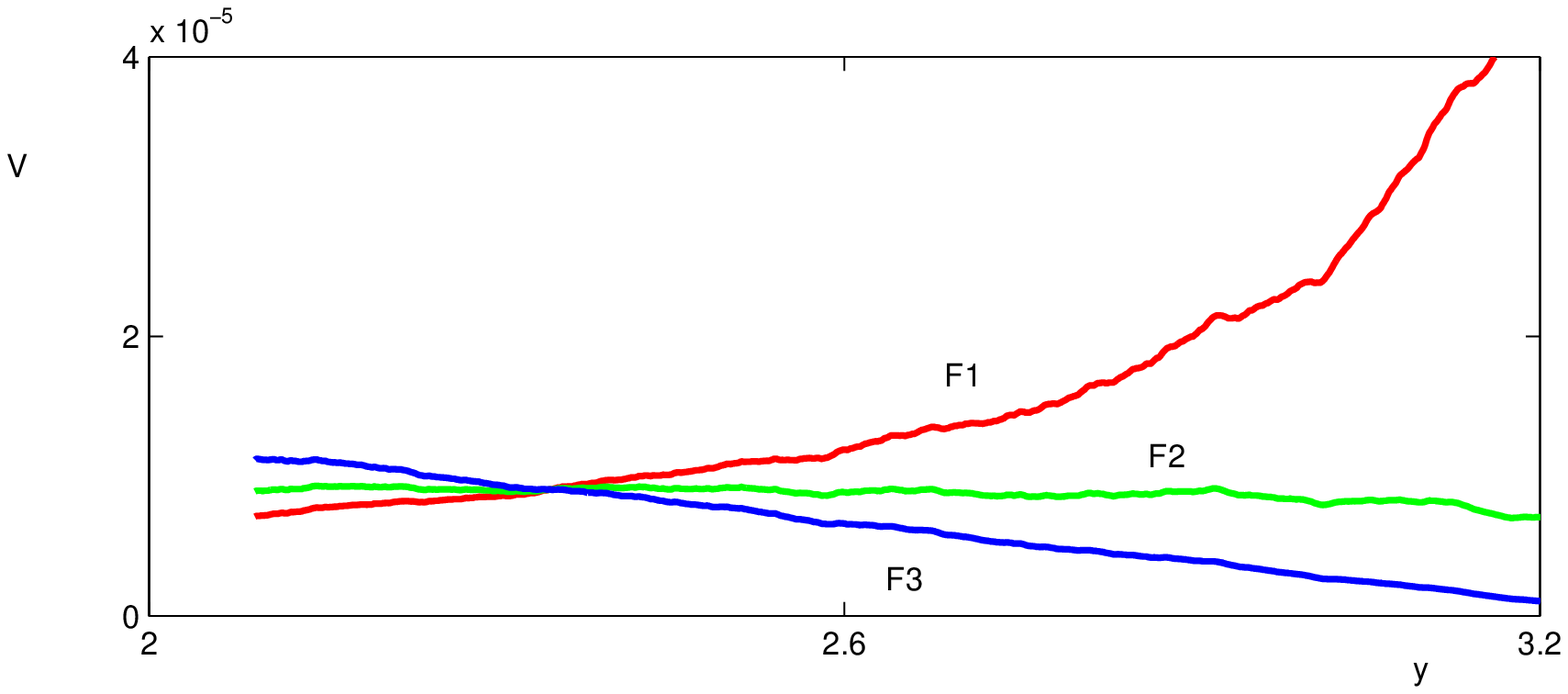}
\caption{\label{fig:fig15}Average variance $V=\text{Var}(x(y))$ for \eqref{eq:Madu} for $F(y)=1$ (red), $F(y)=\sqrt{y_c-y}$ (green), $F(y)=y_c-y$ (blue) and $g(x,y)=1$ over 100 sample paths; fixed parameter values are given in \eqref{eq:pvals_pitch} and $(\epsilon,\sigma)=(0.005,0.007)$. The initial condition is $(x_0,y_0)=(0,2)$. The point where the three variances cross corresponds approximately to $y^*=2.3$. This is expected since the three functions in \eqref{eq:pitch_noise} are equal at $y=y^*$.}
\end{figure}

In \cite{VenkadesanGuckenheimerValero-Cuevas} the authors investigate how humans control a spring near instability. The experimental setup asks participants to use their thumbs to compress the spring near the threshold of the classical \texttt{Euler buckling instability}; see Figure \ref{fig:fig13}. A mathematical model for this problem is provided by a subcritical pitchfork bifurcation with quintic non-linearity given by
\be
\label{eq:pitch1}
\theta'=p_1(\cF_s-p_2) \theta+p_3\theta^3-p_4\theta^5
\ee
where $\cF_s$, $p_j$ for $j\in\{1,2,3,4\}$ are parameters and $\theta$ represents the angle of the spring with respect to its vertical/upright position \cite{VenkadesanGuckenheimerValero-Cuevas}. The parameter $\cF_s$ is viewed as the force applied to the spring. The bifurcation diagram of \eqref{eq:pitch1} is shown in Figure \ref{fig:fig13}. To stay within the framework of \cite{VenkadesanGuckenheimerValero-Cuevas} we have chosen fixed parameter values
\be
\label{eq:pvals_pitch}
p_1=2.639, \qquad p_2=3.3, \qquad p_3=106.512, \qquad p_4=385. 
\ee
The experiment in \cite{VenkadesanGuckenheimerValero-Cuevas} asked participants to slowly compress the spring so that it does not buckle but also comes as close as possible to the pitchfork bifurcation. In Figure \ref{fig:fig13} this corresponds to moving along the stable equilibrium branch $\{(\theta,F_s)\in\R^2:F_s<3.3\}$. The experimental data do contain quite a bit of noise so that it is very reasonable to consider the system
\be
\label{eq:Madu}
\begin{array}{lcl}
dx&=&\frac1\epsilon\left[p_1(y-p_2) x+p_3x^3-p_4x^5\right]ds+\frac{\sigma}{\sqrt\epsilon}F(y)dW,\\
dy&=& 1~ds.
\end{array}
\ee
The deterministic critical manifold of \eqref{eq:Madu} is $C_0=\{(x,y)\in\R^2:p_1(y-p_2) x+p_3x^3-p_4x^5=0\}$. We focus on the trivial branch $C_0^*=\{x=0\}$ and the attracting subset $C_0^a:=C_0^*\cap \{y<3.3\}$. In the previous applications we usually assumed that $F(y)=const.$ which corresponds to additive noise. For the spring compression experiment this does not seem reasonable since participants could try to minimize the noisy fluctuations once they are very close the subcritical pitchfork bifurcation; in fact, they know that a noise-induced critical transition could occur before the bifurcation point. Figure \ref{fig:fig14}(b)-(d) shows sample paths for different types of noise
\be
\label{eq:pitch_noise}
F(y)=1,\qquad F(y)=\sqrt{y_c-y},\qquad F(y)=y_c-y.
\ee
We used the same realization for $dW$ for all three paths. It can already be observed that we have three different behaviors (``increase, constant, decay'') for the variance $V=\text{Var}(x(y))$. Figure \ref{fig:fig15} confirms this behavior as it shows the average variance over 100 sample paths for the different types of noise given in \eqref{eq:pitch_noise}. We can calculate from Theorem \ref{thm:CK1} that to leading order in the approach towards the pitchfork, but not in a small neighbourhood near it, we have the scaling laws
\benn
\text{Var}(x(y))=\cO_y^*\left(\frac{N(y)}{y_c-y}\right)=\cO_y^*\left(\frac{F^2(y)}{y_c-y}\right)=\left\{\begin{array}{lcl}
\cO_y^*\left(\frac{1}{y-y_c}\right) & & \text{if $F(y)=1$,}\\
\cO_y^*\left(1\right) & & \text{if $F(y)=\sqrt{y_c-y}$,}\\ 
\cO_y^*\left(y-y_c\right) & & \text{if $F(y)=y_c-y$.}\\  
\end{array}\right.
\eenn 
This explains precisely what is shown in Figure \ref{fig:fig15} and shows that multiplicative noise can yield a wide variety of different early-warning signals or even no visible trend of the variance near a critical transition. Hence we can conjecture that balancing/controlling objects near an instability involves suitable noisy perturbations and the quick processing of a time series history to generate the appropriate control.

\section{Discussion and Outlook}
\label{sec:conclusions}

This paper has only started to develop a mathematical framework for critical transitions and prediction. Here we briefly outline the main steps and how this framework can be extended to address future problems. 

The first part of this paper, motivated by Definition \ref{defn:ct}, only covers the singular limit $\epsilon=0$, $\sigma=0$. We derive slow flow conditions to reach a critical transition and record the relevant linearizations to develop stochastic scaling laws. Although this is the most precise starting point one could consider extensions. In fact, the sample paths viewpoint of Definition \ref{defn:ct} naturally extends. Let 
\benn
\gamma_{\epsilon,\sigma}=\gamma_{\epsilon,\sigma}(t):[0,T]\ra \R^{m+n},\qquad \gamma_{\epsilon,\sigma}(0)=\gamma(0)=(x(0),y(0))
\eenn
be a sample path of \eqref{eq:gen_SDE}. The first deterministic extension is to consider $\gamma_{0,0}$ but remove the requirement from Definition \ref{defn:ct} that the transition point $p$ is normally hyperbolic and to change (C1) so that a candidate $\gamma_{0,0}(t_{j-1},t_j)$ can lie in any part of the critical manifold. This allows for canard orbits and delay as shown in Figure \ref{fig:fig4}(b). As an example consider the pitchfork bifurcation \eqref{eq:nf_pitchfork} with $y(0)<0$ then the point $(x,y)=(0,\min(-y(0),y_b))$ becomes a critical transition where $y_b>0$ is the buffer point \cite{Neishtadt1,Neishtadt2}. For delays and canards the problem of critical transitions becomes \emph{global} in at least \emph{two} ways:

\begin{itemize}
 \item[(G1)] The initial condition matters to determine which points are critical transitions.
 \item[(G2)] The global distance between critical manifolds becomes relevant. 
\end{itemize}

To understand (G2) consider the supercritical pitchfork bifurcation \eqref{eq:nf_pitchfork}. Depending on the initial condition there may be a jump at $p=(x_p,y_p)$ for $0<y_p\ll1$ or $0<y_p=1$. The length of the fast segment to the next attracting critical manifold $y=x^2$ from $p$ is $\sqrt{y_p}$; see Figure \ref{fig:fig4}. Applications clearly \emph{require} a case distinction between $\sqrt{y_p}\ll1$ which is usually not viewed as critical and $\sqrt{y_p}=1$ which should probably be called critical. Hence an extension to canards must \emph{append} a global distance measure to the sample path space, {e.g.}~the minimum or maximum distance from the transition point to the fast subsystem attractor; see Figure \ref{fig:fig4}(b) where canards with or without head usually yield two different distances. Since this paper entirely restricts to a \emph{local} theory we do not discuss this aspect further.

\begin{figure}[htbp]
\centering
\psfrag{x}{$x$}
\psfrag{xl}{\scriptsize{$x$}}
\psfrag{y}{$y$}
\psfrag{g1}{$\gamma_{\epsilon,0}$}
\psfrag{g2}{$\gamma_{\epsilon,\sigma}$}
\psfrag{g3}{$\gamma_{0,0}$}
\psfrag{g4}{$\gamma_{0,\sigma}$}
\psfrag{a}{(a)}
\psfrag{b}{(b)}
\psfrag{C0}{$C_0$}
 \includegraphics[width=0.9\textwidth]{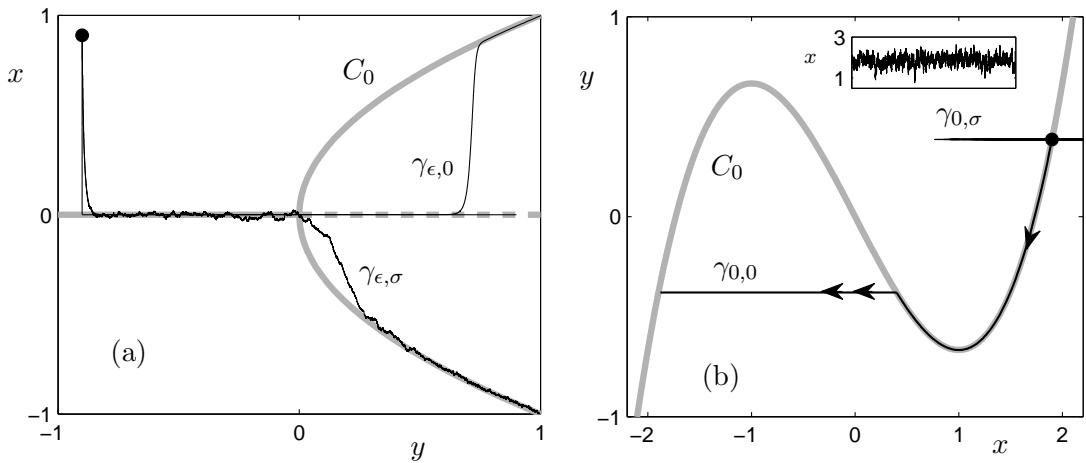}
\caption{\label{fig:fig4}Illustration of possible extensions to Definition \ref{defn:ct}. (a) Phase space for \eqref{eq:gen_SDE} with $f(x,y)=yx-x^3$, $F(x,y)=1$ and $g(x,y)=1$. The critical manifold $C_0$ (grey) and two sample paths $\gamma_{\epsilon,0}$ and $\gamma_{\epsilon,\sigma}$ (black) for $\epsilon=0.01=\sigma$ with initial condition $(x(0),y(0))=(0.9,-0.9)$ are shown. (b) Phase space for \eqref{eq:gen_SDE} with $f(x,y)=y-\frac{x^3}{3}-x$, $g(x,y)=1-x$ and $F(x,y)=1$ with a non-generic fold at $(x,y)=(1,-2/3)$. Again we show two sample paths (black) and the critical manifold (grey). Note that the path $\gamma_{0,\sigma}$ for $\sigma=0.5$ cannot drift in $y$ but will switch, on exponentially long time scales, between the two attracting branches of the critical manifold. The inset shows a time series for this path on a subexponential time scale.}
\end{figure}

Another possible extension is to consider sample paths $\gamma_{\epsilon,0}$ for $0<\epsilon\ll1$; see Figure \ref{fig:fig4}(a). In this case, the extension can just be defined by requiring that $d_H(\gamma_{\epsilon,0},\gamma_{0,0})\ra 0$ as $\epsilon\ra0$ {i.e.}~by checking whether candidates that have a critical transition in the singular limit perturb. The perturbation results are known for the fold, pitchfork, transcritical and Hopf bifurcations \cite{KruSzm1,KruSzm3,KruSzm4,Neishtadt1}. Partial results are  available for the Bogdanov-Takens bifurcation \cite{Chiba1} and the cusp \cite{BroerKaperKrupa} is work in progress; the remaining codimension-two problems are expected to be solvable with similar ideas. One could also add generic cases for higher-dimensional and non-minimal slow variables such as folded singularities in $\R^3$ \cite{SzmolyanWechselberger1}.

The case $\gamma_{\epsilon,0}$ is primarily of mathematical interest since for $\gamma_{\epsilon,\sigma}$ and $\sigma>0$ a delay/canard effect is shortened substanially by noise in applications (see {e.g.}~Theorem 2.11 of \cite{BerglundGentz6}) as long as the noise is not exponentially small \cite{Sowers}. A typical delay time is of order $\sqrt{\epsilon|\ln\sigma|}$ so the local singular limit results are relevant; see also Figure \ref{fig:fig4}(a). However, there is a major open issue for applications we do not address here corresponding to transitions driven purely by noise or a combination of noise and bifurcations. 

The case $\gamma_{0,\sigma}$ for a sample path starting near an attracting critical manifold $C^{a1}_0$ is covered by the theory of large deviations \cite{FreidlinWentzell} and purely noise-induced transitions can occur to an attracting critical manifold $C^{a2}_0$; see Figure \ref{fig:fig4}(b) where the upper path will eventually escape. Again, the sample path viewpoint is well-suited as we ask for an estimate of probabilites {e.g.}
\be
\label{eq:prob_est}
\P\left(\left[\inf_{[0,T]}t:d_h(\gamma_{0,\sigma}(t),C^{a2}_0)<\delta_2,d_H(\gamma_{0,\sigma}(0),C^{a1}_0)<\delta_1\right]>t^*\right)
\ee
for suitable small constants $\delta_{1,2}$ and a given time $t^*>0$. Hence one can again use paths and Definition \ref{defn:ct} as a basis but then has to add for each point on $C^{a1}_0$ a probabilistic description how likely the escape is which usually yields exponentially long time scales to escape. This is again a \emph{global} problem. For cases with one fast variable and $\epsilon=0$ it is often possible to obtain explicit solutions using Fokker-Planck equations {e.g.}~see \cite{Gardiner,ArnoldSDE,LindnerSchimansky-Geier,ThompsonSieber2,KuehnCT1}.\\

\textit{Remark:} After the suggestion of the Definition \ref{defn:ct} in \cite{KuehnCT1}, recent work of Ashwin et {al.}~\cite{AshwinWieczorekVitoloCox} suggested a related applied classification of critical transitions distinguishing between B-tipping ('bifurcation-induced'), N-tipping ('noise-induced') and R-tipping ('rate-induced'). Basically B-tipping aims to cover paths $\gamma_{\epsilon,0}$ for $\epsilon\ra 0$ and N-tipping considers paths $\gamma_{0,\sigma}$; it is currently work in progress to understand R-tipping better.\\  

The most general case is to consider $\gamma_{\epsilon,\sigma}$ for $\sigma,\epsilon>0$ where noise-induced escape \emph{shortly before} a fast subsystem bifurcation point on \emph{non-exponential time scales} becomes relevant. One of the key goals of the mathematical framework presented in this paper was to also allow for a natural extension of the methods and definitions to this case. It is future work to combine the ideas from Definition \ref{defn:ct} by adding to it pathwise probability estimates of the form \eqref{eq:prob_est}. This should yield the full mathematical framework based upon sample paths with all parameters: $\sigma>0$, $\epsilon>0$, distance to the next attractor and escape probability during $[0,T]$. 

For the local theory of codimension-one bifurcations several studies on various regimes with $\sigma,\epsilon>0$ near bifurcation points exist. Overall, the fold \cite{BerglundGentz8,Sowers}, pitchfork/transcritical \cite{Kuske,BerglundGentz6} and Hopf bifurcations \cite{BerglundGentz2,SuRubinTerman,BerglundGentz3} are quite well understood. One basic insight is that scaling regimes are identified under which noise-induced effects or deteterministic drift dominate. Another important conclusion are probabilistic estimates for certain distinct dynamical regimes to occur. For higher codimension phenomena not many results are known but see {e.g.}~\cite{BerglundGentzKuehn,BerglundLandon}. As far as the stochastic scaling laws for codimension-two cases considered in this paper are concerned there does not seem to be any work prior to this paper in this direction.\\

The second contribution of this paper is to understand fluctuations and scaling laws of paths better before fast subsystem bifurcations to determine early-warning signs. In particular, leading-order scaling behaviour for covariance matrices have been derived. We have only covered the basic case of local bifurcations up to codimension-two with white noise in the region (R1) with a suitable scaling of noise and time scale separation which makes early escapes unlikely. Large fluctuations before the bifurcation and scaling results near bifurcations are certainly not well-studied for all bifurcations up to codimension two. Early-warning signs for other types of noise (colored noise, shot/burst noise \cite{Gardiner}), for degenerate noise terms \cite{TouboulWainrib} and for more general stochastic processes ({e.g.}~L{\'{e}}vy Processes \cite{Kallenberg,ImkellerPavlyukevich}) are interesting directions. As before, sample paths and singular limits are still available, even for very general high-dimension bifurcations and stochastic processes.

Global bifurcations \cite{Kuznetsov} have not been considered and would be an interesting direction for future analysis. There is work in progress to understand these bifurcations and their warning signs in models as well as in a normal form setup. Another possible extension are early-warning signs for spatially-extended problems; see \cite{Dakosetal1,DonangeloFortDakosSchefferNes,Dakosetal2} for models from ecology. In this context, it is well-known that many classes of pattern-forming partial differential equations (PDEs) and stochastic partial differential equations (SPDEs) can be written as evolution equations with well-defined paths or stochastic sample paths \cite{Henry,DaPratoZabczyk}. Several relevant PDEs, such as excitable systems \cite{Nagumo,Barkley2} with diffusion, are often already in a natural fast-slow form. Presumably one should find many other interesting early-warning signs for spatial systems but these could also be more difficult to measure and apply in practical applications since the collection and analysis of much larger data sets arises; a typical area where this already proved to be very difficult are epileptic seizures \cite{MormannAndzejakElgerLehnertz,MeiselKuehn}.\\

The third contribution of the current work are examples, several of them in application domains (epidemics, systems biology and biomechanics) where the new techniques for early-warning signs have not been considered. Furthermore, the examples provide illustrations of the theory and also show its limitations where prediction becomes impossible or misleading if one relies on the scaling of the variance. There are many important directions for making the theory more applicable {e.g.}~detailed statistical tests such as receiver-operator curves \cite{HallerbergKantz,KuehnZschalerGross,BoettingerHastings}, analysis of limited data and its interpretation \cite{DitlevsenJohnsen}, linking critical transitions to experiments \cite{DrakeGriffen,Veraartetal}, desirable tipping points in applications and their control \cite{Jensen,KuehnNetworks} as well as networks and deterministic metastability \cite{KuehnZschalerGross}.\\       
     
\textbf{Acknowledgments:} I would like to thank Martin Zumsande for suggesting the model from systems biology in Section \ref{ssec:SysBio} and Thilo Gross for insightful discussions about network dynamics. I also would like to thank two anonymous referees and the editor for many helpful comments that helped to improve the manuscript. Part of this work was supported by the European Commission (EC/REA) via a Marie-Curie International Re-integration Grant.

\small{
}

\end{document}